\documentclass{amsart}
\usepackage{amsmath,amssymb,mathrsfs}
\usepackage{amsthm,bbm,mdframed,url,MnSymbol,hyperref}

\usepackage{graphicx}
\usepackage{tikz}
\usetikzlibrary{shapes,arrows,backgrounds,fit,positioning}

\date{}
\allowdisplaybreaks
\def\tb{\qopname\relax o{tb}}
\def\lk{\qopname\relax o{lk}}
\def\tr{\qopname\relax o{tr}}

\newtheorem{theo}{Theorem}[section]
\newtheorem{lemm}{Lemma}[section]
\newtheorem{prop}{Proposition}[section]

\theoremstyle{remark}
\newtheorem{rema}{Remark}[section]

\theoremstyle{definition}
\newtheorem{defi}{Definition}[section]
\newtheorem*{conv}{Convention}

\graphicspath{{./pictures/}}	

\author {Ivan Dynnikov and Vladimir Shastin}
\address{\noindent V.A.Steklov Mathematical Institute of Russian Academy of Science, 8 Gubkina Str., Moscow 119991, Russia}
\address{\noindent St.\ Petersburg State University, Line 14th (Vasilyevsky Island), 29, Saint Petersburg, 199178, Russia}
\email{dynnikov@mech.math.msu.su}
\address{Department of Mechanics and Mathematics of Moscow State University, 1 Leninskije gory, Moscow 119991, Russia}
\address{\noindent St.\ Petersburg State University, Line 14th (Vasilyevsky Island), 29, Saint Petersburg, 199178, Russia}
\email{vashast@gmail.com}

\title{Distinguishing Legendrian knots with trivial
orientation-preserving symmetry group}

\begin{document}

\begin{abstract}
In a recent work of I.\,Dynnikov and M.\,Prasolov a new method of comparing Legendrian knots
is proposed. In general, to apply the method requires a lot of technical work. In particular, one needs
to search all rectangular diagrams of surfaces realizing certain dividing configurations. In this paper,
it is shown that, in the case when the orientation-preserving symmetry group of the knot is trivial,
this exhaustive search is not needed, which simplifies the procedure considerably.
This allows one to distinguish Legendrian knots in certain cases when the computation of
the known algebraic invariants is infeasible or not informative. In particular,
it is disproved here that
when~$A\subset\mathbb R^3$ is an annulus tangent to the standard contact structure
along~$\partial A$, then the two components
of~$\partial A$ are always equivalent Legendrian knots. A candidate counterexample
was proposed recently by I.\,Dynnikov and M.\,Prasolov, but the proof of the fact that the two components
of~$\partial A$ are not Legendrian equivalent was not given. Now this work is accomplished.
It is also shown here that the problem of comparing two Legendrian knots having the same topological
type is algorithmically solvable provided that the orientation-preserving symmetry
group of these knots is trivial.
\end{abstract}
\maketitle

\section*{Introduction}
Deciding whether or knot two Legendrian knots in~$\mathbb S^3$ having the same classical invariants (see definitions below)
are Legendrian isotopic is not an easy task in general. There are two major tools used for classification
of Legendrian knots of a fixed topological type: Legendrian knot invariants having algebraic
nature (see Chekanov~\cite{che2002}, Eliashberg~\cite{El}, Fuchs~\cite{fuchs2003}, Ng~\cite{ng2003,ng2011},
Ozsv\'ath--Szab\'o--Thurston~\cite{OST}, Pushkar'--Chekanov~\cite{PC2005}),
and Giroux's convex surfaces endowed with the characteristic
foliation (see Eliashberg--Fraser~\cite{EF1,EF2}, Etnyre--Honda~\cite{etho2001},
Etnyre--LaFountain--Tosun~\cite{etlafato2012}, Etnyre--Ng--V\'ertesi~\cite{etngve2013}, Etnyre--V\'ertesi~\cite{etver2016}).

The Legendrian knot atlas by W.\,Chongchitmate and L.\,Ng~\cite{chong2013} summarizes the classification results
for Legendrian knots having arc index at most~$9$. As one can see from~\cite{chong2013}
there are still many gaps in the classification even for knots with a small arc index/crossing number.
Namely, there are many pairs of Legendrian knot types which are conjectured to be distinct but
are not distinguished by means of the existing methods.

The works~\cite{representability,distinguishing} by I.\,Dynnikov and M.\,Prasolov propose
a new combinatorial technique for dealing with Giroux's convex surfaces.
This includes a combinatorial presentation of convex surfaces in~$\mathbb S^3$ and a method
that allows, in certain cases, to decide whether or not
a convex surface with a prescribed topological structure of the dividing set
exists.

The method of~\cite{representability,distinguishing} is useful for distinguishing
Legendrian knots, but it requires, in each individual case, a substantial amount
of technical work and a smart choice of a Giroux's convex surface
whose boundary contains one of the knots under examination.

In the present paper we show that there is a way to make this smart choice in
the case when the examined knots have no topological (orientation-preserving) symmetries,
so that the remaining technical work described in~\cite{distinguishing} becomes
unnecessary as the result is known in advance. This makes
the procedure completely algorithmic and allows us, in particular, to distinguish
two specific Legendrian knots for which computation of the known
algebraic invariants is infeasible due to the large complexity of the knot presentations.

The two knots in question are of interest due to the fact that they cobound an annulus embedded in~$\mathbb S^3$
and have zero relative Thurston--Bennequin and rotation invariants. They were proposed
in~\cite{representability} as a candidate counterexample to the claim of~\cite{Gos} that
the two boundary components of such an annulus must be Legendrian isotopic.

The main technical result of this paper was announced by us in~\cite{equiv}
without complete proof.
In~\cite{trleg} the method of this paper is used to show that one can
compare transverse link types in a similar fashion provided
that the orientation-preserving symmetry group of the links is trivial.
In a forthcoming paper by I.\,Dynnikov and M.\,Prasolov it will be
shown how to drop the no symmetry assumption and to produce
algorithms for comparing Legendrian and transverse link types in
the general case.

The paper is organized as follows.
In Section~\ref{legendrian-knot-sec} we recall the definition of a Legendrian knot, and introduce
the basic notation.
In Section~\ref{annuli-sec} we discuss annuli with Legendrian boundary whose
components have zero relative Thurston--Bennequin number.
In Section~\ref{s3-symmetry-group-sec} we define the orientation-preserving symmetry group
of a knot and introduce some $\mathbb S^3$-related notation.
In Section~\ref{rd_of_knots-sec} we recall the definition of a rectangular
diagram of a knot and discuss the relation of rectangular diagrams to Legendrian knots.
Section~\ref{rd_of_surfaces-sec} discusses rectangular diagrams of surfaces.
Here we describe the smart choice of a surface mentioned above (Lemma~\ref{special-diagram-exists}).
In Section~\ref{triviality-sec} we prove the triviality of the orientation-preserving symmetry group
of the concrete knots that are discussed in the paper (modulo the proof of hyperbolicity of
the two complicated knots cobounding an annulus but not Legendrian equivalent, which
is postponed till Section~\ref{append-sec}).
In Section~\ref{app-sec} we prove a number of statements about the non-equivalence of
the considered Legendrian knots.

\subsection*{Acknowledgement}
The work is supported by the Russian Science Foundation under grant~19-11-00151.

\section{Legendrian knots}\label{legendrian-knot-sec}
All general statements about knots in this paper can be extended to many-component links. To simplify
the exposition, we omit the corresponding formulations, which are pretty obvious but sometimes slightly more complicated.

All knots in this paper are assumed to be oriented. The knot obtained from a knot~$K$ by reversing the orientation
is denoted by~$-K$.

\begin{defi}
Let~$\xi$ be a \emph{contact structure} in the three-space~$\mathbb R^3$, that is, a smooth $2$-plane
distribution that locally has the form~$\ker\alpha$, where~$\alpha$ is a differential $1$-form
such that~$\alpha\wedge d\alpha$ does not vanish. A smooth curve~$\gamma$ in~$\mathbb R^3$
is called \emph{$\xi$-Legendrian} if it is tangent to~$\xi$ at every point~$p\in\gamma$.

\emph{A $\xi$-Legendrian knot} is a knot in~$\mathbb R^3$ which is a $\xi$-Legendrian curve.
Two $\xi$-Legendrian knots~$K$ and~$K'$ are said to be \emph{equivalent} if there is a
diffeomorphism~$\varphi:\mathbb R^3\rightarrow\mathbb R^3$ preserving~$\xi$ such that~$\varphi(K)=K'$
(this is equivalent to saying that there is an isotopy from~$K$ to~$K'$ through Legendrian knots).

The contact structure~$\xi_+=\ker(x\,dy+dz)$, where~$x,y,z$ are the coordinates in~$\mathbb R^3$,
will be referred to as \emph{the standard contact structure}. If~$\xi=\xi_+$ we often abbreviate~`$\xi$-Legendrian'
to `Legendrian'.
\end{defi}

In this paper we also deal with the following contact structure, which is a mirror image of~$\xi_+$:
$$\xi_-=\ker(x\,dy-dz).$$

We denote by~$r_-,r_\medvert:\mathbb R^3\rightarrow\mathbb R^3$ the orthogonal reflections in the $xy$- and $xz$-planes,
respectively: $r_-(x,y,z)=(x,y,-z)$, $r_\medvert(x,y,z)=(x,-y,z)$. Clearly, if~$K$ is a $\xi_+$-Legendrian knot, then~$r_-(K)$
and~$r_\medvert(K)$ are $\xi_-$-Legendrian
knots, and vice versa.
It is also clear that the contact structures~$\xi_+$ and~$\xi_-$ are invariant under
the transformation~$r_-\circ r_\medvert:(x,y,z)\mapsto(x,-y,-z)$
(however, if the contact structures are endowed with an orientation, then the latter is flipped).

It is well known that a Legendrian knot in~$\mathbb R^3$ is uniquely recovered from its
\emph{front projection}, which is defined as the projection to the $yz$-plane along
the $x$-axis, provided that this projection is generic (a projection is \emph{generic} if it has only finitely many cusps
and only double self-intersections, which are also required to be disjoint from cusps).
Note that a front projection
always has cusps, since the tangent line to the projection cannot be parallel to the $z$-axis.
Note also that at every double point of the projection the arc having smaller
slope $dz/dy$ is overpassing.

An example of a generic front projection is shown in Figure~\ref{front-fig}.
\begin{figure}[ht]
\includegraphics[scale=0.6]{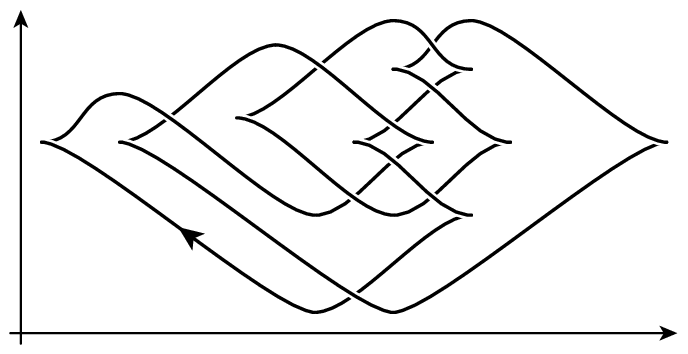}\put(-8,13){$y$}\put(-188,94){$z$}
\caption{Front projection of a Legendrian knot}\label{front-fig}
\end{figure}

There are two well known integer invariants of Legendrian knots called Thurston--Bennequin number and rotation number. We recall their
definitions.

\begin{defi}
\emph{The Thurston--Bennequin number} $\tb(K)$ of a Legendrian knot~$K$ having generic front projection is defined
as
$$\tb(K)=w(K)-c(K)/2,$$
where~$w(K)$ is the writhe of the projection (that is, the algebraic number of double points),
and~$c(K)$ is the total number of cusps of the projection.
\end{defi}

\begin{defi}
A cusp of a front projection is said to be \emph{oriented up} if the outgoing arc appears above the incoming one,
and \emph{oriented down} otherwise (see Figure~\ref{cusps-fig}).
\begin{figure}[ht]
\begin{tabular}{ccc}
\includegraphics[scale=0.5]{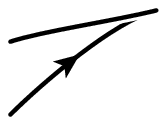}\hskip3mm\includegraphics[scale=0.5]{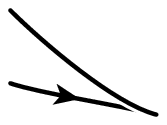}\hskip3mm
\includegraphics[scale=0.5]{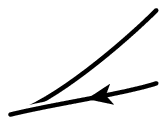}\hskip3mm\includegraphics[scale=0.5]{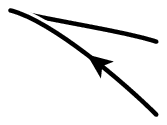}&\hbox to 1cm{\hss}&
\includegraphics[scale=0.5]{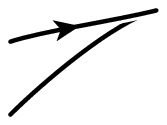}\hskip3mm\includegraphics[scale=0.5]{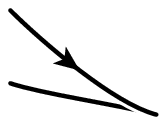}\hskip3mm
\includegraphics[scale=0.5]{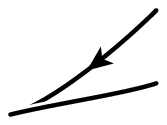}\hskip3mm\includegraphics[scale=0.5]{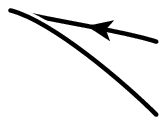}\\
oriented up&&oriented down
\end{tabular}
\caption{Cusps oriented up and down}\label{cusps-fig}
\end{figure}

\emph{The rotation number} $r(K)$ of a Legendrian knot~$K$ having generic front projection is
defined as
$$r(K)=(c_{\mathrm{down}}(K)-c_{\mathrm{up}}(K))/2,$$
where~$c_{\mathrm{down}}(K)$ (respectively, $c_{\mathrm{up}}(K)$)
is the number of cusps of the front projection of~$K$ oriented down (respectively, up).
\end{defi}

For instance, if~$K$ is the Legendrian knot shown in Figure~\ref{front-fig}, then~$\tb(K)=-10$, $r(K)=1$.

The topological meaning of~$\tb$ and~$r$ is as follows. Let~$v$ be a normal vector field to~$\xi$. Then~$\tb(K)$
is the linking number~$\lk(K,K')$, where~$K'$ is obtained from~$K$ by a small shift along~$v$.
The rotation number~$r(K)$ is equal to the degree of the map~$K\rightarrow\mathbb S^1$ defined
in a local parametrization~$(x(t),y(t),z(t))$ of~$K$ by~$(x,y,z)\mapsto(\dot x,\dot y)/\sqrt{\dot x^2+\dot y^2}$.

If~$K$ is a Legendrian knot, then by \emph{the classical invariants of~$K$} one means the topological type of~$K$ together
with~$\tb(K)$ and~$r(K)$.

Sometimes the classical invariants determine the equivalence class of a Legendrian knot completely
(in which case the knot is said to be Legendrian simple).
This occurs, for instance, when~$K$ is an unknot~\cite{EF1,EF2}, a figure eight knot, or a torus knot~\cite{etho2001}.
But many examples of Legendrian non-simple knots are known.

\begin{defi}
Let~$K$, $K'$ be Legendrian knots. We say that~$K'$ is obtained from~$K$ by a \emph{positive stabilization} (respectively, \emph{negative
stabilization}), and~$K$ is obtained from~$K'$ by a \emph{positive destabilization} (respectively, \emph{negative destabilization}),
if there are Legendrian knots~$K''$, $K'''$ equivalent to~$K$ and~$K'$, respectively, such that
the front projection of~$K'''$ is obtained from the front projection of~$K''$ by a local
modification shown in Figure~\ref{stab-fig}(a) (respectively, Figure~\ref{stab-fig}(b)).
\begin{figure}[ht]
\begin{tabular}{ccc}
\includegraphics[scale=0.25]{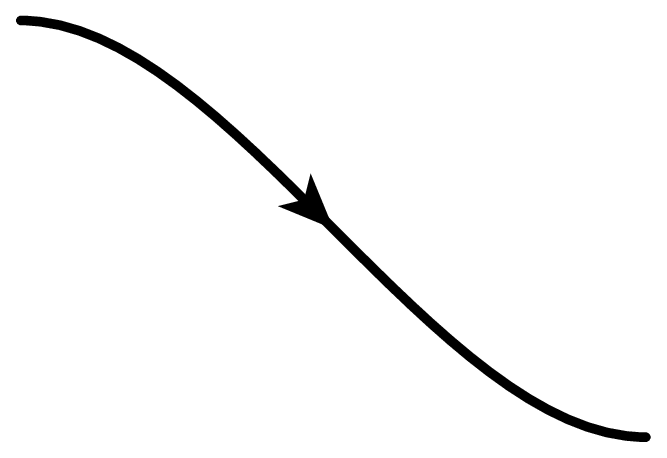}\put(-48,-5){$K''$}\raisebox{26pt}{$\longleftrightarrow$}
\includegraphics[scale=0.25]{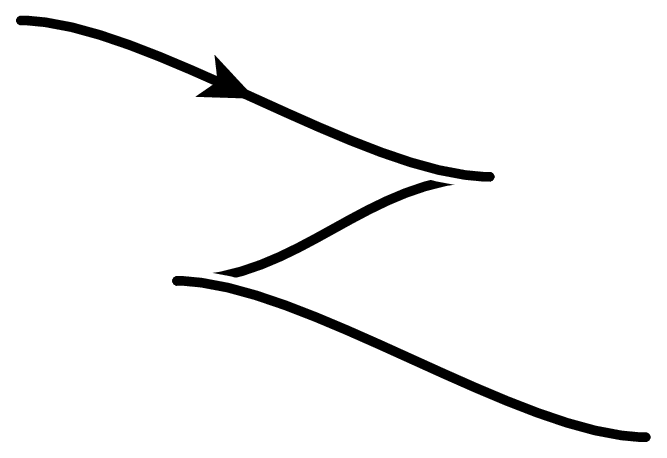}\put(-48,-5){$K'''$}&\hbox to 1 cm{\hss}&
\includegraphics[scale=0.25]{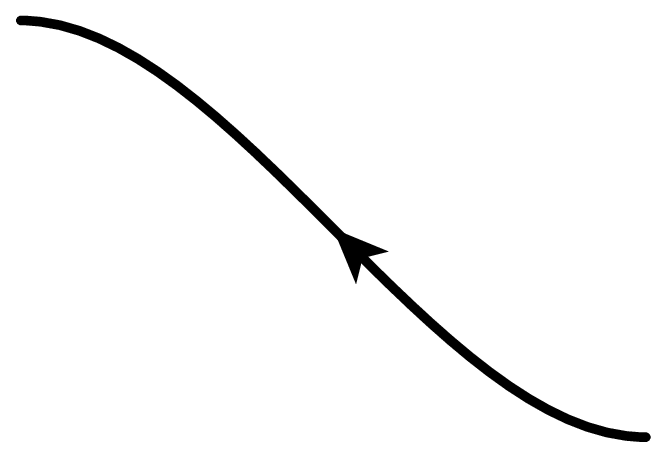}\put(-48,-5){$K''$}\raisebox{26pt}{$\longleftrightarrow$}
\includegraphics[scale=0.25]{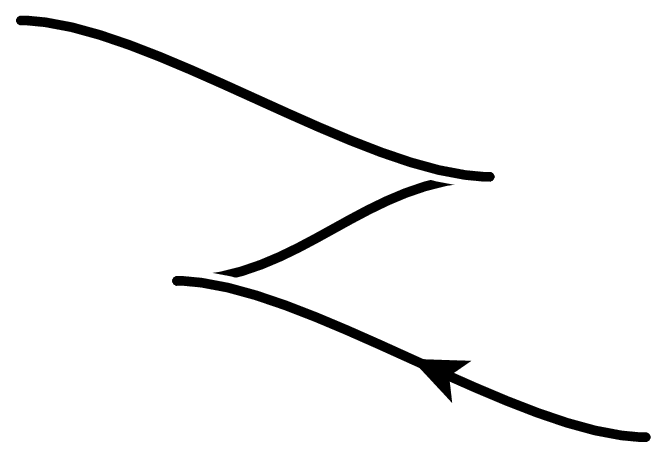}\put(-48,-5){$K'''$}\\[5mm]
(a)&&(b)
\end{tabular}
\caption{Stabilizations and destabilizations of Legendrian knots: (a) positive; (b) negative}\label{stab-fig}
\end{figure}
\end{defi}

A positive (respectively, negative) stabilization shifts the~$(\tb,r)$ pair
of the Legendrian knot by~$(-1,1)$ (respectively, by~$(-1,-1)$),
so stabilizations and destabilizations always change the equivalence class of a Legendrian knot.
If~$K$ is a Legendrian knot we denote by~$S_+(K)$ (respectively, $S_-(K)$) the result of a positive
(respectively, negative) stabilization applied to~$K$.

One can see
that the equivalence class of the Legendrian knot~$S_+(K)$ is well defined.
If~$\mathscr L$ is an equivalence class of Legendrian knots, then by~$S_+(\mathscr L)$
(respectively, $S_-(\mathscr L)$) we denote the class~$\{S_+(K):K\in\mathscr L\}$
(respectively,~$\{S_-(K):K\in\mathscr L\}$).

\begin{rema}
In the case of
links having more than one component, the result of a stabilization, viewed up to Legendrian equivalence,
depends on which component of the link the modification shown in Figure~\ref{stab-fig} is applied to,
so the notation should be refined accordingly.
\end{rema}

As shown in~\cite{futa97} any two Legendrian knots that have the same topological type can be obtained
from one another by a sequence of stabilizations and destabilizations.

\begin{defi}
If~$K$ is a $\xi_+$-Legendrian or $\xi_-$-Legendrian knot then the image of~$K$ under the transformation~$r_-\circ r_\medvert$
is called \emph{the Legendrian mirror of~$K$} and denoted by~$\mu(K)$.
\end{defi}

Note that in terms of the respective front projections Legendrian mirroring is just a rotation by~$\pi$ around the origin.
It preserves the Thurston--Bennequin number of the knot and reverses the sign of its rotation number.
Thus, if~$K$ is a Legendrian knot with~$r(K)=0$, then~$K$ and~$\mu(K)$ have the same classical invariants.
However, it happens pretty often in this case
that~$\mu(K)$ and~$K$ are not equivalent Legendrian knots (see examples in Section~\ref{app-sec}).

Similarly, if~$K$ is a Legendrian knot whose topological type is invertible,
then~$-\mu(K)$ and~$K$ have the same classical invariants, but may not be equivalent Legendrian knots.

\begin{defi}
If~$K$ is a $\xi_-$-Legendrian knot, the Thurston--Bennequin and rotation numbers of~$K$, as well as positive and
negative stabilizations,
are defined by using the mirror image~$r_\medvert(K)$ as follows:
$$\tb(K)=\tb(r_\medvert(K)),\quad r(K)=r(r_\medvert(K)),\quad S_+(K)=
r_\medvert\bigl(S_+(r_\medvert(K))\bigr),\quad S_-(K)=r_\medvert\bigl(S_-(r_\medvert(K))\bigr).$$
\end{defi}

\section{Annuli}\label{annuli-sec}

\begin{defi}
Let~$K$ be a Legendrian knot, and let~$F$ be an oriented compact surface embedded in~$\mathbb R^3$
such that~$K\subset\partial F$ and the orientation of~$K$ agrees with the induced orientation of~$\partial F$.
Let also~$v$ be a normal vector field to~$\xi_+$. By \emph{the Thurston--Bennequin number
of~$K$ relative to~$F$} denoted~$\tb(K;F)$ we call the intersection index of~$F$ with a knot obtained from~$K$
by a small shift along~$v$.

If~$F$ is an arbitrary compact surface embedded in~$\mathbb R^3$ such that~$K\subset\partial F$, then~$\tb(K;F)$
is defined as~$\tb(K;F')$, where~$F'$ is the appropriately oriented
intersection of a small tubular neighborhood~$U$ of~$K$ with~$F$ (the shift of~$K$ along~$v$ should then be chosen so
small that the shifted knot does not escape from~$U$).
\end{defi}

Let~$K$ be a Legendrian knot, and let~$F\subset\mathbb R^3$ be a compact surface such that~$K\subset\partial F$.
It is elementary to see that the following three conditions are equivalent:
\begin{enumerate}
\item
$\tb(K;F)=0$;
\item
$F$ is isotopic relative to~$K$ to a surface~$F'$ such that~$F'$ is tangent to~$\xi_+$ along~$K$;
\item
$F$ is isotopic relative to~$K$ to a surface~$F'$ such that~$F'$ is transverse to~$\xi_+$ along~$K$.
\end{enumerate}

In 3-dimensional contact topology, Giroux's convex surfaces play a fundamental role~\cite{Gi,gi1,giroux01}.
Especially important are convex annuli with Legendrian boundary and
relative Thurston--Bennequin numbers of both boundary component equal to zero,
since, vaguely speaking, any closed convex surface, viewed up
to isotopy in the class of convex surfaces, can be build up from such annuli
by gluing along a Legendrian graph.

Let~$A\subset\mathbb R^3$ be an annulus with boundary consisting of two Legendrian
knots~$K_1$ and~$K_2$ such that~$\tb(K_1;A)=\tb(K_2;A)=0$, $\partial A=K_1\cup(-K_2)$.
Then the knots~$K_1$ and~$K_2$ have the same classical invariants,
and it is natural to ask whether they must always be equivalent as Legendrian knots.

A quick look at this problem reveals no obvious reason why~$K_1$ and~$K_2$ must
be equivalent, but constructing a counterexample appears to be tricky.

Theorem~8.1 of~\cite{Gos}, which is given without a complete proof, implies
that $K_1$ and~$K_2$ are always equivalent
Legendrian knots even in
a more general situation in which~$\mathbb R^3$ is replaced by an arbitrary tight contact $3$-manifold.

However, counterexamples to this more general claim appeared earlier in a work of P.\,Ghiggini~\cite{ghi} (without a special
emphasis on the phenomenon), the
simplest of which is as follows. Endow the three-dimensional torus~$\mathbb T^3=(\mathbb R/\mathbb Z)^3$
with the contact structure~$\xi=\ker\bigl(\sin(2\pi z)\,dx+\cos(2\pi z)\,dy\bigr)$, and take for~$A$
the annulus~$(\mathbb R/\mathbb Z)\times\{0\}\times[0;1/2]$. This annulus is clearly tangent to~$\xi$
along~$\partial A$, but the boundary components are not Legendrian isotopic according to~\cite[Proposition~7.1]{ghi}
(the fact that~$(\mathbb T^3,\xi)$ is a tight contact manifold was established earlier by E.\,Giroux).

In this example, and in similar ones from~\cite{ghi}, any connected component of~$\partial A$ can be taken to the other
by a contactomorphism of~$(\mathbb T^3,\xi)$. So, it is important here that the group of contactomorphisms
of~$(\mathbb T^3,\xi)$ is disconnected, which is not the case for the standard contact structure on~$\mathbb R^3$.
Another feature of this example is that the boundary components of~$A$ are not null-homologous.

The following statement shows that the assertion of~\cite[Theorem~8.1]{Gos} is false in the case of~$\mathbb R^3$, too.

\begin{theo}\label{not-equiv-thm}
There exists an oriented annulus~$A\subset\mathbb R^3$ with boundary
$\partial A=K_1\cup(-K_2)$ such that~$K_1$ and~$K_2$ are nonequivalent Legendrian knots
having zero Thurston--Bennequin number relative to~$A$.
\end{theo}

The proof is by producing an explicit example, and the example we use here
is proposed by~I.\,Dynnikov and M.\,Prasolov in~\cite{representability}. Front projections of the Legendrian knots from this example
are shown in Figure~\ref{monster-knots-fig}. It is shown in~\cite{representability} that
they cobound an embedded annulus such that~$\tb(K_1;A)=\tb(K_2;A)=0$,
and it has been remaining unproved that~$K_1$ and~$K_2$
are not Legendrian equivalent.

\begin{figure}[ht]
\includegraphics[scale=1.3]{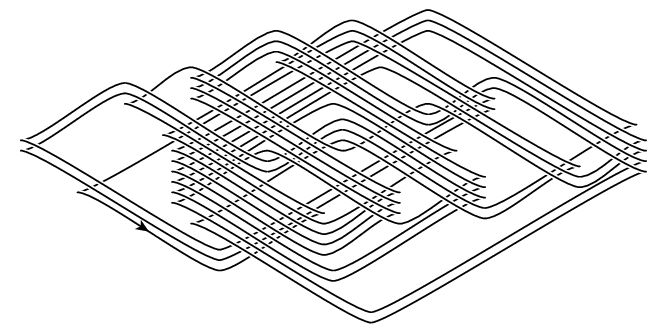}\put(-300,20){$K_1$}\\
\includegraphics[scale=1.3]{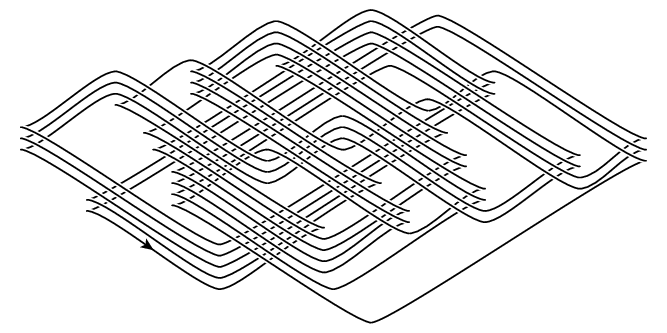}\put(-300,20){$K_2$}
\caption{Nonequivalent Legendrian knots~$K_1$ and~$K_2$ cobounding an annulus~$A$
such that~$\tb(K_1;A)=\tb(K_2;A)=0$, $\partial A=K_1\cup(-K_2)$}\label{monster-knots-fig}
\end{figure}

The proof of Theorem~\ref{not-equiv-thm} is given in Section~\ref{app-sec}.

\section{$\mathbb S^3$ settings. The orientation-preserving symmetry group}\label{s3-symmetry-group-sec}

By~$\mathbb S^3$ we denote the unit $3$-sphere in~$\mathbb R^4$,
which we identify with the group~$SU(2)$ in the standard way. We use the following
parametrization of this group:
$$(\theta,\varphi,\tau)\mapsto\begin{pmatrix}\cos(\pi\tau/2)e^{\mathbbm i\varphi}&\sin(\pi\tau/2)e^{\mathbbm i\theta}\\
-\sin(\pi\tau/2)e^{-\mathbbm i\theta}&\cos(\pi\tau/2)e^{-\mathbbm i\varphi}\end{pmatrix},$$
where~$(\theta,\varphi,\tau)\in(\mathbb R/(2\pi\mathbb Z))\times(\mathbb R/(2\pi\mathbb Z))\times[0;1]$.
The coordinate system~$(\theta,\varphi,\tau)$ can also be viewed as the one coming from the join
construction~$\mathbb S^3\cong\mathbb S^1*\mathbb S^1$, with~$\theta$ the coordinate
on~$\mathbb S^1_{\tau=1}$, and~$\varphi$ on~$\mathbb S^1_{\tau=0}$. Let~$\alpha_+$ be the following right-invariant
$1$-form on~$\mathbb S^3\cong SU(2)$:
$$\alpha_+(X)=\frac12\tr\left(X^{-1}\begin{pmatrix}\mathbbm i&0\\0&-\mathbbm i\end{pmatrix}dX\right)=
\sin^2\Bigl(\frac{\pi\tau}2\Bigr)\,d\theta+\cos^2\Bigl(\frac{\pi\tau}2\Bigr)\,d\varphi.$$

It is known (see~\cite{Ge}) that, for any point~$p\in\mathbb S^3$, there is a diffeomorphism~$\phi$ from~$\mathbb R^3$ to~$\mathbb S^3\setminus\{p\}$
that takes the contact structure~$\xi_+$ to the one defined by~$\alpha_+$, that is, to $\ker\alpha_+$.
For this reason, the latter is denoted by~$\xi_+$, too. Two Legendrian knots in~$\mathbb R^3$
are equivalent if and only if so are their images under~$\phi$ in~$\mathbb S^3$.
We will switch between the~$\mathbb R^3$ and~$\mathbb S^3$ settings depending on which is more
suitable in the current context. The~$\mathbb R^3$ settings are usually more
visual, but sometimes are not appropriate. In particular, the definition of the knot
symmetry group given below requires the~$\mathbb S^3$ settings.

\begin{defi}
Let~$K$ be a smooth knot in~$\mathbb S^3$.
Denote by~$\mathrm{Diff}^*(\mathbb S^3;K)$ the group of diffeomorphisms of~$\mathbb S^3$
preserving the orientation of~$\mathbb S^3$ and the orientation of~$K$, and
by~$\mathrm{Diff}^*_0(\mathbb S^3;K)$ the connected component of this
group containing the identity.
The group~$\mathrm{Diff}^*(\mathbb S^3;K)/\mathrm{Diff}^*_0(\mathbb S^3;K)$
is called \emph{the orientation-preserving symmetry group of~$K$}
and denoted~$\mathrm{Sym^*}(K)$.
\end{defi}

Clearly the group~$\mathrm{Sym^*}(K)$ depends only on the topological type of~$K$.
In this paper we are dealing with knots~$K$ for which~$\mathrm{Sym^*}(K)$
is a trivial group.

In the $\mathbb S^3$ settings, we also define the mirror image~$\xi_-$ of~$\xi_+$ as
$$\xi_-=\ker\left(\sin^2\Bigl(\frac{\pi\tau}2\Bigr)\,d\theta-\cos^2\Bigl(\frac{\pi\tau}2\Bigr)\,d\varphi\right).$$

\section{Rectangular diagrams of knots}\label{rd_of_knots-sec}

We denote by~$\mathbb T^2$ the two-dimensional torus~$\mathbb S^1\times\mathbb S^1$,
and by~$\theta$ and~$\varphi$ the angular coordinates on the first and the second~$\mathbb S^1$ factor, respectively.

\begin{defi}
\emph{An oriented rectangular diagram of a link} is a finite subset~$R\subset\mathbb T^2$
with an assignment~`$+$' or~`$-$' to every point in~$R$ such that every meridian~$\{\theta\}\times\mathbb S^1$
and every longitude~$\mathbb S^1\times\{\varphi\}$ contains either no or exactly two points from~$R$,
and in the latter case one of the points is assigned~`$+$' and the other~`$-$'.
The points in~$R$ are called \emph{vertices} of~$R$, and the pairs~$\{u,v\}\subset R$
such that~$\theta(u)=\theta(v)$ (respectively, $\varphi(u)=\varphi(v)$) are called \emph{vertical edges}
(respectively, \emph{horizontal edges}) of~$R$.

\emph{A rectangular diagram of a link} is defined similarly, without assignment `$+$' or~`$-$'
to vertices.

An (oriented) rectangular diagram~$R$ of a link is called \emph{an \emph(oriented\emph) rectangular diagram of a knot}
if it is \emph{connected} in the sense that, for any two vertices~$v,v'\in R$, there exists
a sequence~$v_0=v,v_1,v_2,\ldots,v_k=v'$ of vertices of~$R$ such that
any pair~$v_{i-1},v_i$ of successive elements in it is an edge of~$R$.
\end{defi}

From the combinatorial point of view, oriented rectangular diagrams of links are the same thing as grid diagrams~\cite{mos}
viewed up to cyclic permutations of rows and columns. They are also nearly the same
thing as arc-presentations (see \cite{simplification}).

\begin{conv}
In this paper we mostly work with oriented knots and knot diagrams. For brevity,
unless a rectangular diagram is explicitly specified as unoriented it is assumed
to be oriented.
\end{conv}

With every rectangular diagram of a knot~$R$ one associates a knot, denoted~$\widehat R$, in~$\mathbb S^3$
as follows. For a vertex~$v\in R$ denote by~$\widehat v$ the image of the
arc~$v\times[0;1]$ in~$\mathbb S^3\cong\mathbb S^1*\mathbb S^1=(\mathbb T^2\times[0;1])/{\sim}$
oriented from~$0$ to~$1$ if~$v$ is assigned~`$+$', and from~$1$ to~$0$ otherwise.
The knot~$\widehat R$ is by definition~$\bigcup_{v\in V}\widehat v$.

To get a planar diagram of a knot in~$\mathbb R^3$ equivalent to~$\widehat R$ one can proceed a follows.
Cut the torus~$\mathbb T^2$ along a meridian and a longitude not passing through a vertex of~$R$
to get a square. For every edge~$\{u,v\}$ of~$R$ join~$u$ and~$v$ by a straight
line segment, and let vertical segments overpass horizontal ones at every crossing point.
Vertical edges are oriented from `$+$' to~`$-$', and the horizontal ones from~`$-$' to~`$+$', see Figure~\ref{rdiagram-fig}.
\begin{figure}[ht]
\includegraphics[scale=.3]{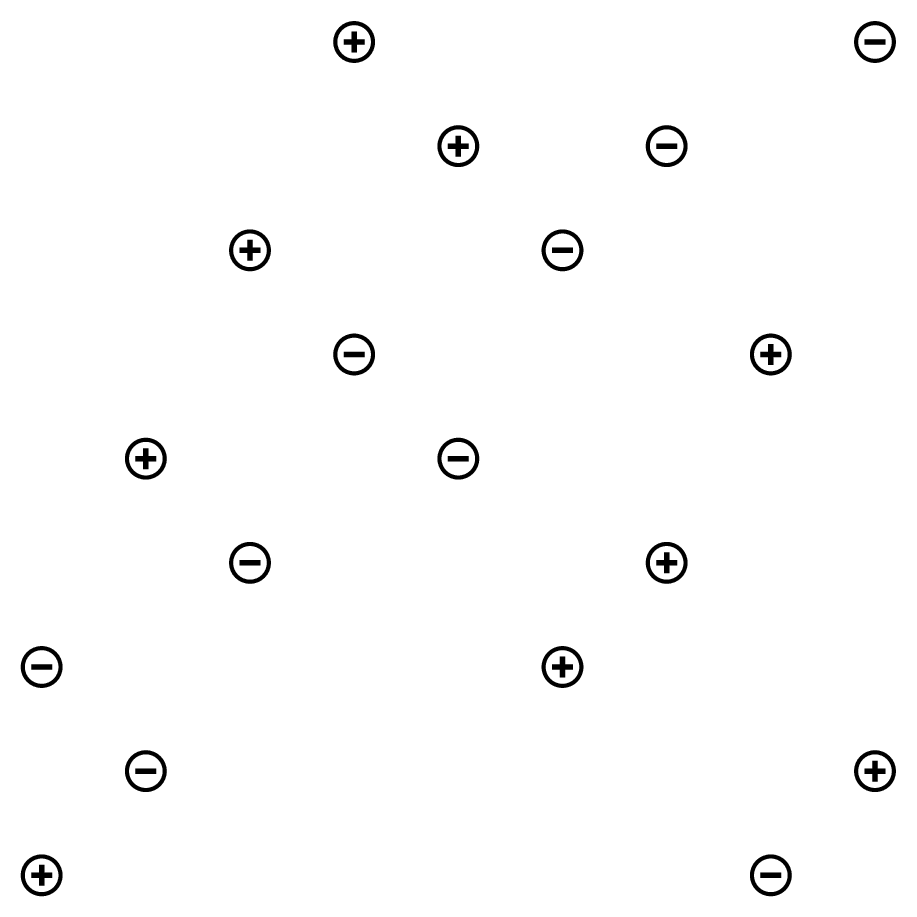}\hskip2cm
\includegraphics[scale=.3]{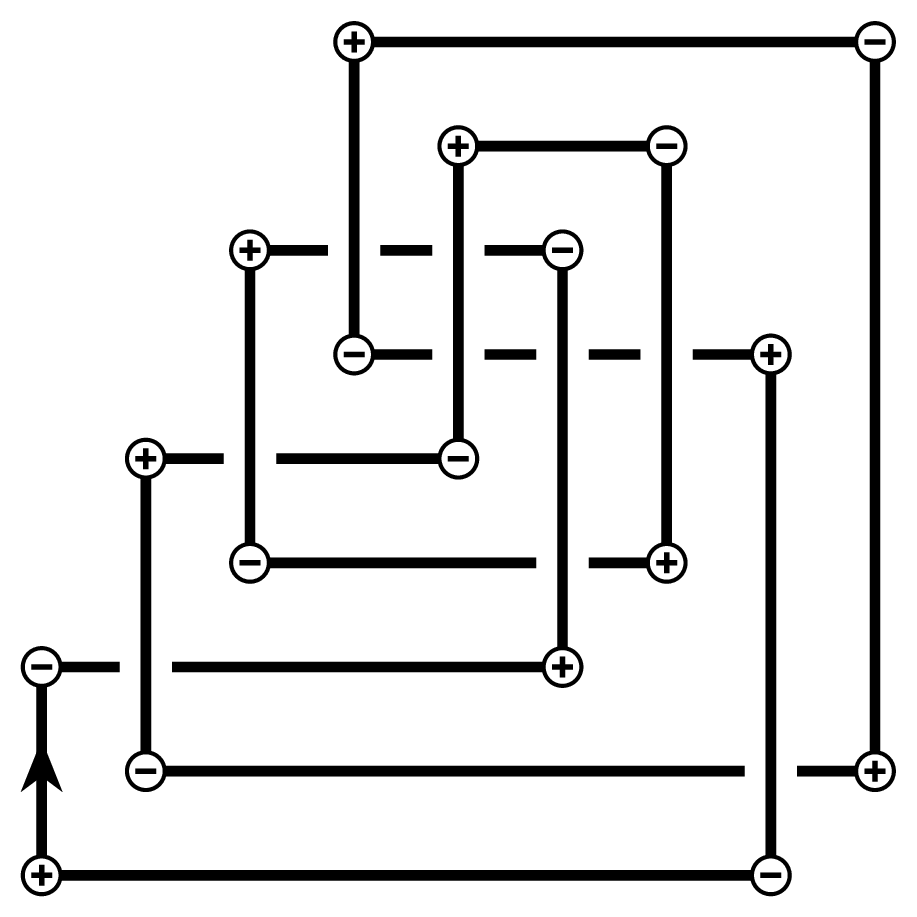}
\caption{A rectangular diagram of a knot and the corresponding planar diagram}\label{rdiagram-fig}
\end{figure}
One can show (see~\cite{simplification}) that the obtained planar diagram represents
a knot equivalent to~$\widehat R$.

For two distinct points~$x,y\in\mathbb S^1$ we denote by~$[x;y]$ the arc of~$\mathbb S^1$ such that,
with respect to the standard orientation of~$\mathbb S^1$, it has the starting point at~$x$,
and the end point at~$y$.

\begin{defi}\label{moves-def}
Let~$R_1$ and~$R_2$ be rectangular diagrams of a knot such that,
for some~$\theta_1,\theta_2,\varphi_1,\varphi_2\in\mathbb S^1$, the following holds:
\begin{enumerate}
\item
$\theta_1\ne\theta_2$, $\varphi_1\ne\varphi_2$;
\item
the symmetric difference~$R_1\triangle R_2$ is~$\{\theta_1,\theta_2\}\times\{\varphi_1,\varphi_2\}$;
\item
$R_1\triangle R_2$ contains an edge of one of the diagrams~$R_1$, $R_2$;
\item
none of~$R_1$ and~$R_2$ is a subset of the other;
\item
the intersection of the rectangle~$[\theta_1;\theta_2]\times[\varphi_1;\varphi_2]$
with~$R_1\cup R_2$ consists of its vertices, that is, $\{\theta_1,\theta_2\}\times\{\varphi_1,\varphi_2\}$;
\item
each~$v\in R_1\cap R_2$ is assigned the same sign in~$R_1$ as in~$R_2$.
\end{enumerate}
Then we say that the passage~$R_1\mapsto R_2$ is \emph{an elementary move}.

An elementary move~$R_1\mapsto R_2$ is called:
\begin{itemize}
\item
\emph{an exchange move} if~$|R_1|=|R_2|$,
\item
\emph{a stabilization move} if~$|R_2|=|R_1|+2$, and
\item
\emph{a destabilization move} if~$|R_2|=|R_1|-2$,
\end{itemize}
where~$|R|$ denotes the number of vertices of~$R$.
\end{defi}

We distinguish two \emph{types} and four \emph{oriented types} of stabilizations and destabilizations as follows.

\begin{defi}
Let~$R_1\mapsto R_2$ be a stabilization, and let~$\theta_1,\theta_2,\varphi_1,\varphi_2$ be as in Definition~\ref{moves-def}.
Denote by~$V$ the set of vertices of the rectangle~$[\theta_1;\theta_2]\times[\varphi_1;\varphi_2]$.
We say that the stabilization~$R_1\mapsto R_2$ and the destabilization~$R_2\mapsto R_1$
are of \emph{type~\rm I} (respectively, of \emph{type~\rm II}) if
$R_1\cap V\in\{(\theta_1,\varphi_1),(\theta_2,\varphi_2)\}$
(respectively, $R_1\cap V\in\{(\theta_1,\varphi_2),(\theta_2,\varphi_1)\}$).

Let~$\varphi_0\in\{\varphi_1,\varphi_2\}$ be such that~$\{\theta_1,\theta_2\}\times\{\varphi_0\}\subset R_2$.
The stabilization~$R_1\mapsto R_2$ and the destabilization~$R_2\mapsto R_1$
are of \emph{oriented type~$\overrightarrow{\mathrm I}$}
(respectively, of \emph{oriented type~$\overrightarrow{\mathrm{II}}$}) if they are of type~I (respectively, of type~II)
and~$(\theta_2,\varphi_0)$ is a positive vertex of~$R_2$.
The stabilization~$R_1\mapsto R_2$ and the destabilization~$R_2\mapsto R_1$
are of \emph{oriented type~$\overleftarrow{\mathrm I}$}
(respectively, of \emph{oriented type~$\overleftarrow{\mathrm{II}}$}) if they are of type~I (respectively, of type~II)
and~$(\theta_2,\varphi_0)$ is a negative vertex of~$R_2$.
\end{defi}

Our notation for stabilization types follows~\cite{bypasses}. The correspondence with the notation of~\cite{OST} is as follows:

\centerline{\begin{tabular}{|l|c|c|c|c|}
\hline
notation of~\cite{bypasses}&
$\overrightarrow{\mathrm I}$&$\overleftarrow{\mathrm I}$&$\overrightarrow{\mathrm{II}}$&$\overleftarrow{\mathrm{II}}$\\\hline
notation of~\cite{OST}&
\emph{X:NE}, \emph{O:SW}&
\emph{X:SW}, \emph{O:NE}&
\emph{X:SE}, \emph{O:NW}&
\emph{X:NW}, \emph{O:SE}\\\hline
\end{tabular}}

\vskip1em
With every rectangular diagram of a knot~$R$ we associate an equivalence
class~$\mathscr L_+(R)$ of $\xi_+$-Legendrian knots and an equivalence
class~$\mathscr L_-(R)$ of $\xi_-$-Legendrian knots as follows.
The front projection of a representative of~$\mathscr L_+(R)$ (respectively, of~$\mathscr L_-(R)$)
is obtained
from~$R$ in the following three steps: (1) produce a conventional planar diagram from~$R$
as described above; (2) rotate it counterclockwise (respectively, clockwise)
by any angle between~$0$ and~$\pi/2$; (3) smooth out. See Figure~\ref{associated-legendrian-fig} for an example.
\begin{figure}[ht]
\begin{tabular}{ccccc}
\includegraphics[scale=.4]{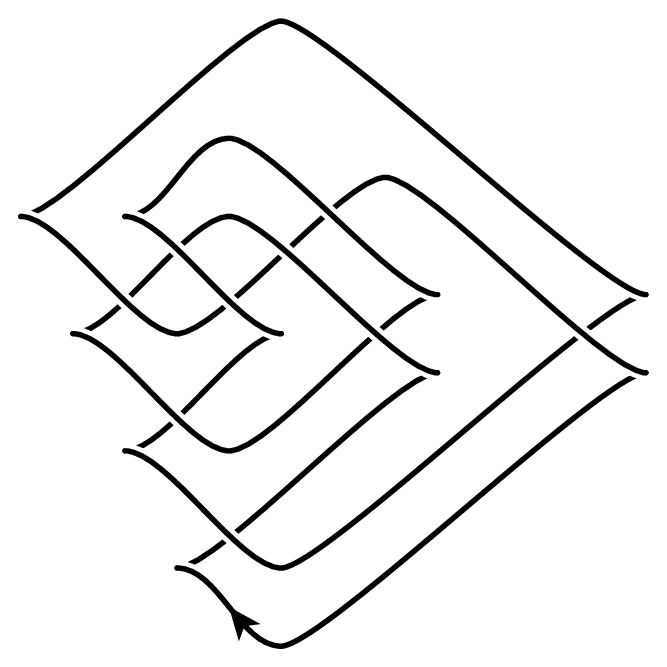}&\qquad&
\includegraphics[scale=.25]{rd2.eps}&\qquad&
\includegraphics[scale=.4]{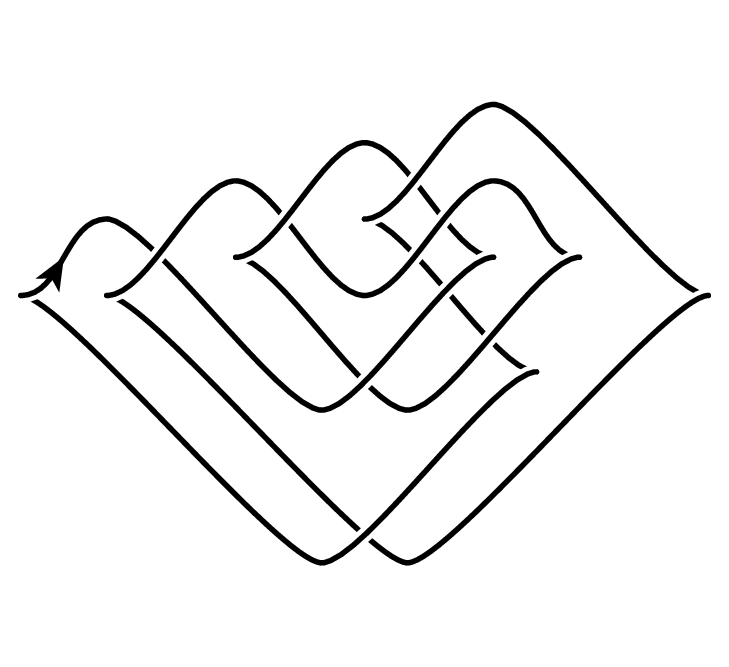}\\
a representative of $\mathscr L_+(R)$&&$R$&&a representative of $\mathscr L_-(R)$
\end{tabular}
\caption{Legendrian knots associated with a rectangular diagram of a knot}\label{associated-legendrian-fig}
\end{figure}

\begin{theo}[\cite{OST}]\label{rect-desc-of-leg-theo}
Let~$R_1$ and~$R_2$ be rectangular diagrams of a knot. The classes~$\mathscr L_+(R_1)$
and~$\mathscr L_+(R_2)$ \emph(respectively, $\mathscr L_-(R_1)$ and~$\mathscr L_-(R_2)$\emph)
coincide if and only if the diagrams~$R_1$ and~$R_2$ are related by a finite sequence
of elementary moves in which all stabilizations and destabilizations are of type~I \emph(respectively, of type~II\emph).

Moreover, if~$R_1\mapsto R_2$ is a stabilization of oriented type~$T$, then
$$\mathscr L_-(R_2)=\left\{\begin{aligned}
S_-(\mathscr L_-(R_1)),\quad&\text{if }T=\overleftarrow{\mathrm{I}},\\
S_+(\mathscr L_-(R_1)),\quad&\text{if }T=\overrightarrow{\mathrm{I}},
\end{aligned}\right.\hskip1cm
\mathscr L_+(R_2)=\left\{\begin{aligned}
S_+(\mathscr L_+(R_1)),\quad&\text{if }T=\overleftarrow{\mathrm{II}},\\
S_-(\mathscr L_+(R_1)),\quad&\text{if }T=\overrightarrow{\mathrm{II}}.
\end{aligned}\right.$$
\end{theo}

The following is the key result of the present work.

\begin{theo}\label{main-theo}
Let~$K$ be a knot with trivial orientation-preserving symmetry group, and
let~$R_1$ and~$R_2$ be rectangular diagrams of knots
isotopic to~$K$. Then the following two conditions are equivalent:
{\def\theenumi{\roman{enumi}}
\begin{enumerate}
\item
we have~$\mathscr L_+(R_1)=\mathscr L_+(R_2)$ and~$\mathscr L_-(R_1)=\mathscr L_-(R_2)$\emph;
\item
the diagram $R_2$ can be obtained from~$R_1$ by a sequence of exchange moves.
\end{enumerate}}
\end{theo}

The proof is given in the next section.

\section{Rectangular diagrams of surfaces}\label{rd_of_surfaces-sec}

Here we recall some definitions from~\cite{representability,distinguishing}.

By a \emph{rectangle} we mean a subset~$r\subset\mathbb T^2$ of the form~$[\theta_1;\theta_2]\times[\varphi_1;\varphi_2]$.
Two rectangles $r_1$, $r_2$ are said to be \emph{compatible}
if their intersection satisfies one of the following:
\begin{enumerate}
\item $r_1\cap r_2$ is empty;
\item $r_1\cap r_2$ is a subset of vertices of $r_1$ (equivalently: of~$r_2$);
\item $r_1\cap r_2$ is a rectangle disjoint from the vertices of both rectangles $r_1$ and $r_2$.
\end{enumerate}

\begin{defi}\label{rect-diagr-def}
\emph{A rectangular diagram of a surface} is a collection $\Pi=\{r_1,\ldots,r_k\}$
of pairwise compatible rectangles in~$\mathbb T^2$ such
that every meridian $\{\theta\}\times\mathbb S^1$ and every longitude $\mathbb S^1\times\{\varphi\}$
of the torus contains at most two free vertices, where by \emph{a free vertex}
we mean a point that is a vertex of exactly one rectangle in~$\Pi$.

The set of all free vertices of $\Pi$ is called \emph{the boundary of $\Pi$} and
denoted by $\partial\Pi$.
\end{defi}

One can see that the boundary of a rectangular diagram of a surface is an unoriented
rectangular diagram of a link. In particular, for any rectangle~$r$, the boundary of~$\{r\}$
is the set of vertices of~$r$, and~$\widehat{\partial\{r\}}$ is an unknot.

With every rectangular diagram of a surface~$\Pi$ one associates a $C^1$-smooth surface~$\widehat\Pi\subset\mathbb S^3$
with piecewise smooth boundary, as we now describe.

By \emph{the torus projection} we mean the map~$\mathfrak t:\mathbb S^3\setminus\bigl(\mathbb S^1_{\tau=1}\cup\mathbb S^1_{\tau=0}\bigr)\rightarrow\mathbb T^2$
defined by~$(\theta,\varphi,\tau)\mapsto(\theta,\varphi)$.
With every rectangle~$r\subset\mathbb T^2$ one can associate a disc~$\widehat r\subset\mathbb S^3$ having the form of a curved
quadrilateral contained in~$\overline{\mathfrak t^{-1}(r)}$ and spanning the loop~$\widehat{\partial\{r\}}$
so that the following conditions hold:
\begin{enumerate}
\item
for each rectangle~$r$,
the restriction of~$\mathfrak t$ to the interior of~$\widehat r$ is a one-to-one map onto the interior or~$r$;
\item
if~$r_1$ and~$r_2$ are compatible rectangles, then the interiors of~$\widehat r_1$ and~$\widehat r_2$ are disjoint;
\item
if~$r=[\theta_1;\theta_2]\times[\varphi_1;\varphi_2]$, then~$\widehat r$ is tangent to~$\xi_+$
along the sides~$\widehat{(\theta_1,\varphi_2)}$ and~$\widehat{(\theta_2,\varphi_1)}$,
and to~$\xi_-$ along the sides~$\widehat{(\theta_1,\varphi_1)}$ and~$\widehat{(\theta_2,\varphi_2)}$.
\end{enumerate}
An explicit way to define the discs~$\widehat r$, which are referred to as \emph{tiles},
is given in~\cite[Subsection~2.3]{representability}.

The surface~$\widehat\Pi$ associated with a rectangular diagram of a surface~$\Pi$ is then defined
as
$$\widehat\Pi=\bigcup_{r\in\Pi}\widehat r.$$
One can show that we have~$\partial\widehat\Pi=\widehat{\partial\Pi}$
and, for any connected component~$R$ of~$\partial\Pi$, the relative Thurston--Bennequin number~$\tb_+(\widehat R;\widehat\Pi)$
(respectively, $\tb_-(\widehat R;\widehat\Pi)$) equals minus half the number of vertices of~$R$ which are bottom-right or top-left
(respectively, bottom-left or top-right) vertices of some rectangles of~$\Pi$.

On every rectangular diagram of a surface~$\Pi$ we introduce two binary relations, $\udotdot$ and~$\ddotdot$, that
keep the information about which vertices are shared between two rectangles from~$\Pi$. Namely,
if~$r_1,r_2\in\Pi$, then $r_1\udotdot r_2$ means that~$r_1$ and~$r_2$ have the form
$$r_1=[\theta_1;\theta_2]\times[\varphi_1;\varphi_2],\quad r_2=[\theta_2;\theta_3]\times[\varphi_2;\varphi_3],$$
and~$r_1\ddotdot r_2$ means that~$r_1$ and~$r_2$ have the form
$$r_1=[\theta_1;\theta_2]\times[\varphi_2;\varphi_3],\quad r_2=[\theta_2;\theta_3]\times[\varphi_1;\varphi_2].$$

\begin{prop}\label{auxiliary-prop}
Let~$R_1$ and~$R_2$ be rectangular diagrams of a knot such that the knots~$\widehat R_1$ and~$\widehat R_2$
are topologically equivalent and have trivial orientation-preserving symmetry group. Suppose
that~$\mathscr L_+(R_1)=\mathscr L_+(R_{2})$ and~$\mathscr L_-(R_1)=\mathscr L_-(R_{2})$.

Then, for any rectangular diagram of a surface~$\Pi=\{r_1,\ldots,r_m\}$ such that~$R_1\subset\partial\Pi$, there exists
a rectangular diagram of a surface~$\Pi'=\{r_1',\ldots,r_m'\}$ and a rectangular diagram of a knot~$R_2'$ such that:
\begin{enumerate}
\item
$R_2$ and~$R_2'$ are related by a sequence of exchange moves;
\item
there exists an orientation preserving self-homeomorphism of~$\mathbb S^3$ that takes~$\widehat R_1$ to~$\widehat R_2'$,
and~$\widehat r_i$ to~$\widehat r_i'$, $i=1,\ldots,m$;
\item
$r_i\udotdot r_j\Leftrightarrow r_i'\udotdot r_j'$,
$r_i\ddotdot r_j\Leftrightarrow r_i'\ddotdot r_j'$.
\end{enumerate}
\end{prop}

\begin{proof}
This statement is a consequence of the results of~\cite[Section~2]{distinguishing}, namely, of Theorems~2.1 and~2.2,
as we will now see.
The reader is referred to~\cite[Section~2]{distinguishing} for the terminology that we use here.

Denote by~$D=(\delta_+,\delta_-)$ a canonic dividing configuration of~$\widehat\Pi$.
By hypothesis we have~$\tb_+(R_1)=\tb_+(R_2)$ and~$\tb_-(R_1)=\tb_-(R_2)$,
which implies that~$\widehat\Pi$ is both $+$-compatible and $-$-compatible with~$R_2$.
By~\cite[Theorem~2.1]{distinguishing} there exist a proper $+$-realization~$(\Pi_+,\phi_+)$ of~$\delta_+$
and a proper $-$-realization~$(\Pi_-,\phi_-)$ of~$\delta_-$ at~$R_2$.

Since the orientation-preserving symmetry group of~$\widehat R_2$ is trivial there
is an isotopy from~$\phi_+$ to~$\phi_-$ preserving~$\widehat R_2$.
One can clearly find a $-$-realization~$(\Pi_-,\phi_-')$ at~$R_2$ of an abstract
dividing set equivalent to~$\delta_-$ such that
there be an isotopy from~$\phi_+$ to~$\phi_-'$ that fixes~$\widehat R_2$ pointwise.

By~\cite[Theorem~2.2]{distinguishing}
this implies the existence of a proper realization~$(\Pi',\phi)$ of~$D$ and a rectangular diagram of a knot~$R_2'$
obtained from~$R_2$ by a sequence of exchange moves, and
such that~$\phi(\widehat R_1)=\widehat R_2'$, which is just a reformulation of the assertion of
Proposition~\ref{auxiliary-prop}.
\end{proof}

\begin{defi}
Two rectangular diagrams of a surface (or of a knot) are said to be \emph{combinatorially equivalent}
if one can be taken to the other by a homeomorphism~$\mathbb T^2\rightarrow\mathbb T^2\cong\mathbb S^1\times\mathbb S^1$ of the form
$f\times g$, where~$f$ and~$g$ are orientation preserving homeomorphisms of the circle~$\mathbb S^1$.
\end{defi}

Let~$\Pi$ be a rectangular diagram of a surface. The relations~$\udotdot$ and~$\ddotdot$ on~$\Pi$ defined above
constitute what is called in~\cite{distinguishing} the (equivalence class of a) \emph{dividing code} of~$\Pi$.
In other words, two diagrams~$\Pi_1$ and~$\Pi_2$ have equivalent dividing codes
if there is a bijection~$\Pi_1\rightarrow\Pi_2$ that preserves the relations~$\udotdot$ and~$\ddotdot$.
In general, this does not imply that the diagrams~$\Pi_1$ and~$\Pi_2$ are combinatorially equivalent
(see~\cite[Figure~2.2]{distinguishing} for an example).

\begin{lemm}\label{special-diagram-exists}
For any rectangular diagram of a link~$R$, there exists a rectangular diagram of a surface~$\Pi$
such that the following holds\emph:
\begin{enumerate}
\item
$R\subset\partial\Pi$\emph;
\item
whenever a rectangular diagram of a surface~$\Pi'$ has the same
dividing code as~$\Pi$ has, the diagrams~$\Pi$ and~$\Pi'$ are combinatorially equivalent.
\end{enumerate}
\end{lemm}

\begin{proof}
For simplicity we assume that~$R$ is connected. In the case of a many-component link
the proof is essentially the same, but a cosmetic change of notation is needed.

Let
$$(\theta_1,\varphi_1),\ (\theta_1,\varphi_2),\ (\theta_2,\varphi_2),\ldots,
(\theta_{n-1},\varphi_n),\ (\theta_n,\varphi_n),\ (\theta_n,\varphi_1)$$
be the vertices of~$R$. We put~$\theta_0=\theta_n$ and~$\varphi_0=\varphi_n$.

Pick an~$\varepsilon>0$ not larger than the length
of any of the intervals~$[\theta_i;\theta_j]$ and~$[\varphi_i;\varphi_j]$, $i\ne j$.
For~$i\in\{1,2,\ldots,n\}$ and~$j\in\{0,1,2,3,4,5\}$ denote:
$$\theta_{i,j}=\theta_i+\frac{j\varepsilon}6,\quad
\varphi_{i,j}=\varphi_i+\frac{j\varepsilon}6.$$

The sought-for diagram~$\Pi$ is constructed in the following four steps illustrated in Figure~\ref{special-rd-fig}.
\begin{figure}[ht]
\includegraphics[scale=.31]{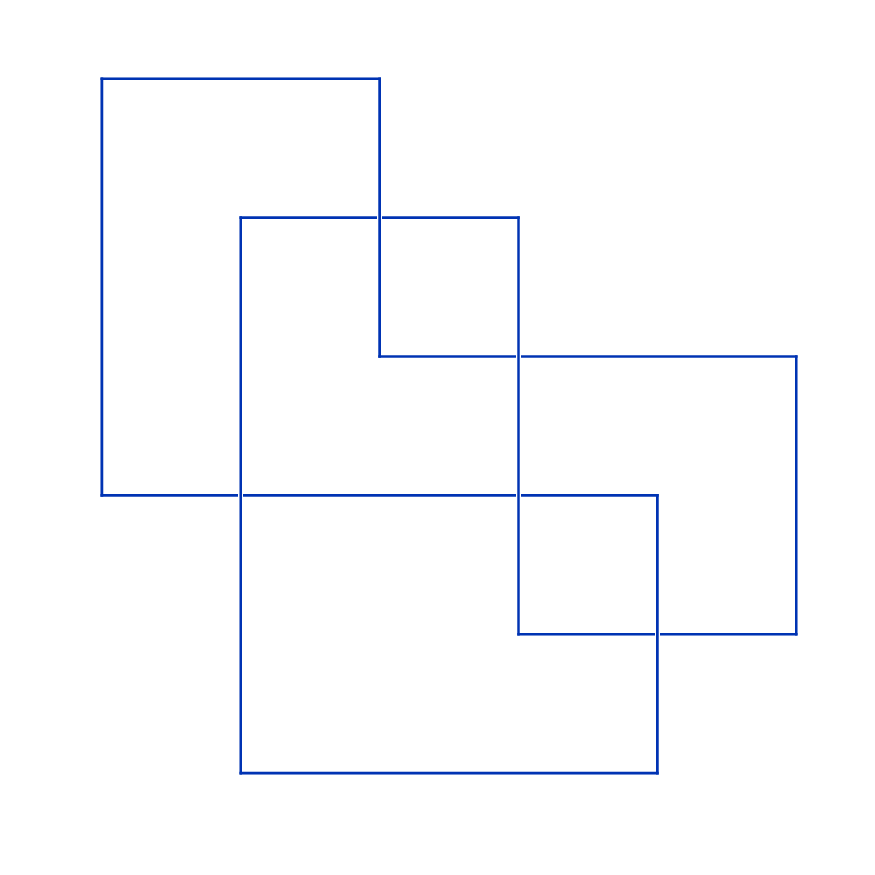}\put(-70,-7){$R$}\quad
\includegraphics[scale=.31]{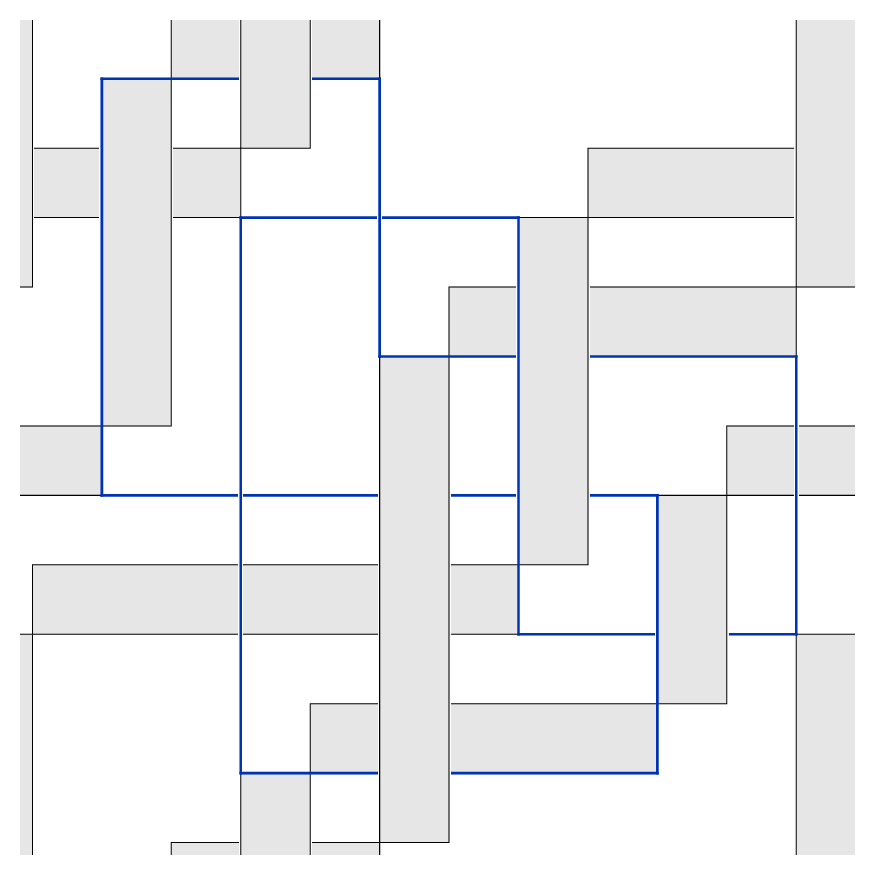}\put(-70,-7){$\Pi_1$}\quad
\includegraphics[scale=.31]{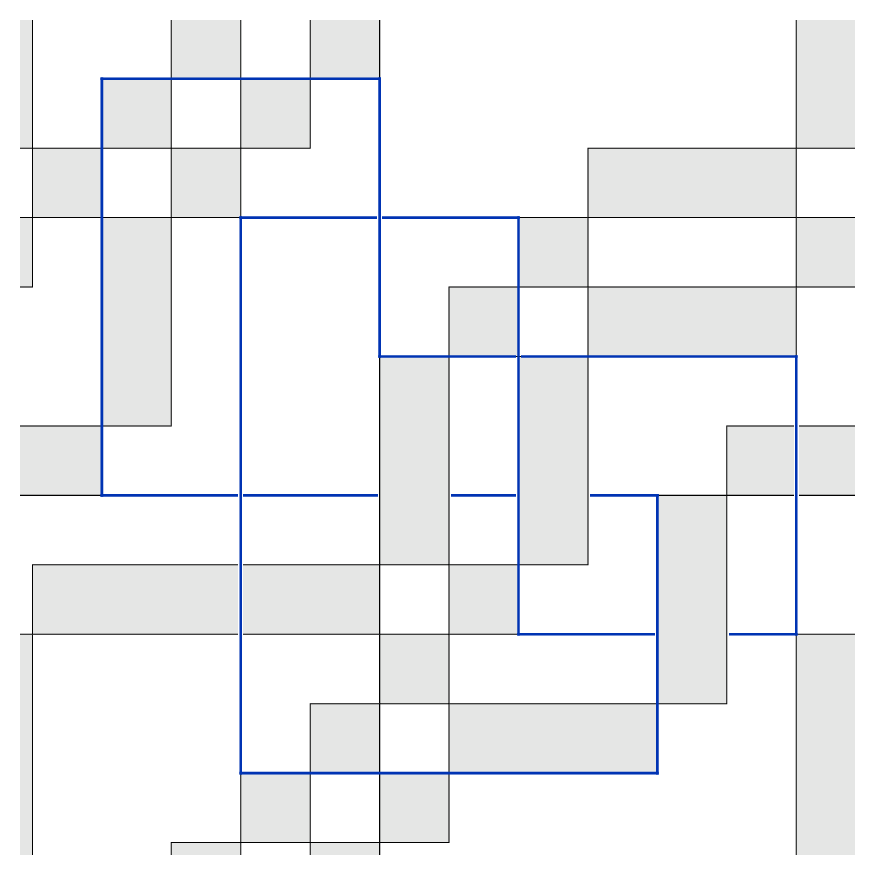}\put(-70,-7){$\Pi_2$}

\includegraphics[scale=.5]{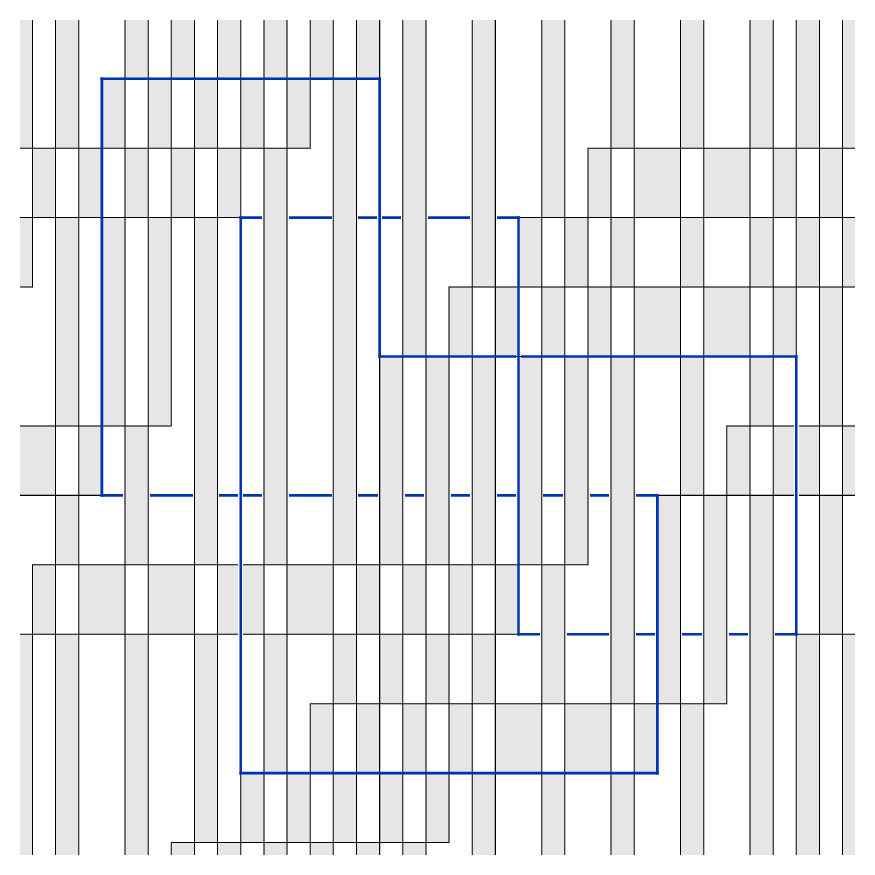}\put(-105,-7){$\Pi_3$}\quad
\includegraphics[scale=.5]{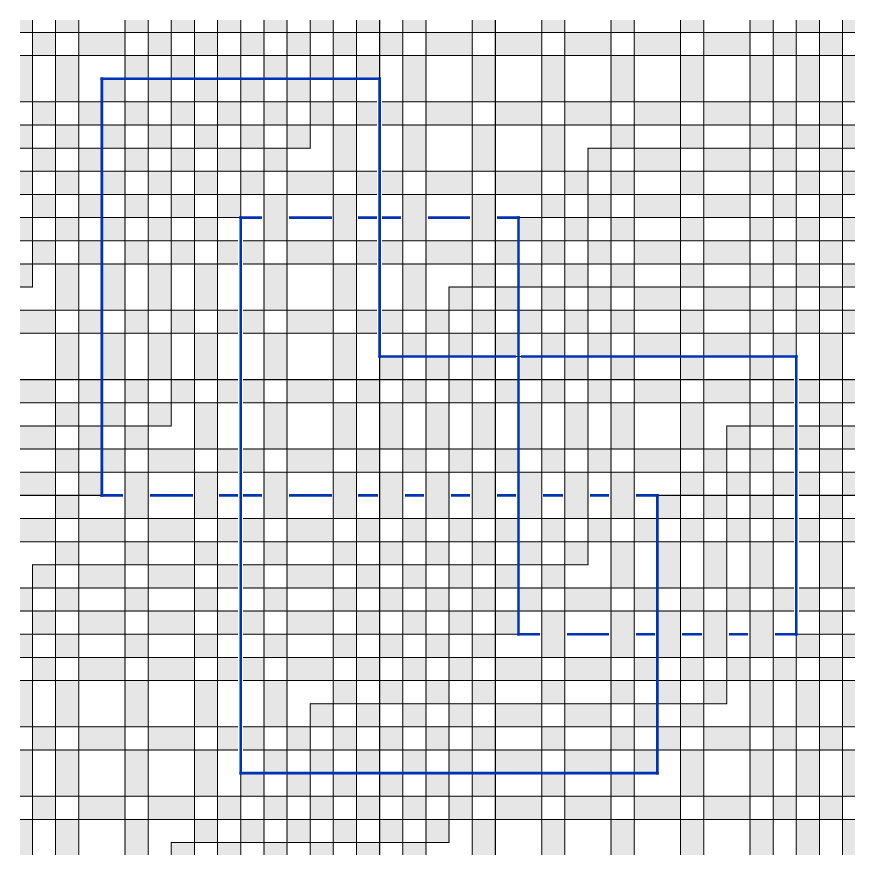}\put(-105,-7){$\Pi$}
\caption{Constructing the diagram~$\Pi$ in the proof of Lemma~\ref{special-diagram-exists}}\label{special-rd-fig}
\end{figure}

\smallskip\noindent\emph{Step~1.} Put
$$\Pi_1=\{[\theta_{i,0};\theta_{i,3}]\times[\varphi_{i,3};\varphi_{i+1,0}],
[\theta_{i,3};\theta_{i+1,0}]\times[\varphi_{i+1,0};\varphi_{i+1,3}]
\}_{i=0,1,\ldots,n}.$$

\smallskip\noindent\emph{Step~2.}
A rectangular diagram of a surface is uniquely defined by the union of its rectangles.
Define~$\Pi_2$ so that
$$\bigcup_{r\in\Pi_2}r=\overline{\bigcup_{r\in\Pi_1}r\,\setminus
\bigcup_{r,r'\in\Pi_1;\,r\ne r'}(r\cap r')}.$$

\smallskip\noindent\emph{Step~3.}
Define~$\Pi_3$ by
$$\bigcup_{r\in\Pi_3}r=\overline{\bigcup_{r\in\Pi_2}r\,\triangle
\bigcup_{i=1}^n\bigl(([\theta_{i,1};\theta_{i,2}]\cup[\theta_{i,4};\theta_{i,5}])\times\mathbb S^1\bigr)}.$$

\smallskip\noindent\emph{Step~4.}
Finally,~$\Pi$ is defined by
$$\bigcup_{r\in\Pi}r=\overline{\bigcup_{r\in\Pi_3}r\,\triangle
\bigcup_{i=1}^n\bigl(\mathbb S^1\times([\varphi_{i,1};\varphi_{i,2}]\cup[\varphi_{i,4};\varphi_{i,5}])\bigr)}.$$

One can see that~$R\subset\partial\Pi_1=\partial\Pi_2=\partial\Pi_3=\partial\Pi$. We claim that the combinatorial
type of~$\Pi$ is uniquely recovered from the dividing code of~$\Pi$.

Indeed, suppose we have forgotten the values of~$\theta_{i,j}$ and~$\varphi_{i,j}$, and
keep only the information about which pairs~$(\theta_{i,j},\varphi_{i',j'})$
are vertices of which rectangles in~$\Pi$ (this information is extracted from the dividing code).

For any~$i\in\{1,2,\ldots,n\}$ and~$j\in\{1,2,4,5\}$ the point~$(\theta_{i,j},\varphi_{1,1})$
is a vertex of some rectangle in~$\Pi$. Hence
the cyclic order on $\{\theta_{i,j}\}_{i\in\{1,2,\ldots,n\};\,j\in\{1,2,4,5\}}\subset\mathbb S^1$
is prescribed by the dividing code.

For each~$i\in\{1,2,\ldots,n\}$, we denote by~$i^-$ the unique element of~$\{1,2,\ldots,n\}$
such that~$(\theta_{i^-};\theta_i)\times\mathbb S^1$ does not contain vertices of~$R$.
One can see that for any~$i\in\{1,2,\ldots,n\}$ there exist $j,j'\in\{1,2,\ldots,n\}$
such that $(\theta_{i^-,5},\varphi_{j,1})$, $(\theta_{i,0},\varphi_{j,1})$, $(\theta_{i,1},\varphi_{j,1})$,
$(\theta_{i,2},\varphi_{j',1})$, $(\theta_{i,3},\varphi_{j',1})$, $(\theta_{i,4},\varphi_{j',1})$
are vertices of some rectangles in~$\Pi$. This prescribes the cyclic order on~$\{\theta_{i^-,5},\theta_{i,0},\theta_{i,1}\}$
and~$\{\theta_{i,2},\theta_{i,3},\theta_{i,4}\}$ for any~$i$. Therefore,
the cyclic order on~$\{\theta_{i,j}\}_{i\in\{1,2,\ldots,n\};\,j\in\{0,1,2,3,4,5\}}$ is
completely determined by the dividing code.

Similarly, completely determined by the dividing code
is the cyclic order on~$\{\varphi_{i,j}\}_{i\in\{1,2,\ldots,n\};\,j\in\{0,1,2,3,4,5\}}$,
and hence so is the combinatorial type of~$\Pi$.
\end{proof}

\begin{proof}[Proof of Theorem~\ref{main-theo}]
By Lemma~\ref{special-diagram-exists} we can find a rectangular diagram of a surface~$\Pi$
such that~$R_1\subset\partial\Pi$ and the combinatorial type of~$\Pi$ is determined by
the dividing code of~$\Pi$. We pick such~$\Pi$ and apply Proposition~\ref{auxiliary-prop}.
Since the combinatorial type of~$\Pi$ is determined by the dividing code of~$\Pi$,
we may strengthen the assertion of Proposition~\ref{auxiliary-prop} in this case
by claiming additionally that~$\Pi'=\Pi$ and~$R_2'=R_1$, which implies the assertion of the theorem.
\end{proof}

\section{Triviality of the orientation-preserving symmetry groups of some knots}\label{triviality-sec}
We use Rolfsen's knot notation~\cite{rolf}. Knots with crossing number $\leqslant10$
are well-studied (see \cite{Henry_Weeks,Kodama_Sakuma}), and the existing results about them imply the following.

\begin{prop}\label{table-knots-prop}
The orientation-preserving symmetry group of each of
the knots $9_{42}$, $9_{43}$, $9_{44}$, $9_{45}$, $10_{128}$, and~$10_{160}$
is trivial.
\end{prop}

The concrete sources for this statement are as follows.
All the knots listed in Proposition~\ref{table-knots-prop}
are known to be invertible (this can be seen from their pictures in~\cite{rolf}),
so the assertion is equivalent to saying that the symmetry group of each of the
knots is~$\mathbb Z_2$.

The knots~$9_{42}$, $9_{43}$, $9_{44}$, $9_{45}$ and $10_{128}$ are Montesinos knots (these
are introduced in~\cite{Mont}):
$$
9_{42}=K\Bigl(\frac{2}{5}, \frac{1}{3}, \frac{-1}{2}\Bigr),\
9_{43}=K\Bigl(\frac{3}{5}, \frac{1}{3}, \frac{-1}{2}\Bigr),\
9_{44}=K\Bigl(\frac{2}{5}, \frac{2}{3}, \frac{-1}{2}\Bigr),\
9_{45}=K\Bigl(\frac{3}{5}, \frac{2}{3}, \frac{-1}{2}\Bigr),\
10_{128}=K\Bigl(\frac{3}{7}, \frac{1}{3}, \frac{-1}{2}\Bigr).
$$

The knots~$9_{42}$, $9_{43}$, $9_{44}$, $9_{45}$ are elliptic Montesinos knots,
for which the symmetry group is computed by M.\,Sakuma~\cite{Sakuma}.
The symmetry group of the knot~$10_{128}$ is computed by M.\,Boileau and B.\,Zimmermann~\cite{BoZim}.
Both works are based on the technique which is due to F.\,Bonahon and L.\,Sieben\-mann~\cite{BoSie}.

The fact that the knot~$10_{160}$ is not periodic is established by~U.\,L\"udicke~\cite{Ludicke},
and that it is not freely periodic is shown by R.\,Hartley~\cite{Hartley}.

\begin{prop}\label{monster-knot-has-trivial-group-prop}
The orientation-preserving symmetry group of the \emph(topologically equivalent\emph)
knots~$K_1$ and~$K_2$ in Figure~\ref{monster-knots-fig} is trivial.
\end{prop}

\begin{proof}
We use the classical methods of the above mentioned works with some technical improvements needed
for reducing the amount of computations. `A direct check' below refers
to a computation that requires only a few minutes of a modern computer's processor time
and standard well known algorithms.

The first direct check is to see that the Alexander polynomial of~$K_1$ and~$K_2$ is
\begin{multline}\label{alexander-poly}
\Delta(t)=t^{20}-t^{19}+t^{18}-3\,t^{17}+3\,t^{16}-5\,t^{15}+10\,t^{14}-5\,t^{13}+6\,t^{12}-14\,t^{11}+15\,t^{10}-\\
14\,t^{9}+6\,t^{8}-5\,t^{7}+10\,t^{6}-5\,t^{5}+3\,t^{4}-3\,t^{3}+t^{2}-t+1.
\end{multline}

According to Murasugi~\cite{Murasugi}, if a knot has period~$p$, with~$p$ prime, then the Alexander polynomial
of this knot reduced modulo~$p$ is either the~$p$th power of a polynomial with coefficients in~$\mathbb Z_p$
or has a factor of the form~$(1+t+\ldots+t^d)^{p-1}$, where~$d\geqslant1$. It is a direct check that
neither of these occurs in the case of the polynomial~\eqref{alexander-poly}
for prime~$p\leqslant19$, and for~$p>19$ the corresponding verification is trivial.

According to {Hartley}~\cite{Hartley}, to prove that our knot has not a free period equal to~$p$ it suffices
to ensure that~$\Delta(t^p)$ does not have a self-reciprocal factor of degree~$\deg\Delta(t)=20$.
For prime~$p<100$ it can be checked directly that~$\Delta(t^p)$ is irreducible.

Suppose, for some prime~$p>100$, we have a factorization~$\Delta(t^p)=f(t)\cdot g(t)$ with self-reciprocal~$f(t),g(t)\in\mathbb Z[t]$
such that~$\deg f=20$. Since~$\Delta(0)=1$ we may assume~$f(0)=1$ without loss of generality.
For a self-reciprocal polynomial~$q(t)$ of even degree we denote by~$\widetilde q(t)$ the Laurent polynomial~$t^{-(\deg q)/2}q(t)$.

For any~$\alpha\in\{1,e^{\pi\mathbbm i/3},\mathbbm i,e^{2\pi\mathbbm i/3},-1\}$ we have
\begin{enumerate}
\item
$\alpha^p\in\{\alpha,\overline\alpha\}$;
\item
$\widetilde\Delta(\alpha)=\widetilde\Delta(\overline\alpha)$, $\widetilde f(\alpha)=\widetilde f(\overline\alpha)$;
\item
$\Delta(\alpha),f(\alpha),g(\alpha)\in\mathbb Z$.
\end{enumerate}

For~$a=(a_1,a_2,a_3,a_4,a_5)$, denote by~$\ell_a(t)\in\mathbb R[t]$ a self-reciprocal polynomial of even degree
not exceeding~$8$ such that~$\widetilde\ell_a(t)$ takes the values~$a_1,a_2,a_3,a_4,a_5$ at the
points~$t=1,e^{\pi\mathbbm i/3},\mathbbm i,e^{2\pi\mathbbm i/3},-1$, respectively.
This polynomial is clearly unique.

Now let~$a\in\mathbb Z^5$ be the list of values of~$\widetilde f$ at the points~$1,e^{\pi\mathbbm i/3},\mathbbm i,e^{2\pi\mathbbm i/3},-1$.
Then the polynomial~$t^{10}\bigl(\widetilde f(t)-\widetilde\ell(t)\bigr)$ is divisible by~$(t^6-1)(t^2+1)$.
Since this polynomial is also self-reciprocal, it is actually divisible by~$(t^6-1)(t^2+1)(t-1)$. Thus, we have
\begin{multline}\label{f-decomp-eq}
f(t)=t^{10}\,\widetilde\ell_a(t)+(t^6-1)(t^2+1)(t-1)\times{}\\
(t^{11}+b_1t^{10}+b_2t^9+b_3t^8+b_4t^7+b_5t^6+
b_5t^5+b_4t^4+b_3t^3+b_2t^2+b_1t+1).\end{multline}
Since~$\widetilde\ell_a$ may have non-zero coefficients only in front of~$t^k$ with~$k\in[-4;4]$,
we see that~$f(t)\in\mathbb Z[t]$ implies~$b_i\in\mathbb Z$, $i=1,\ldots,5$, and~$\ell_a(t)\in\mathbb Z[t]$.

One easily finds that the values of~$\widetilde\Delta(t)$ at the points~$t=1,e^{\pi\mathbbm i/3},\mathbbm i,e^{2\pi\mathbbm i/3},-1$
are~$1,-7,17,13,113$, respectively. Therefore, $a_1,a_2,a_3,a_4,a_5$ must be divisors of~$1,-7,13,13,113$, respectively. Together
with the condition~$\ell_a(t)\in\mathbb Z[t]$ this leaves us only the following 32 options for~$a$:
\begin{align*}
a&=\pm(1,1,1,1,1),&\ell_a(t)&=\pm1;\\
a&=\pm(1,1,1,13,1),&\ell_a(t)&=\pm(-2t^8+2t^7-2t^5+5t^4-2t^3+2t-2);\\
a&=\pm(1,1,17,1,1),&\ell_a(t)&=\pm(4t^8-4t^6+t^4-4t^2+4);\\
a&=\pm(1,1,17,13,1),&\ell_a(t)&=\pm(2t^8+2t^7-4t^6-2t^5+5t^4-2t^3-4t^2+2t+2);\\
a&=\pm(1,-1,1,1,113),&\ell_a(t)&=\pm(5t^8-9t^7+14t^6-19t^5+19t^4-19t^3+14t^2-9t+5);\\
a&=\pm(1,-1,1,13,113),&\ell_a(t)&=\pm(3t^8-7t^7+14t^6-21t^5+23t^4-21t^3+14t^2-7t+3);\\
a&=\pm(1,-1,17,1,113),&\ell_a(t)&=\pm(9t^8-9t^7+10t^6-19t^5+19t^4-19t^3+10t^2-9t+9);\\
a&=\pm(1,-1,17,13,113),&\ell_a(t)&=\pm(7t^8-7t^7+10t^6-21t^5+23t^4-21t^3+10t^2-7t+7);\\
a&=\pm(1,7,1,1,1),&\ell_a(t)&=\pm(-t^8-t^7+t^5+3t^4+t^3-t-1);\\
a&=\pm(1,7,1,13,1),&\ell_a(t)&=\pm(-3t^8+t^7-t^5+7t^4-t^3+t-3);\\
a&=\pm(1,7,17,1,1),&\ell_a(t)&=\pm(3t^8-t^7-4t^6+t^5+3t^4+t^3-4t^2-t+3);\\
a&=\pm(1,7,17,13,1),&\ell_a(t)&=\pm(t^8+t^7-4t^6-t^5+7t^4-t^3-4t^2+t+1);\\
a&=\pm(1,-7,1,1,113),&\ell_a(t)&=\pm(6t^8-8t^7+14t^6-20t^5+17t^4-20t^3+14t^2-8t+6);\\
a&=\pm(1,-7,1,13,113),&\ell_a(t)&=\pm(4t^8-6t^7+14t^6-22t^5+21t^4-22t^3+14t^2-6t+4);\\
a&=\pm(1,-7,17,1,113),&\ell_a(t)&=\pm(10t^8-8t^7+10t^6-20t^5+17t^4-20t^3+10t^2-8t+10);\\
a&=\pm(1,-7,17,13,113),&\ell_a(t)&=\pm(8t^8-6t^7+10t^6-22t^5+21t^4-22t^3+10t^2-6t+8).
\end{align*}

It is another direct check that all roots of~$\Delta$ are located inside the circle~$\{z\in\mathbb C:|z|<3/2\}$.
Therefore, the roots of~$f$ are contained in the circle~$\{z\in\mathbb C:|z|<(3/2)^{1/p}\}$.

For~$k\in\mathbb N$, denote by~$p_k$ the $k$th Newton's sum of~$f$, that is, the sum of the $k$th powers of the roots.
They must
be integers, and we have the following estimate for their absolute values:
\begin{equation}\label{newton-ineq}
|p_k|<20\cdot(3/2)^{k/p}.
\end{equation}
Since~$p>100$, this implies, in particular, that
\begin{equation}\label{newton-ineq-5}
|p_k|\leqslant20\quad\text{for }k=1,2,3,4,5.
\end{equation}

Denote by~$c_k$, $k=1,2,\ldots,19$, the coefficients of~$f$: $f=1+c_1t+c_2t^2+\ldots+c_{19}t^{19}+t^{20}$,
$c_i=c_{20-i}$.
The first (equivalently: the last) five of them are related with~$p_i$ by the following Newton's identities:
$$\begin{aligned}
-p_1&=c_1,\\
-p_2&=c_1p_1+2c_2,\\
-p_3&=c_1p_2+c_2p_1+3c_3,\\
-p_4&=c_1p_3+c_2p_2+c_3p_1+4c_4,\\
-p_5&=c_1p_4+c_2p_3+c_3p_2+c_4p_1+5c_5.
\end{aligned}$$
This Diophantine system has exactly 971\,865 solutions satisfying~\eqref{newton-ineq-5},
which can be searched (another direct check).
The coefficients~$b_1,b_2,b_3,b_4,b_5$ in~\eqref{f-decomp-eq} can obviously be expressed through~$c_1,c_2,c_3,c_4,c_5$.
Thus, we get only $32\cdot 971865=31\,099\,680$ possible candidates for~$f$, and it is the last direct check
that the $k$th Newton's sum of each of the obtained polynomials violates~\eqref{newton-ineq}
for some~$k\leqslant31$ with any~$p>100$. A contradiction.

We have thus established that the orientation preserving symmetry group of the knots~$K_1$ and~$K_2$
has no finite-order elements. It remains to ensure that these knots are not satellite knots (that is, they are hyperbolic).
A way to verify this is explained in the Appendix.
\end{proof}

Proposition~\ref{monster-knot-has-trivial-group-prop} is also directly confirmed by the SnapPy program~\cite{SnapPy}.
For the reader's convenience, we provide here a Dowker--Thistlethwaite code of the diagram
of~$K_1$ shown in Figure~\ref{monster-knots-fig} (the numeration of the crossings starts from the arrowhead):

\smallskip
\begin{mdframed}[leftmargin=5mm, linewidth=0pt, skipbelow=0pt]
$-462$, $-346$, $-76$, $-218$, $156$, $472$, $356$, $66$, $208$, $126$, $-324$, $444$, $132$, $202$, $60$, $362$,
$180$, $-478$, $-338$, $-284$, $-452$, $246$, $302$, $188$, $-460$, $400$, $-296$, $-492$, $-450$, $-286$, $-230$, $-88$,
$-172$, $-122$, $-418$, $352$, $468$, $160$, $-276$, $220$, $154$, $474$, $334$, $384$, $412$, $502$, $-442$, $-24$,
$134$, $200$, $58$, $40$, $146$, $-366$, $-184$, $-222$, $-80$, $8$, $314$, $264$, $380$, $416$, $506$, $-168$,
$-92$, $-234$, $-290$, $-446$, $196$, $54$, $456$, $-488$, $-46$, $-282$, $-340$, $-480$, $-430$, $-272$, $-116$, $-424$,
$-372$, $-256$, $508$, $20$, $504$, $414$, $382$, $266$, $-84$, $-226$, $36$, $150$, $278$, $344$, $-396$, $-458$,
$-52$, $-294$, $-494$, $-448$, $-288$, $-232$, $-90$, $-170$, $-124$, $68$, $354$, $470$, $158$, $-274$, $-428$, $152$,
$476$, $336$, $386$, $410$, $500$, $94$, $-22$, $136$, $198$, $56$, $42$, $-392$, $-140$, $192$, $50$, $2$,
$320$, $-350$, $-72$, $-214$, $-12$, $-332$, $176$, $312$, $114$, $426$, $-482$, $-342$, $106$, $304$, $388$,
$408$, $498$, $-96$, $-238$, $402$, $250$, $-142$, $-394$, $-104$, $-364$, $-182$, $-224$, $-82$, $-10$, $-216$, $-74$,
$-348$, $-464$, $166$, $128$, $206$, $64$, $358$, $-268$, $-434$, $-32$, $306$, $108$, $368$, $-484$, $162$, $466$,
$-376$, $-420$, $-120$, $-174$, $-86$, $-228$, $38$, $148$, $280$, $-186$, $4$, $318$, $260$, $-70$, $-212$, $-14$,
$-330$, $-436$, $-30$, $-102$, $-244$, $454$, $144$, $486$, $-252$, $194$, $-138$, $-242$, $-100$, $-28$, $-438$, $-328$,
$-16$, $-210$, $378$, $262$, $316$, $6$, $-78$, $112$, $310$, $178$, $360$, $62$, $204$, $130$, $-236$, $292$,
$496$, $406$, $390$, $-44$, $-490$, $-298$, $398$, $254$, $164$, $-258$, $-374$, $-422$, $-118$, $-270$, $-432$, $-34$,
$308$, $110$, $370$, $-48$, $190$, $300$, $248$, $404$, $-240$, $-98$, $-26$, $-440$, $-326$, $-18$, $-322$.
\end{mdframed}

\section{Applications}\label{app-sec}

\begin{theo}
There exists an algorithm that decides in finite time whether or not two
given Legendrian knots,~$L_1$ and~$L_2$, say, are equivalent provided that
they are topologically equivalent and have trivial orientation-preserving symmetry
group.
\end{theo}

\begin{proof}
It is understood that~$L_1$ and~$L_2$ are presented in a combinatorial
way that allows to recover actual curves in~$\mathbb R^3$.
Whichever presentation is chosen, it can always be converted into rectangular
diagrams. So, we assume that we are given two rectangular diagrams of a knot,
$R_1$ and~$R_2$, say, such that~$\mathscr L_+(R_1)\ni L_1$ and~$\mathscr L_+(R_2)\ni L_2$.

By~\cite[Theorem~7]{bypasses} there exists a rectangular diagram of a knot~$R_3$
such that~$\mathscr L_+(R_3)=\mathscr L_+(R_1)$ and~$\mathscr L_-(R_3)=\mathscr L_-(R_2)$.
According to Theorem~\ref{rect-desc-of-leg-theo} this is equivalent to saying
that there exists a sequence of elementary moves transforming~$R_1$ to~$R_3$ (respectively,
$R_3$ to~$R_2$) including only exchange moves and type~I (respectively, type~II)
stabilizations and destabilizations. Therefore, such an~$R_3$
can be found by an exhaustive search of sequences of elementary moves
starting at~$R_1$ in which all type~I stabilizations and destabilizations occur before
all type~II ones. Indeed, the combinatorial types of such sequences are enumerable.The search terminates once a sequence with
the above properties arriving at~$R_2$ is encountered. By~\cite[Theorem~7]{bypasses}
this must eventually happen.

Once~$R_3$ is found we check whether or not it is related to~$R_2$
by a sequence of exchange moves. The latter can produce only finitely many combinatorial
types of diagrams from the given one, so this process is finite.
According to Theorem~\ref{main-theo} the diagrams~$R_2$ and~$R_3$
are related by a sequence of exchange moves if and only if~$\mathscr L_+(R_2)=\mathscr L_+(R_3)$,
which is equivalent to~$\mathscr L_+(R_1)=\mathscr L_+(R_2)$.
\end{proof}

Now we use Theorem~\ref{main-theo} to establish some facts that are left in~\cite{chong2013}
as conjectures. These involve knots with trivial orientation-preserving symmetry group that are
listed in Proposition~\ref{table-knots-prop} above.

For a rectangular diagram of a knot~$R$, the set of all rectangular diagrams
obtained from~$R$ by a sequence of exchange moves is called
\emph{the exchange class of~$R$}.

In what follows we use the following notation system. $\xi_+$-Legendrian classes of knots having topological
type~$m_n$ are denoted~$m_n^{k+}$, $k=1,2,\ldots$, or simply~$m_n^+$ if we need to consider only
one Legendrian class and its images under~$\mu$ and orientation reversal. Similarly,
for $\xi_-$-Legendrian we use notation of the form~$m_n^{k-}$ or~$m_n^-$, and
for exchange classes~$m_n^{k\mathrm R}$ or~$m_n^{\mathrm R}$.

The $\xi_\pm$-Legendrian classes and exchange classes of our interest are
defined by specifying a representative.
In order to help the reader to see the correspondence with the notation of~\cite{chong2013}
we define the~$\xi_-$-Legendrian classes via their mirror images, which are $\xi_+$-Legendrian classes.

We use the same notation for natural operations on (exchange classes of) rectangular diagrams as for Legendrian knots:
`$-$' for orientation reversal, $r_\medvert$ and~$r_-$ for the horizontal and the vertical flip, respectively,
and~$\mu$ for~$r_\medvert\circ r_-$. One can see that if~$X$ is an exchange class,
then~$\mathscr L_\pm(-X)=-\mathscr L_\pm(X)$, $\mathscr L_\pm(\mu(X))=\mu(\mathscr L_\pm(X))$,
$\mathscr L_\pm(r_\medvert(X))=r_\medvert(\mathscr L_\mp(X))$.

\tikzset{res/.style={ellipse,draw,minimum height=0.5cm,minimum width=0.8cm}}
\tikzset{literal/.style={rectangle,draw,minimum height=0.5cm,minimum width=0.8cm,text width = 1.2 cm, align = center}}
\tikzset{hfit/.style={rounded rectangle,draw, inner xsep=0pt},
           vfit/.style={rounded corners,draw}}
\begin{figure}[ht!]
\includegraphics[scale=.5]{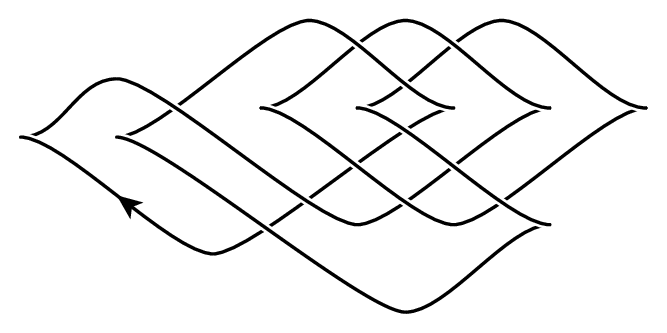}\put(-95,0){$9_{42}^+$}
\includegraphics[scale=.5]{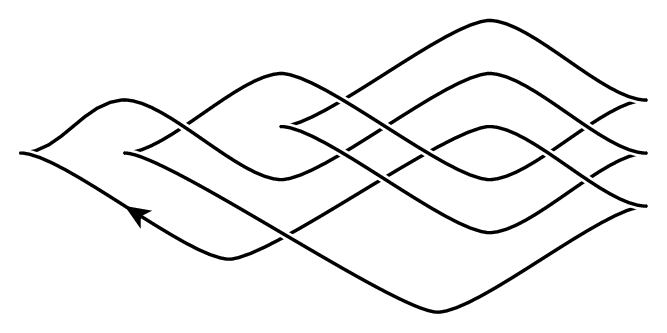}\put(-95,0){$r_\medvert(9_{42}^-)$}
\hskip5mm
\includegraphics[scale=.18]{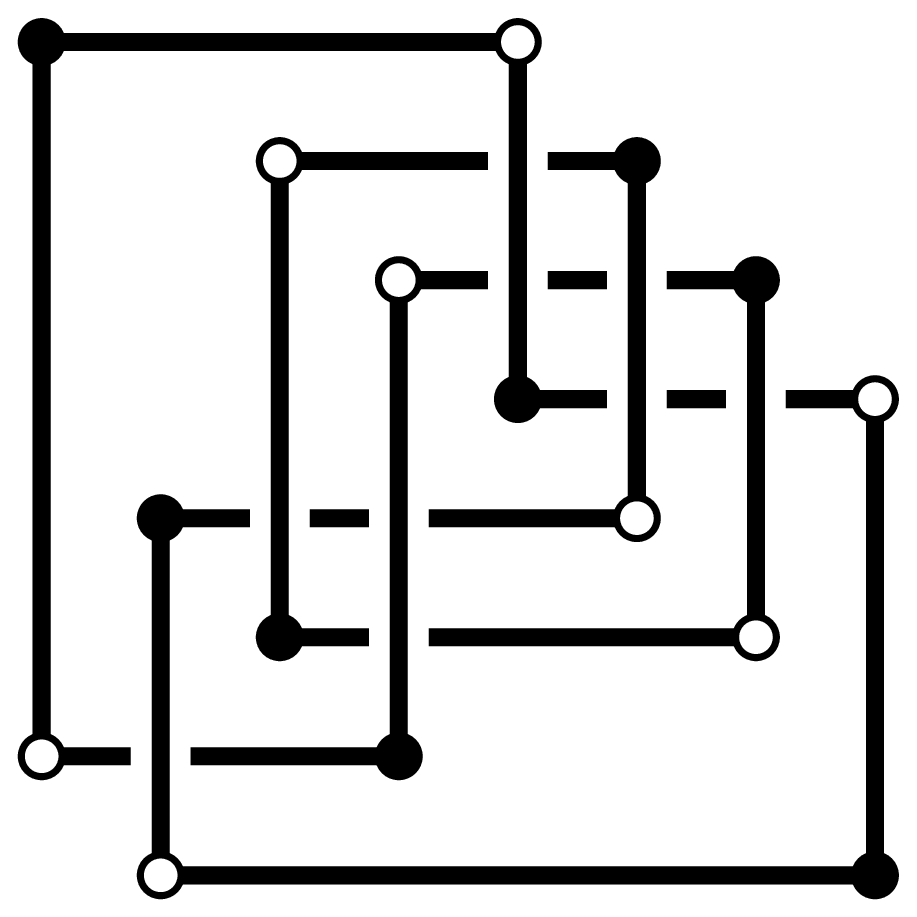}\put(-45,-10){$9_{42}^{\mathrm R}$}
\caption{Legendrian knots in Proposition~\ref{9_42_prop} and an exchange class representing both}\label{9_42_fig}
\end{figure}
\begin{prop}\label{9_42_prop}
For the classes~$9_{42}^+$ and~$9_{42}^-$ whose representatives are shown in Figure~\ref{9_42_fig}, we have
$9_{42}^+{}=-9_{42}^+\ne\mu(9_{42}^+)$ and $9_{42}^-\ne-9_{42}^-{}=\mu(9_{42}^-)$.
\end{prop}

\begin{proof}
We use the exchange class~$9_{42}^{\mathrm R}$ of the diagram shown in Figure~\ref{9_42_fig} on the right. Black
vertices are positive, and white ones are negative.

It is an easy check that the diagram representing the class~$9_{42}^{\mathrm R}$
in Figure~\ref{9_42_fig} admits no non-trivial (that is, changing the combinatorial
type) exchange move, and its combinatorial type changes under reversing
the orientation and under its composition with the rotation~$\mu$. We conclude from this
that
\begin{equation}\label{9-42-ne-eq}
9_{42}^{\mathrm R}\ne-9_{42}^{\mathrm R}\quad\text{and}\quad-9_{42}^{\mathrm R}\ne\mu(9_{42}^{\mathrm R}).
\end{equation}

Now we verify directly that
$S_{\overrightarrow{\mathrm I}}(9_{42}^{\mathrm R})=S_{\overrightarrow{\mathrm I}}(-9_{42}^{\mathrm R})$:
$$\begin{array}{ccccccccccccccccc}
\includegraphics[width=30pt]{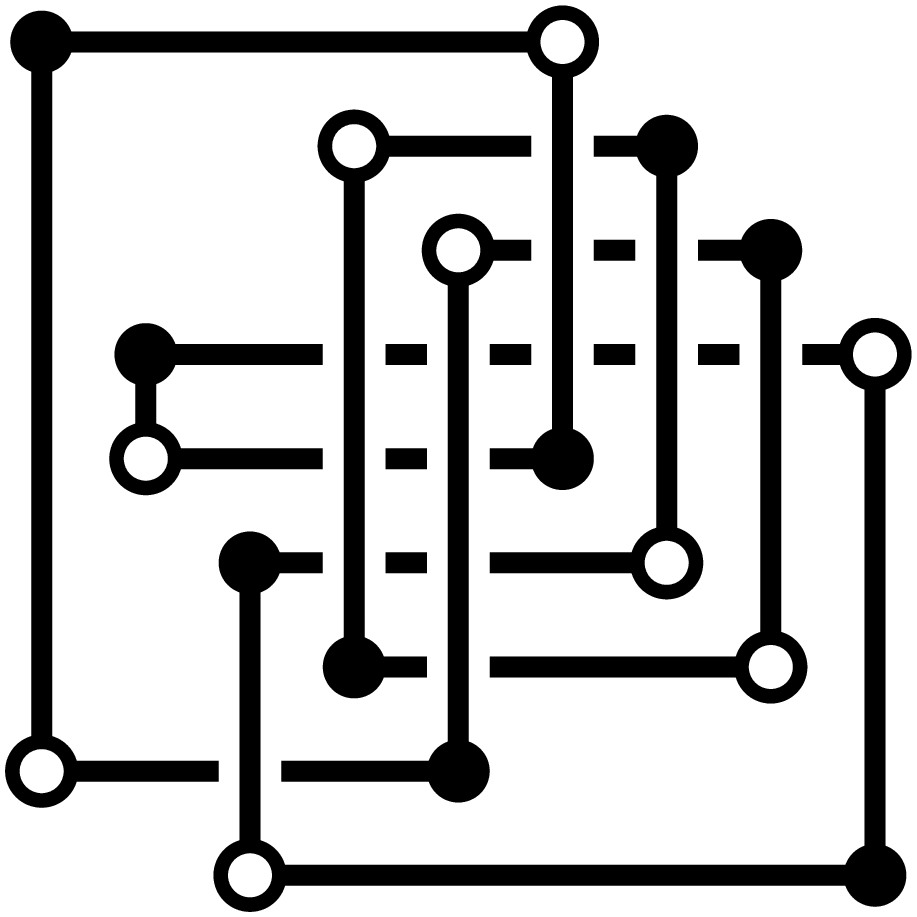}&\kern-.6em\raisebox{12pt}{$\rightarrow$}\kern-.6em&
\includegraphics[width=30pt]{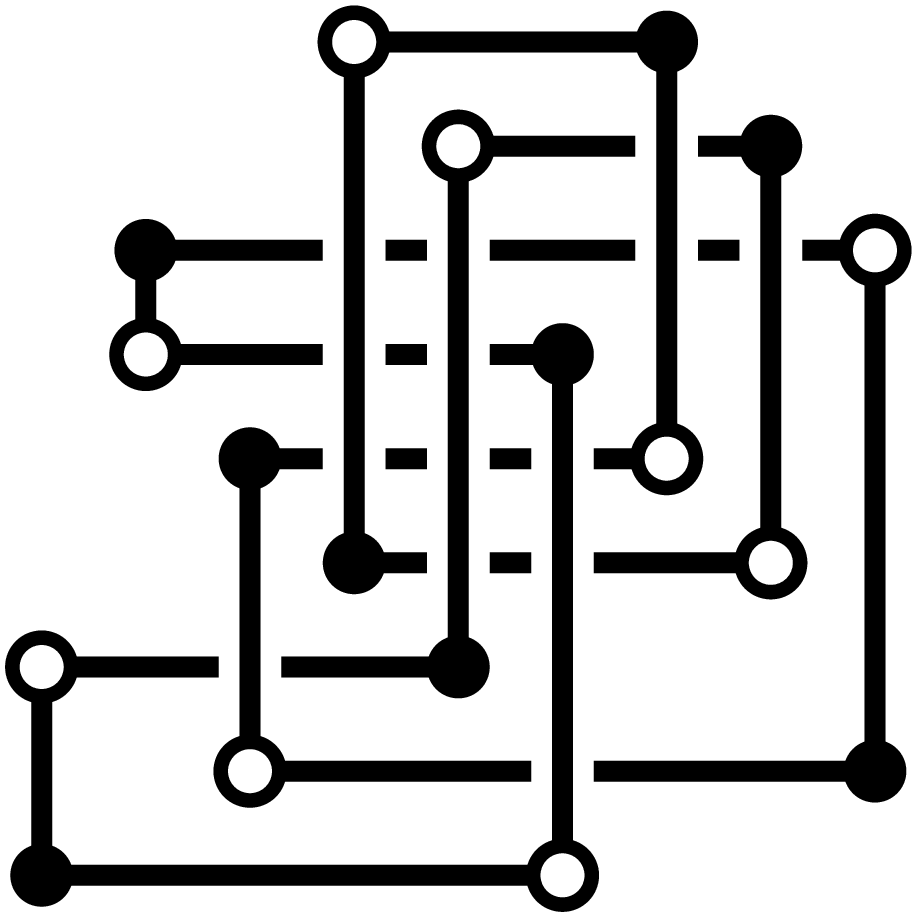}&\kern-.6em\raisebox{12pt}{$\rightarrow$}\kern-.6em&
\includegraphics[width=30pt]{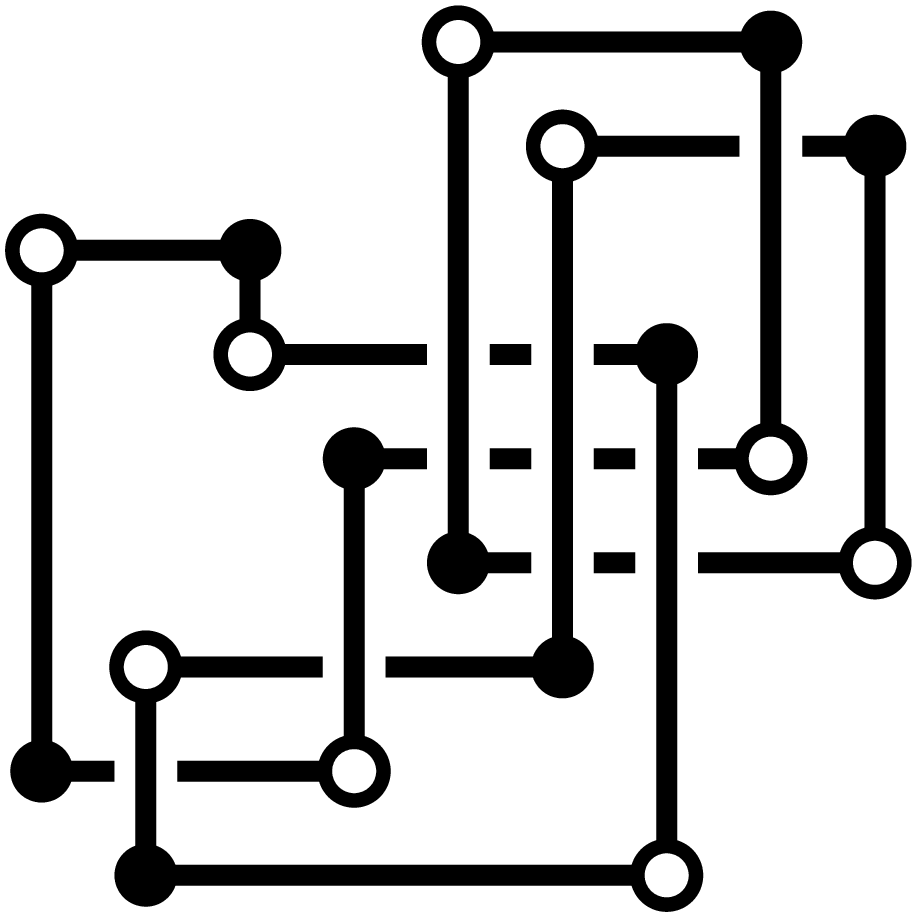}&\kern-.6em\raisebox{12pt}{$\rightarrow$}\kern-.6em&
\includegraphics[width=30pt]{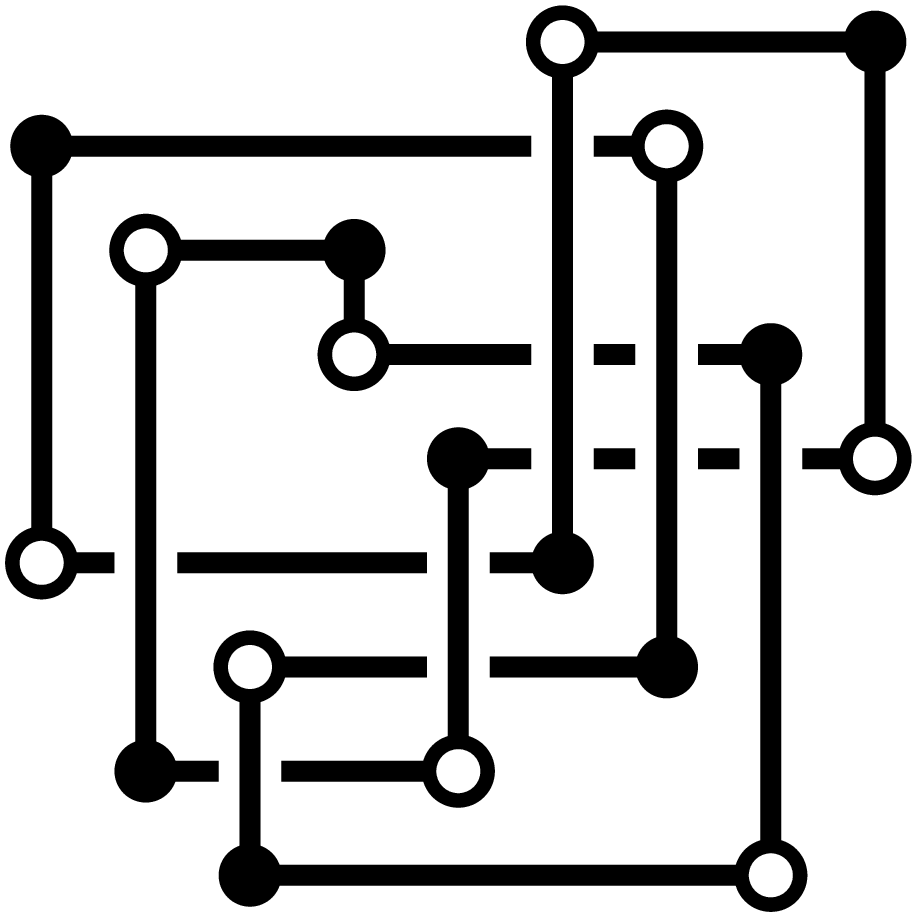}&\kern-.6em\raisebox{12pt}{$\rightarrow$}\kern-.6em&
\includegraphics[width=30pt]{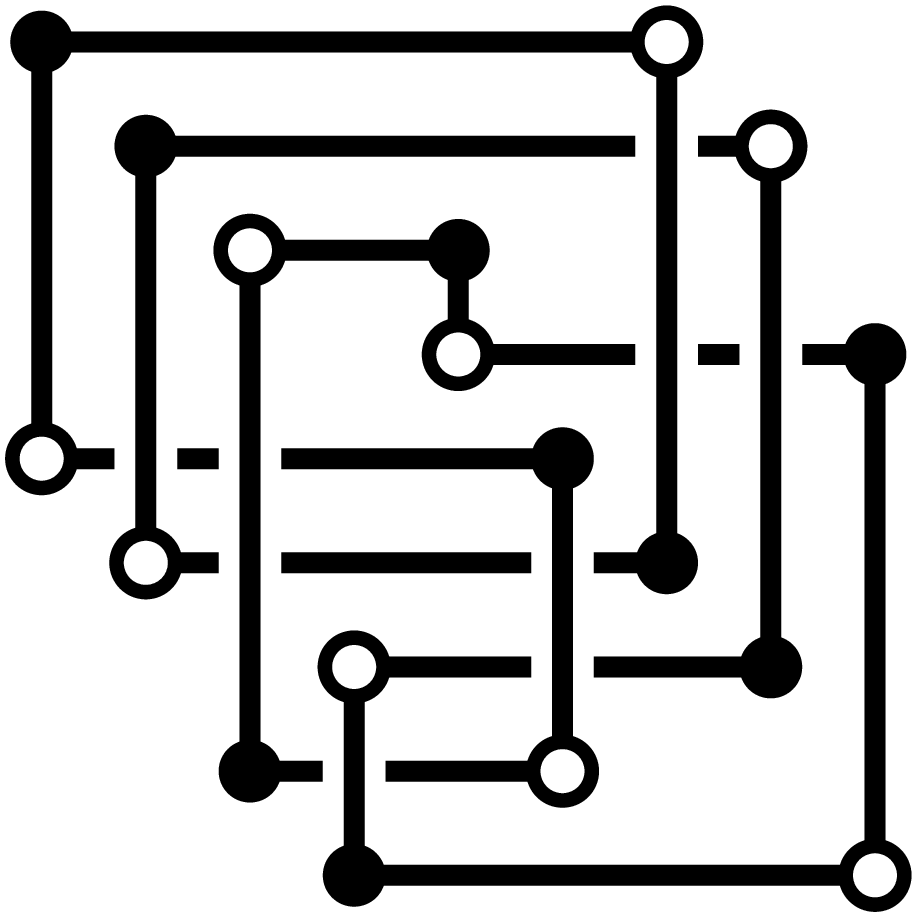}&\kern-.6em\raisebox{12pt}{$\rightarrow$}\kern-.6em&
\includegraphics[width=30pt]{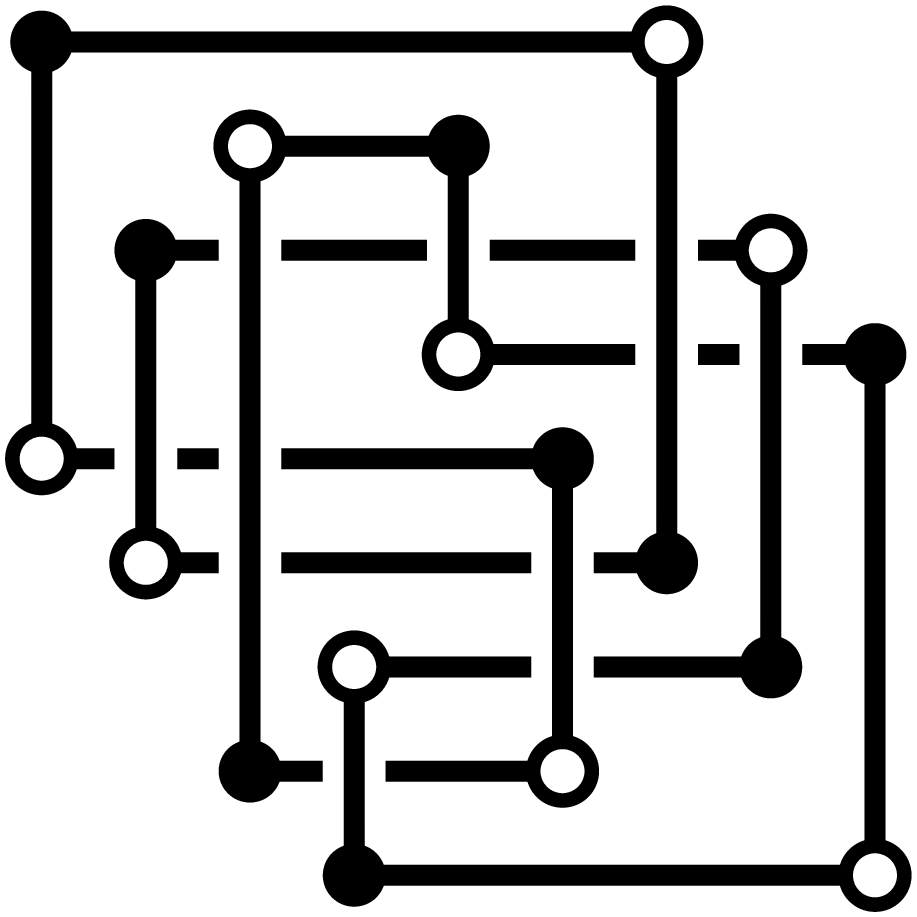}&\kern-.6em\raisebox{12pt}{$\rightarrow$}\kern-.6em&
\includegraphics[width=30pt]{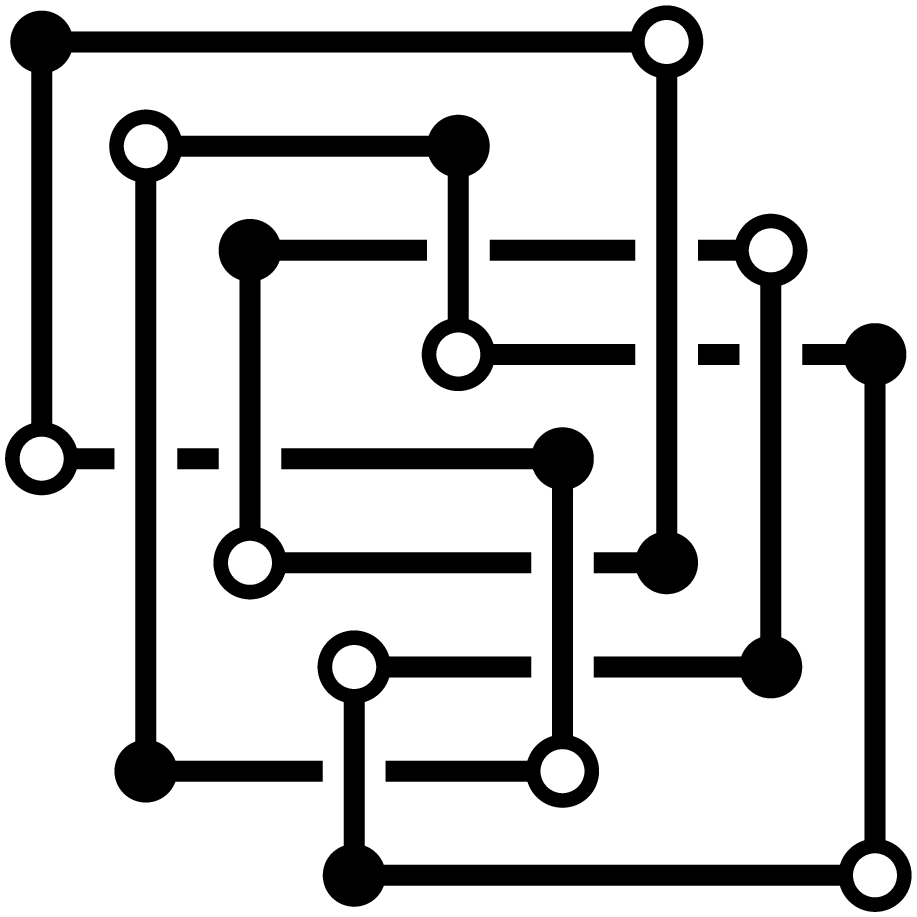}&\kern-.6em\raisebox{12pt}{$\rightarrow$}\kern-.6em&
\includegraphics[width=30pt]{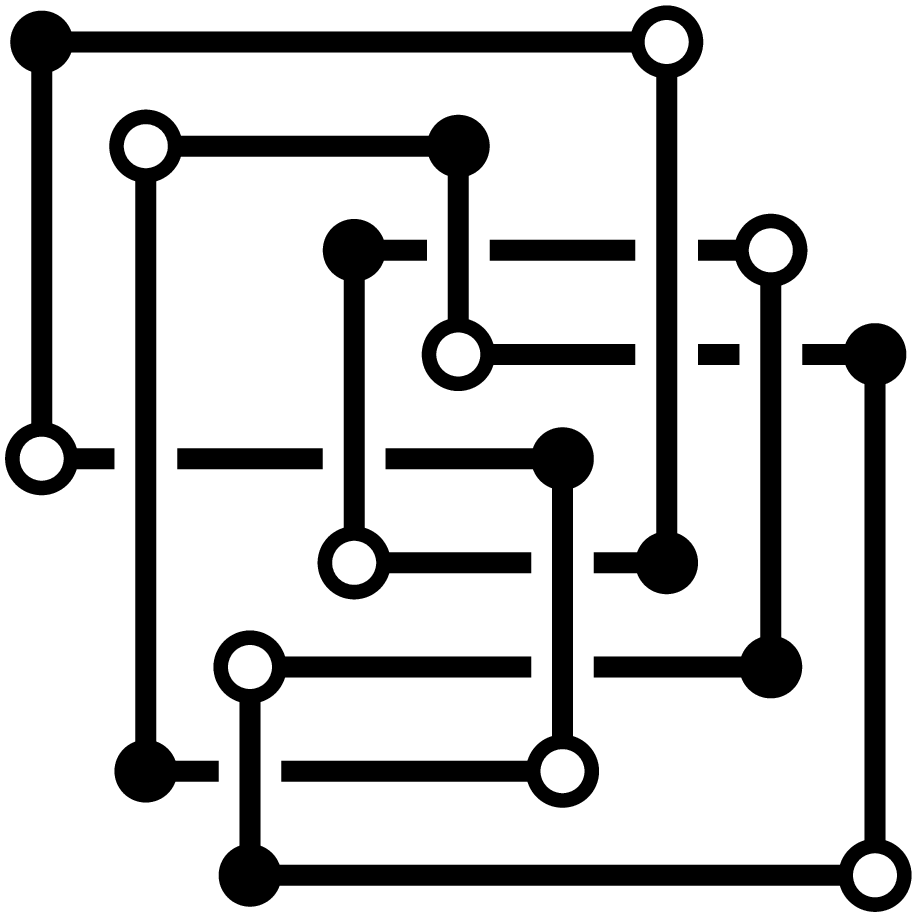}&\kern-.6em\raisebox{12pt}{$\rightarrow$}\kern-.6em&
\includegraphics[width=30pt]{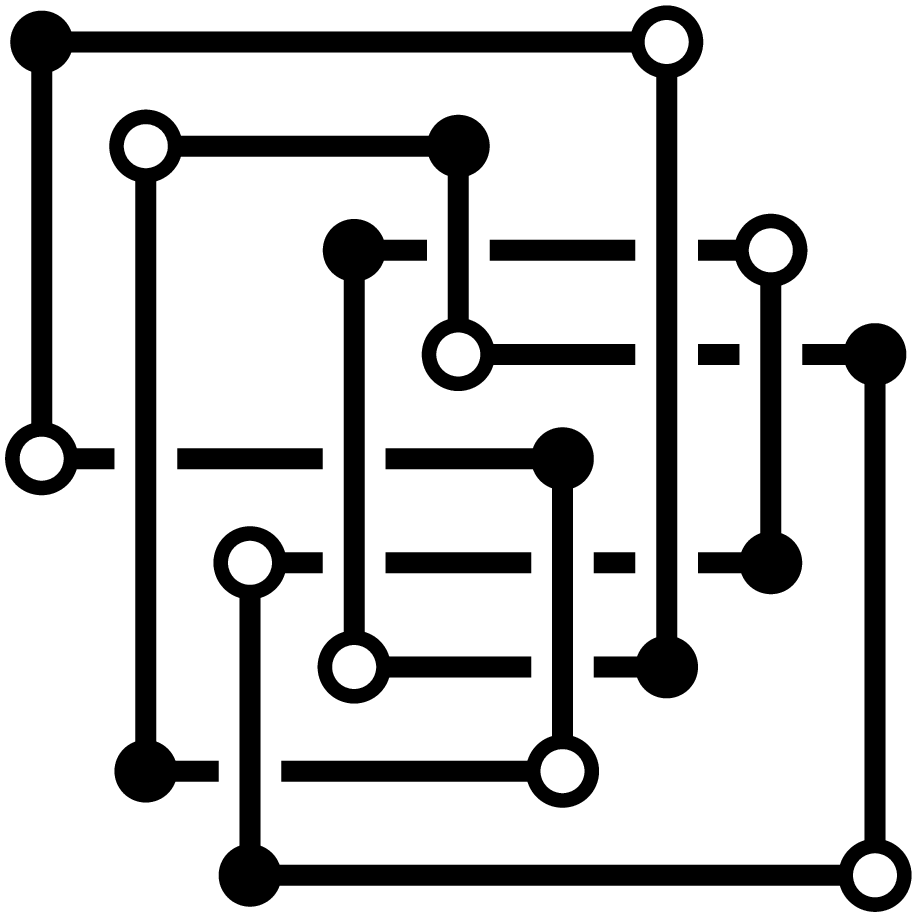}\\
&&&&&&&&&&&&&&&&\downarrow\\
&&
\includegraphics[width=30pt]{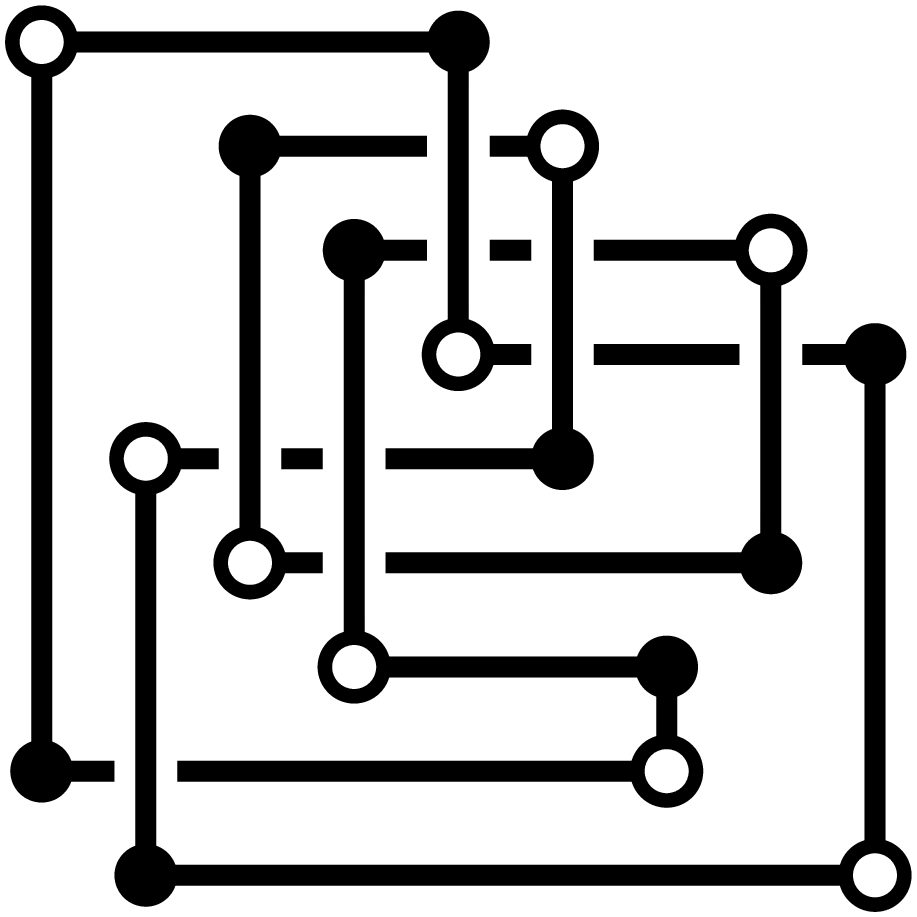}&\kern-.6em\raisebox{12pt}{$\leftarrow$}\kern-.6em&
\includegraphics[width=30pt]{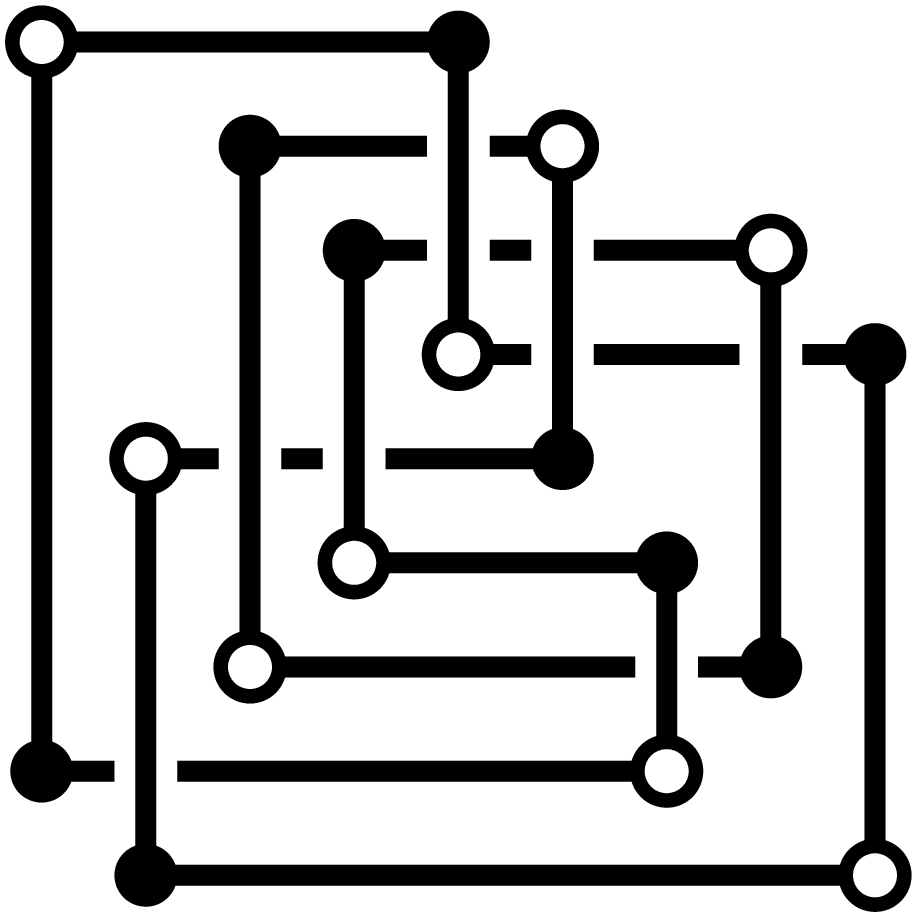}&\kern-.6em\raisebox{12pt}{$\leftarrow$}\kern-.6em&
\includegraphics[width=30pt]{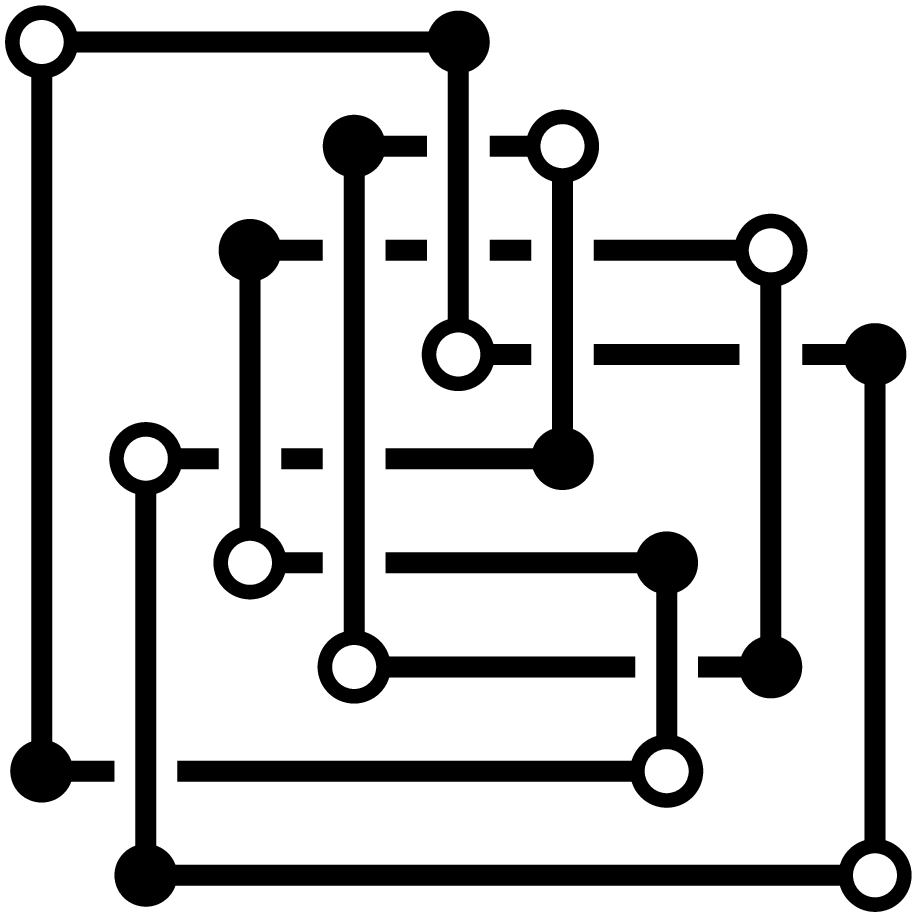}&\kern-.6em\raisebox{12pt}{$\leftarrow$}\kern-.6em&
\includegraphics[width=30pt]{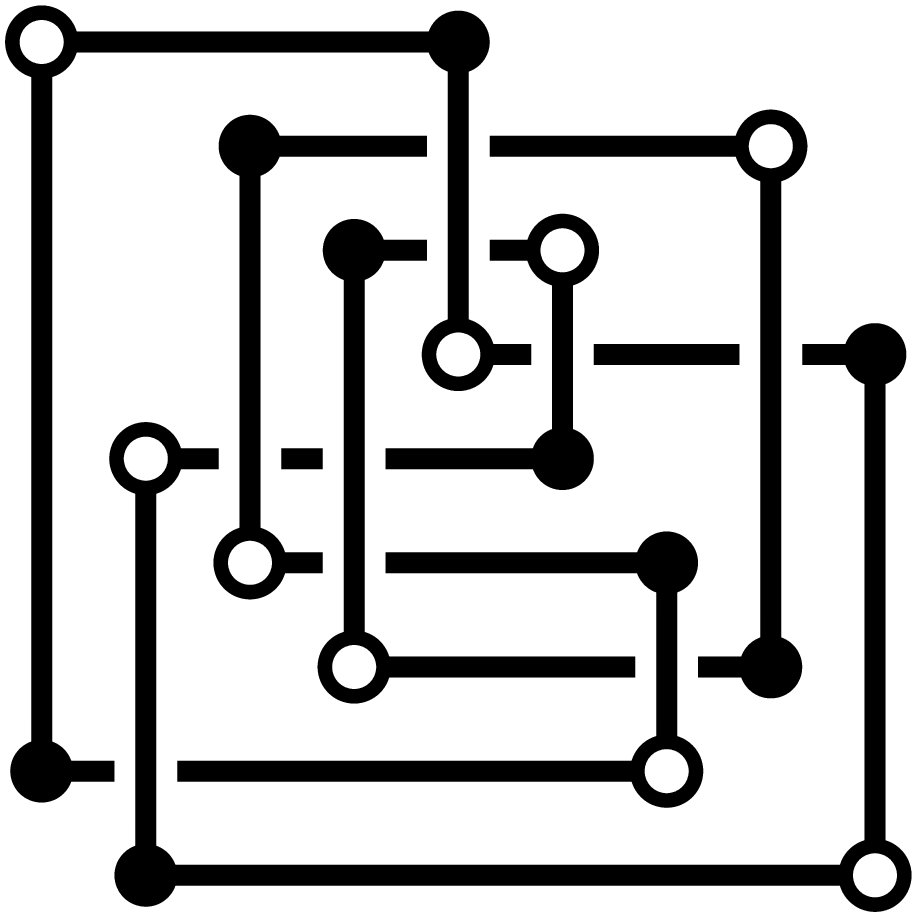}&\kern-.6em\raisebox{12pt}{$\leftarrow$}\kern-.6em&
\includegraphics[width=30pt]{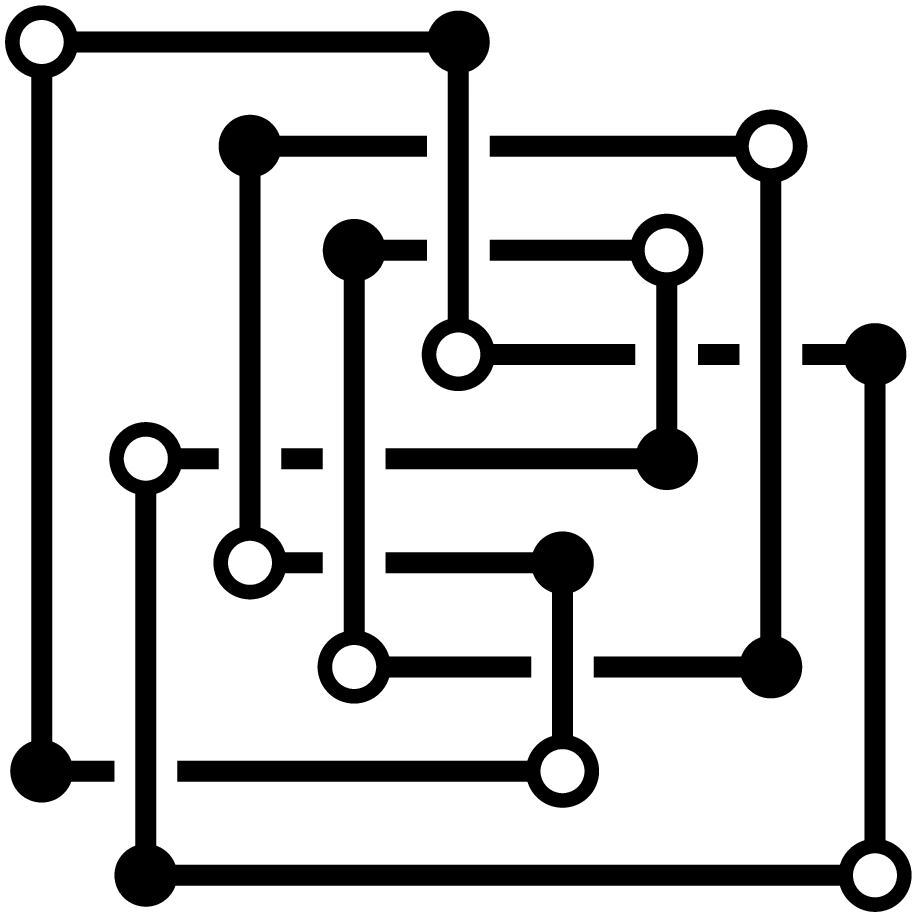}&\kern-.6em\raisebox{12pt}{$\leftarrow$}\kern-.6em&
\includegraphics[width=30pt]{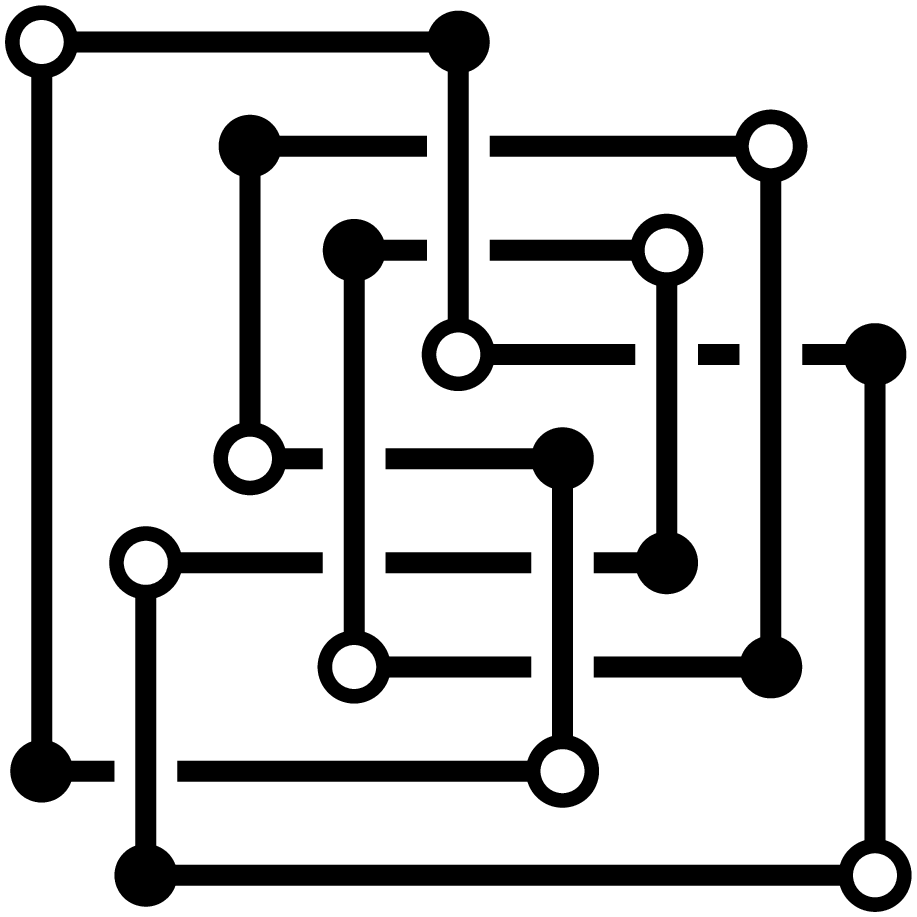}&\kern-.6em\raisebox{12pt}{$\leftarrow$}\kern-.6em&
\includegraphics[width=30pt]{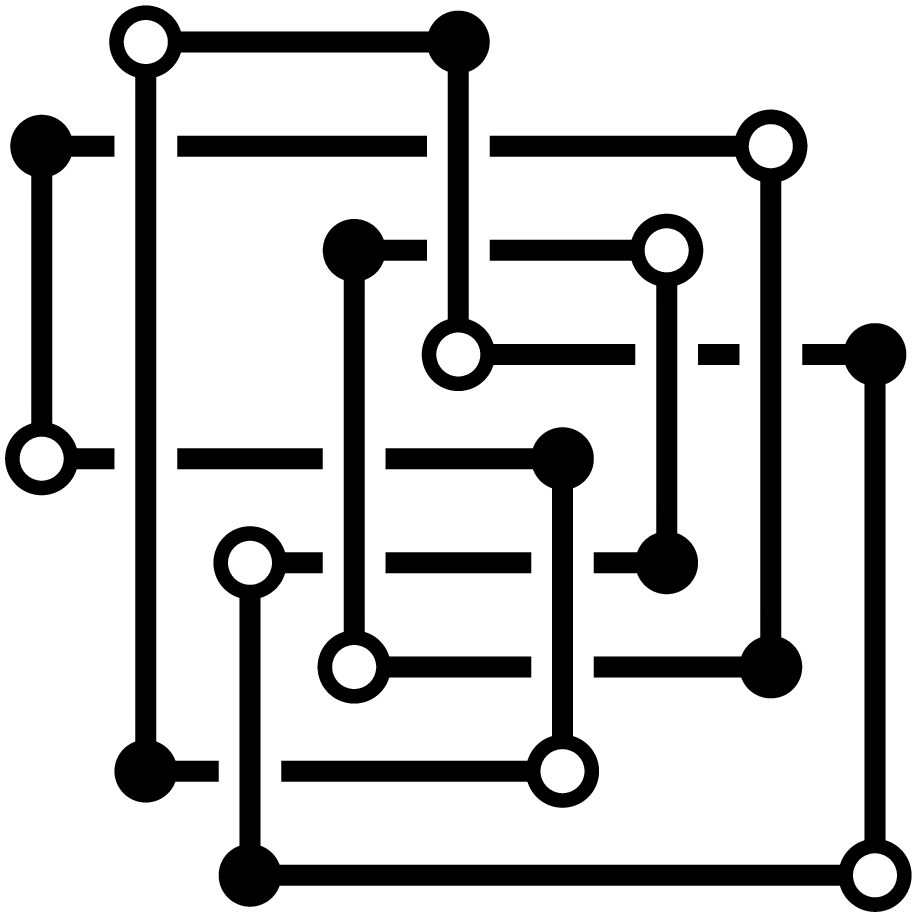}&\kern-.6em\raisebox{12pt}{$\leftarrow$}\kern-.6em&
\includegraphics[width=30pt]{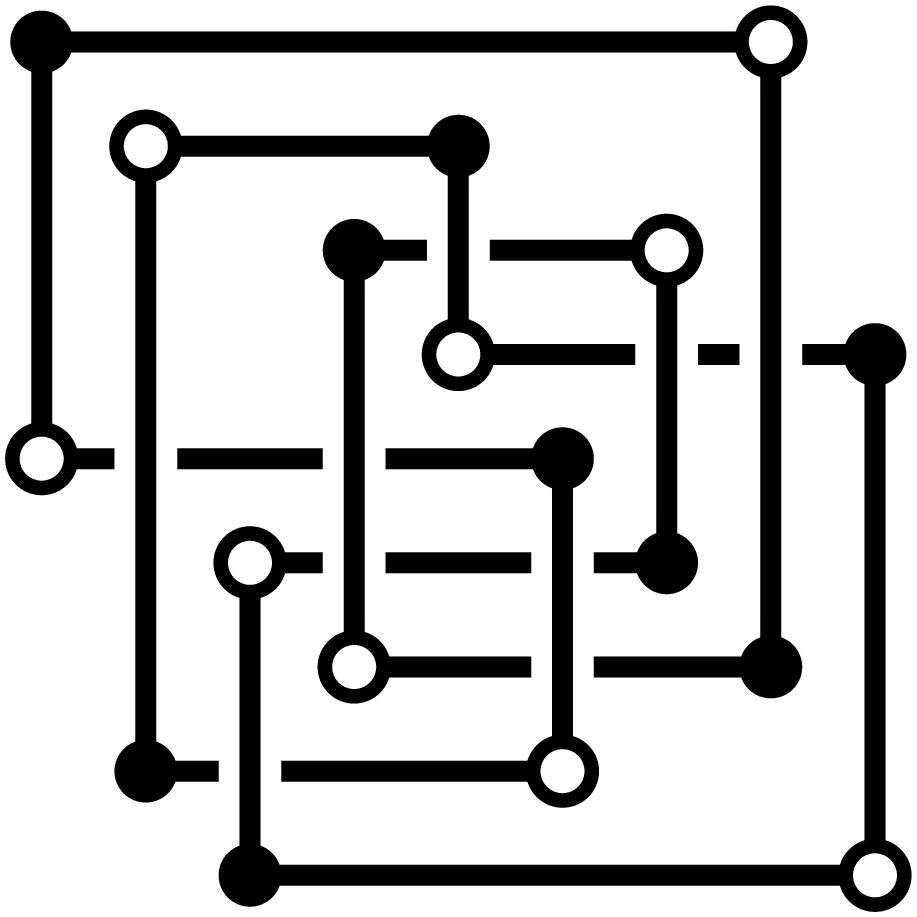}\end{array},$$
and
$S_{\overrightarrow{\mathrm{II}}}(-9_{42}^{\mathrm R})=S_{\overrightarrow{\mathrm{II}}}(\mu(9_{42}^{\mathrm R}))$:
$$\begin{array}{ccccccccc}
\includegraphics[width=30pt]{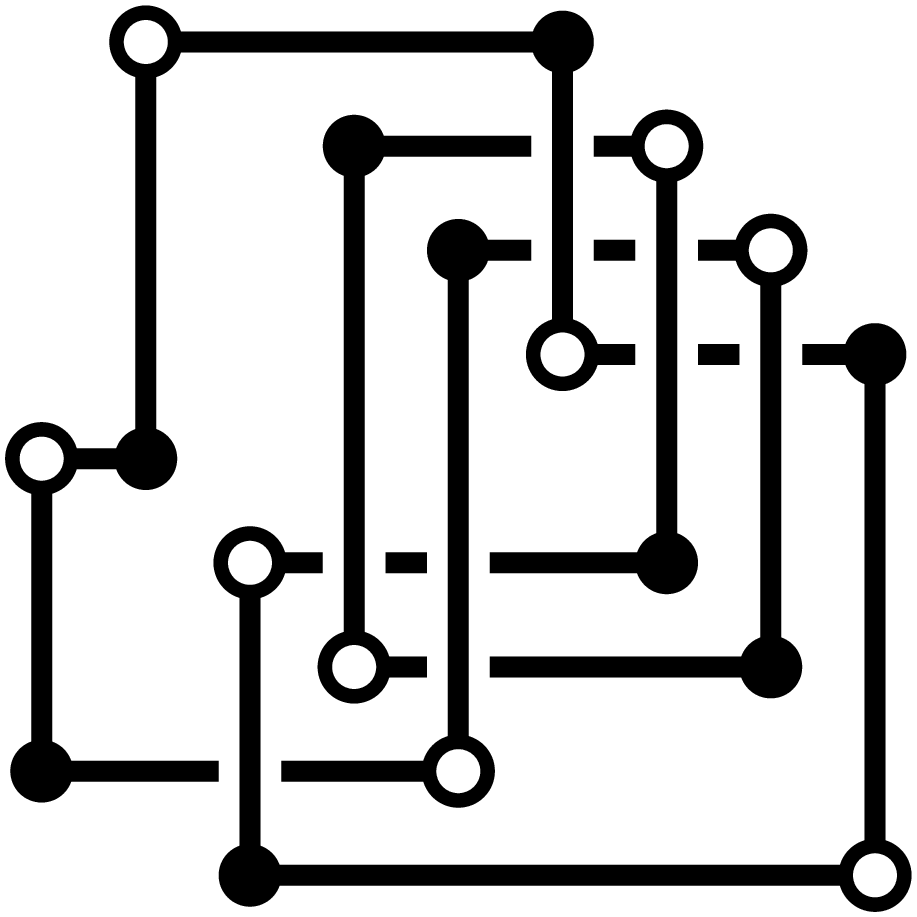}&\kern-.6em\raisebox{12pt}{$\rightarrow$}\kern-.6em&
\includegraphics[width=30pt]{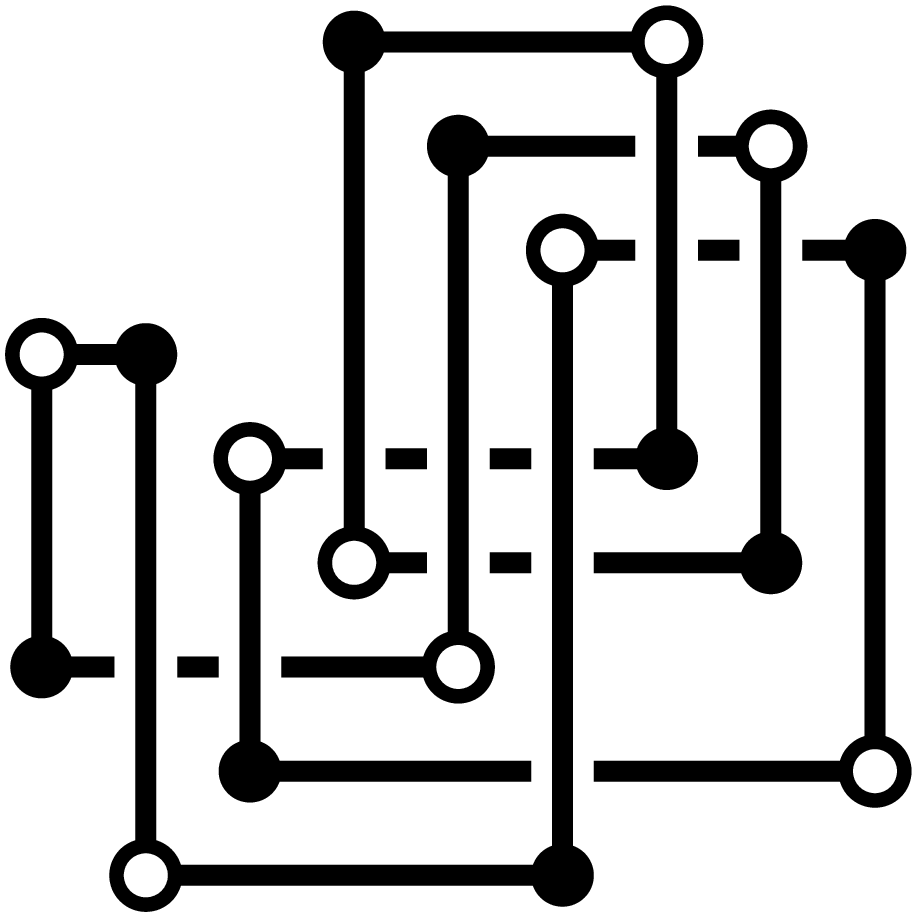}&\kern-.6em\raisebox{12pt}{$\rightarrow$}\kern-.6em&
\includegraphics[width=30pt]{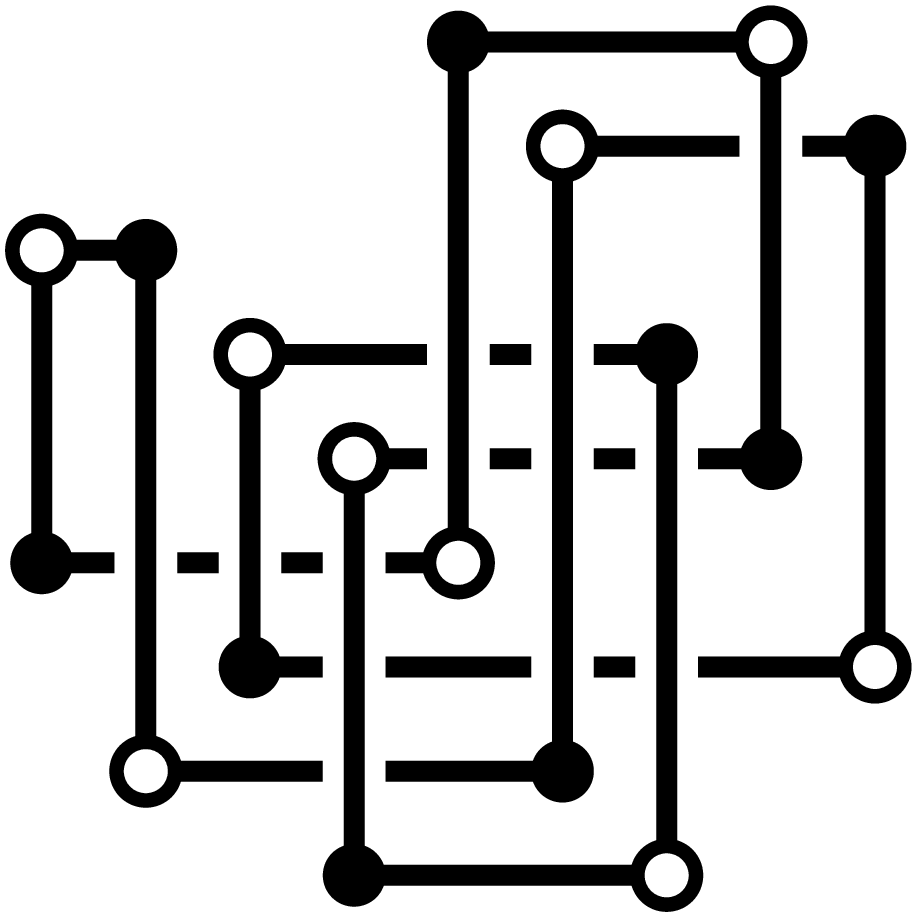}&\kern-.6em\raisebox{12pt}{$\rightarrow$}\kern-.6em&
\includegraphics[width=30pt]{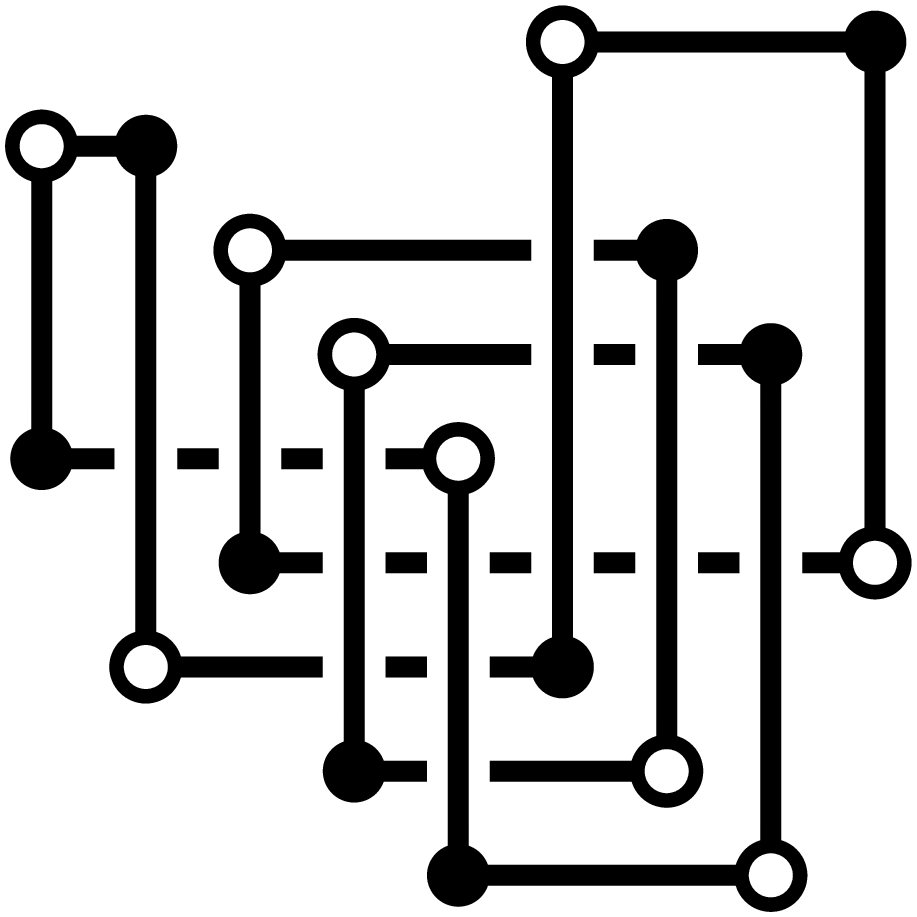}&\kern-.6em\raisebox{12pt}{$\rightarrow$}\kern-.6em&
\includegraphics[width=30pt]{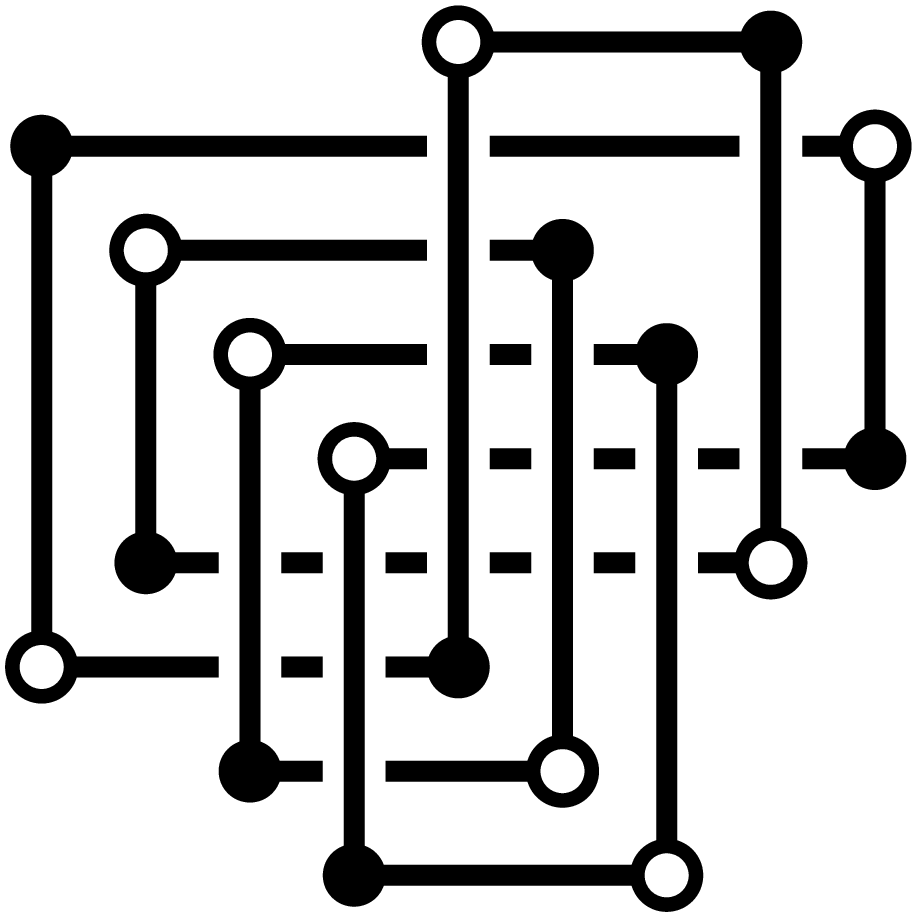}\\
&&&&&&&&\downarrow\\
\includegraphics[width=30pt]{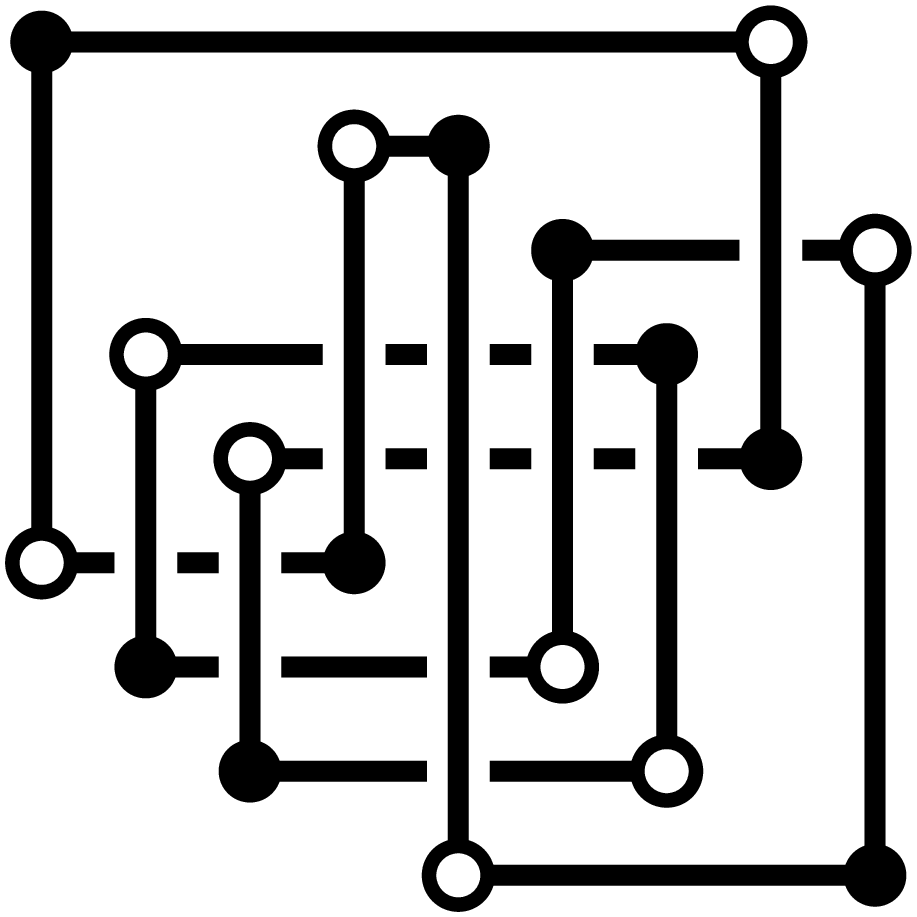}&\kern-.6em\raisebox{12pt}{$\leftarrow$}\kern-.6em&
\includegraphics[width=30pt]{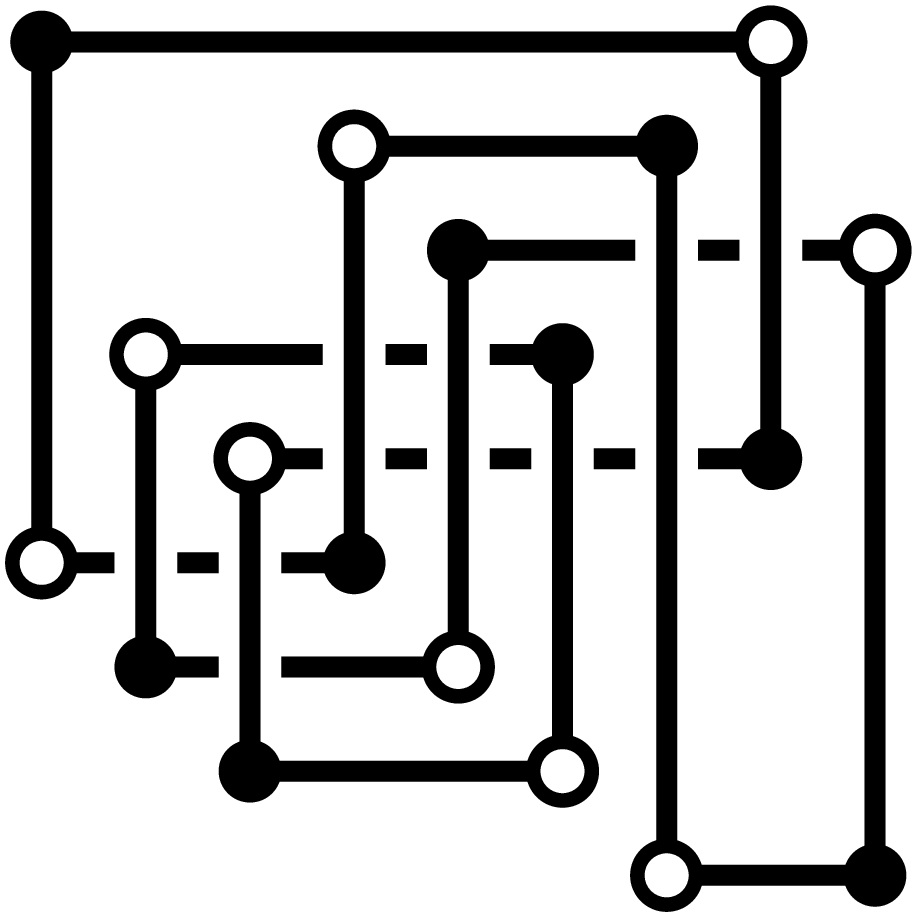}&\kern-.6em\raisebox{12pt}{$\leftarrow$}\kern-.6em&
\includegraphics[width=30pt]{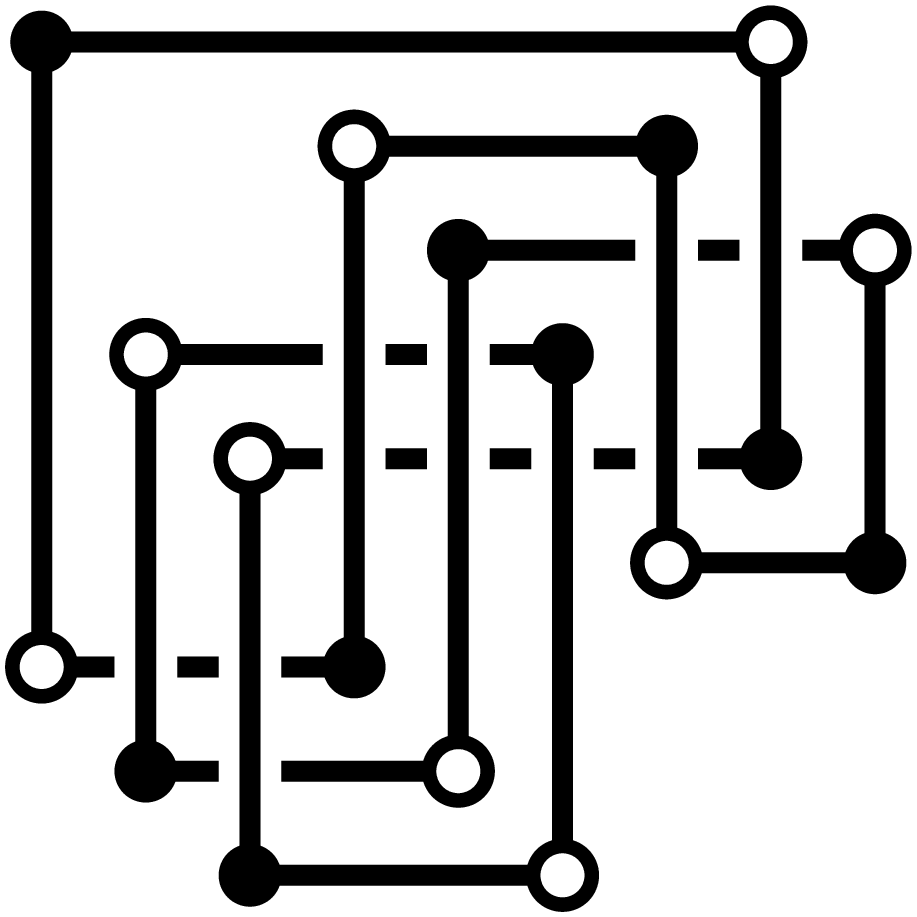}&\kern-.6em\raisebox{12pt}{$\leftarrow$}\kern-.6em&
\includegraphics[width=30pt]{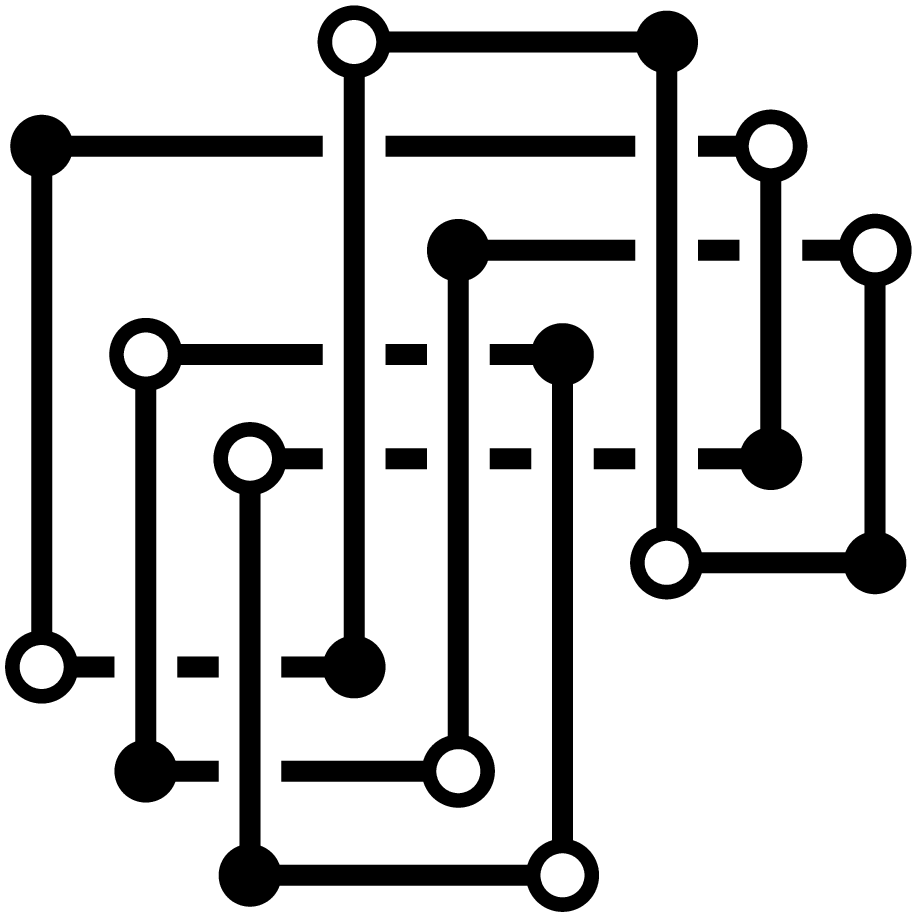}&\kern-.6em\raisebox{12pt}{$\leftarrow$}\kern-.6em&
\includegraphics[width=30pt]{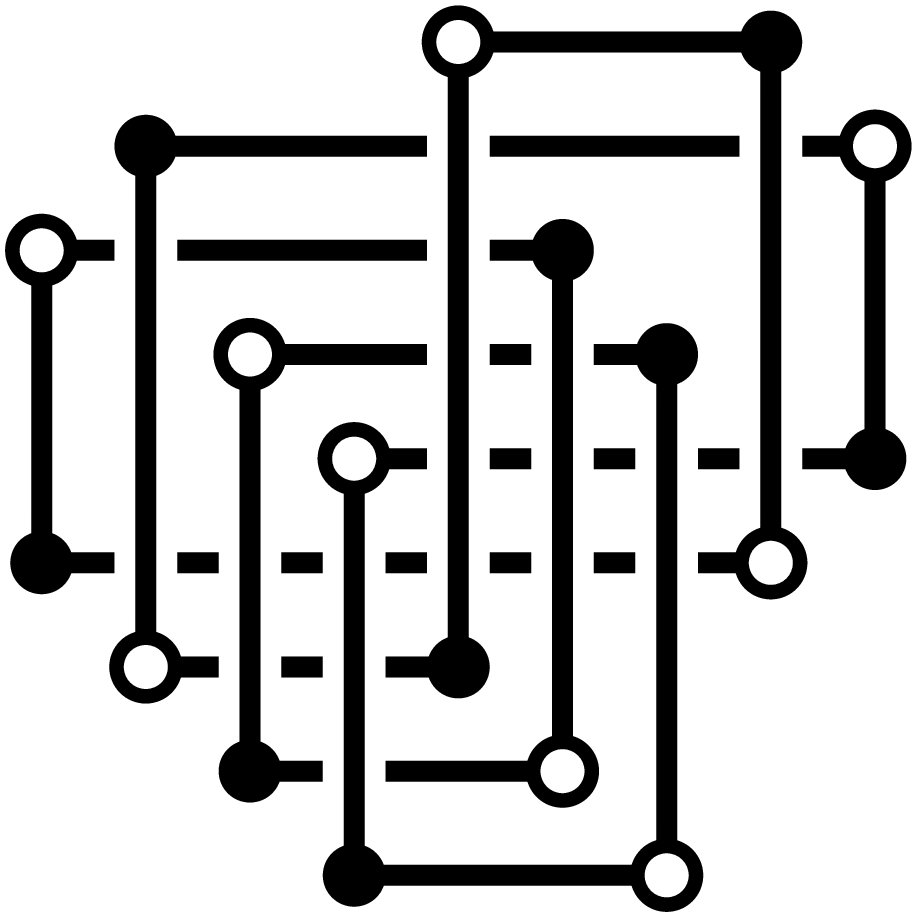}
\end{array}.$$
By Theorem~\ref{rect-desc-of-leg-theo} this implies
\begin{equation}\label{9-42-eq}
\mathscr L_+(9_{42}^{\mathrm R})=\mathscr L_+(-9_{42}^{\mathrm R})\quad\text{and}\quad
\mathscr L_-(-9_{42}^{\mathrm R})=\mathscr L_-(\mu(9_{42}^{\mathrm R})).
\end{equation}
From Proposition~\ref{table-knots-prop} and Theoremd~\ref{main-theo}
we conclude:
$$\eqref{9-42-ne-eq}\ \&\ \eqref{9-42-eq}
\quad\Rightarrow\quad
\mathscr L_-(9_{42}^{\mathrm R})\ne\mathscr L_-(-9_{42}^{\mathrm R})\quad\text{and}\quad
\mathscr L_+(-9_{42}^{\mathrm R})\ne\mathscr L_+(\mu(9_{42}^{\mathrm R})).$$

The front projections in Figure~\ref{9_42_fig} are obtained as described
in Section~\ref{rd_of_knots-sec} (see Figure~\ref{associated-legendrian-fig})
from rectangular diagrams that can be easily guessed from the pictures.
Using Theorem~\ref{rect-desc-of-leg-theo} we find that~$9_{42}^+=\mathscr L_+(9_{42}^{\mathrm R})$:
$$\begin{array}{ccccccccccccccc}
\includegraphics[width=30pt]{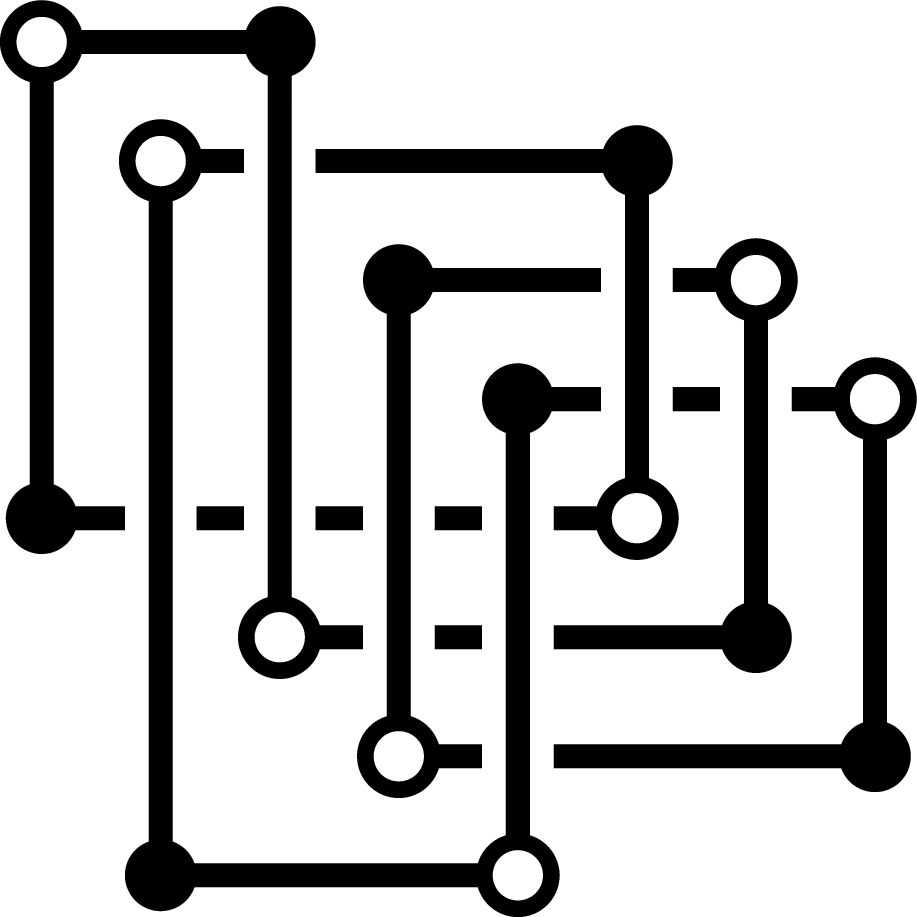}&\kern-.6em\raisebox{12pt}{$\rightarrow$}\kern-.6em&
\includegraphics[width=30pt]{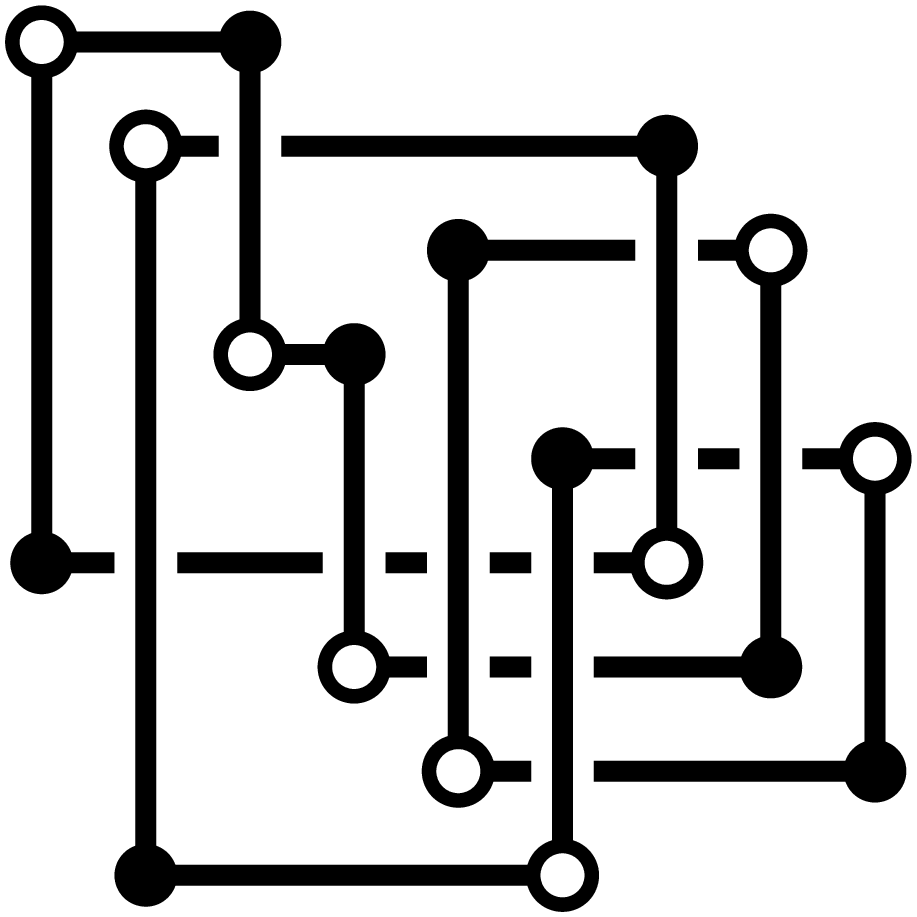}&\kern-.6em\raisebox{12pt}{$\rightarrow$}\kern-.6em&
\includegraphics[width=30pt]{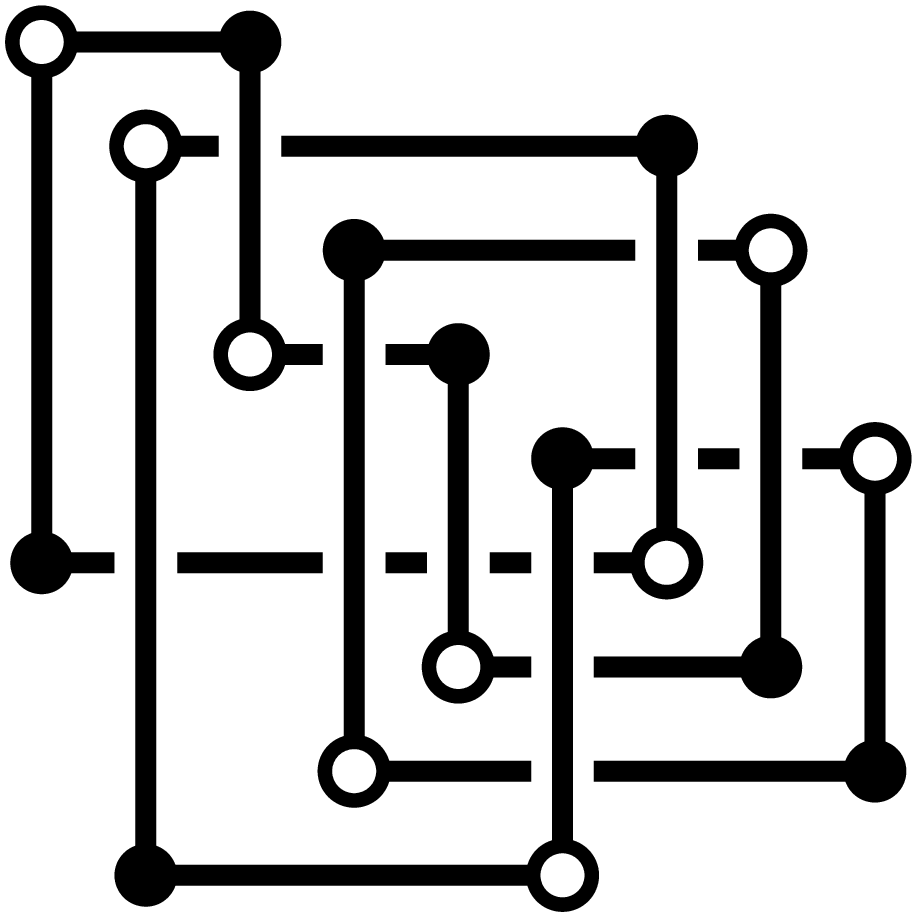}&\kern-.6em\raisebox{12pt}{$\rightarrow$}\kern-.6em&
\includegraphics[width=30pt]{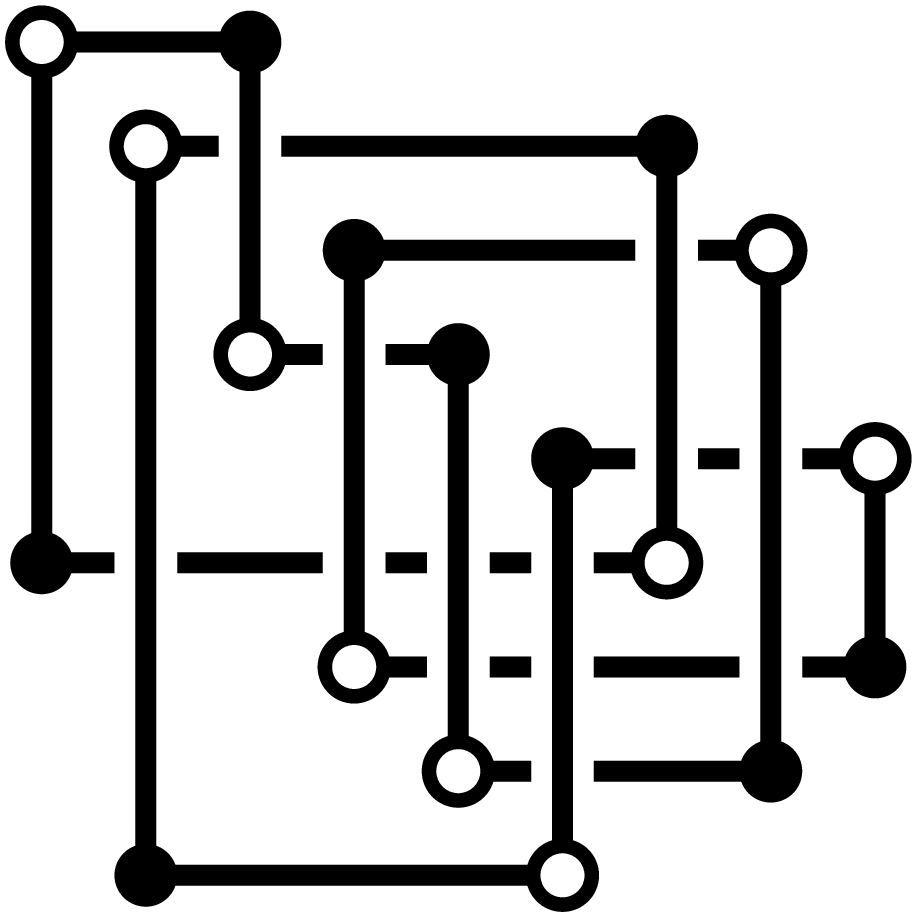}&\kern-.6em\raisebox{12pt}{$\rightarrow$}\kern-.6em&
\includegraphics[width=30pt]{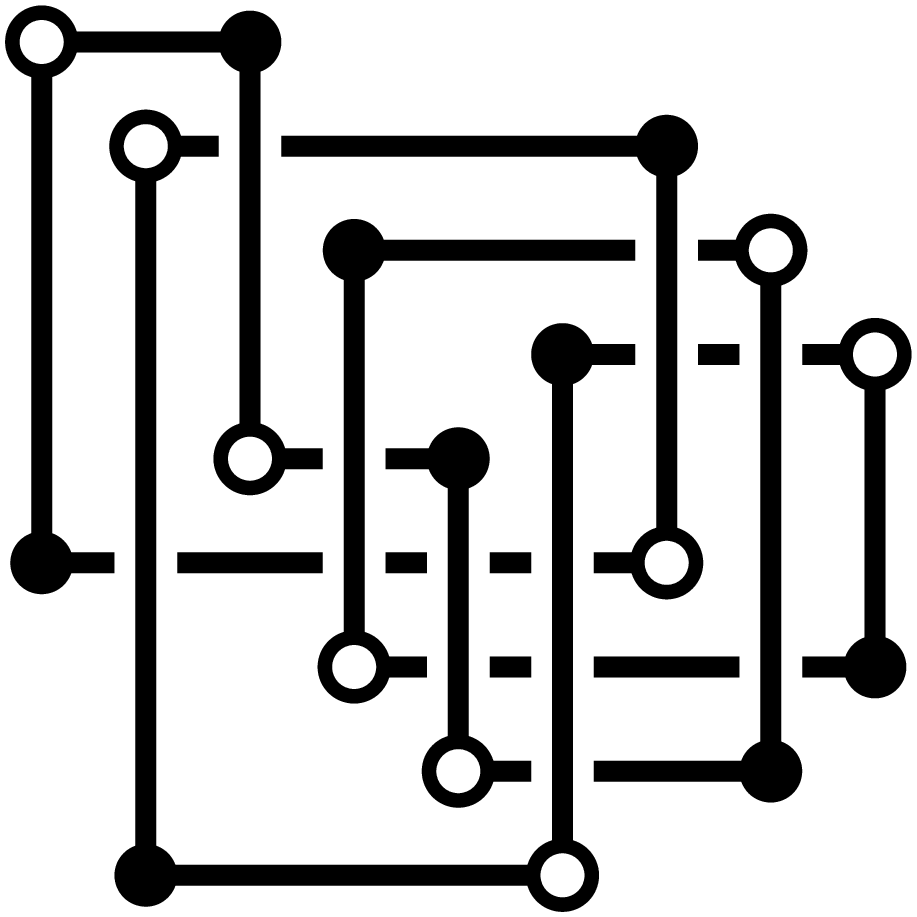}&\kern-.6em\raisebox{12pt}{$\rightarrow$}\kern-.6em&
\includegraphics[width=30pt]{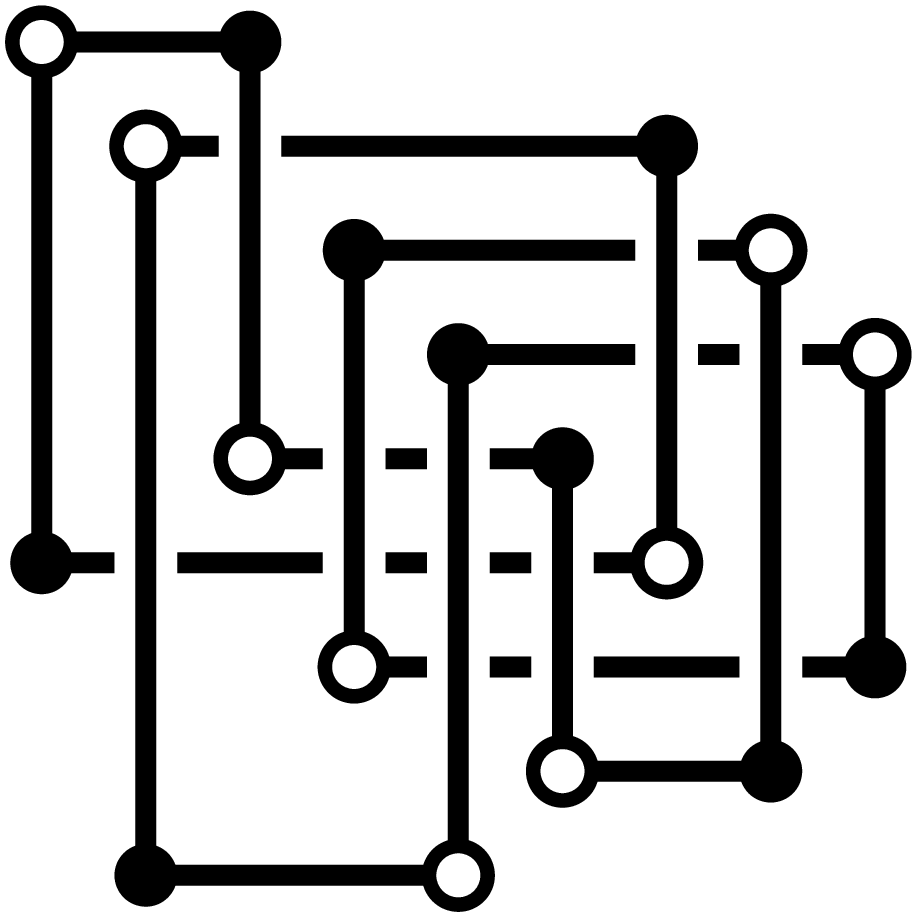}&\kern-.6em\raisebox{12pt}{$\rightarrow$}\kern-.6em&
\includegraphics[width=30pt]{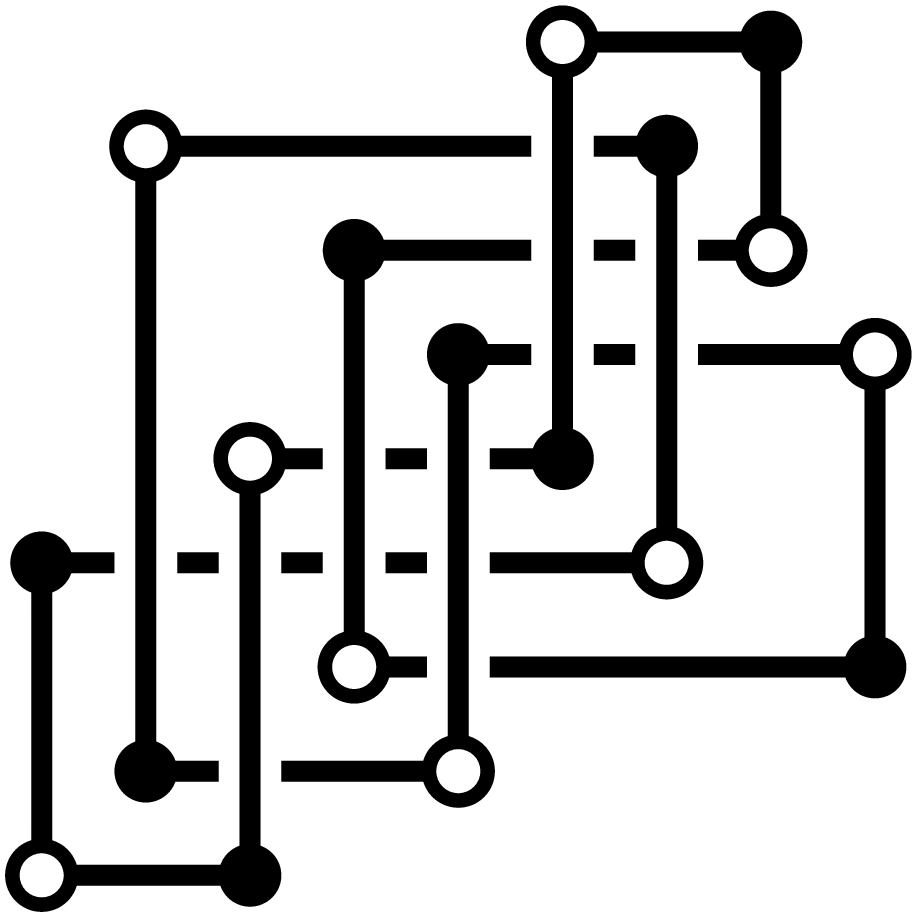}&\kern-.6em\raisebox{12pt}{$\rightarrow$}\kern-.6em&
\includegraphics[width=30pt]{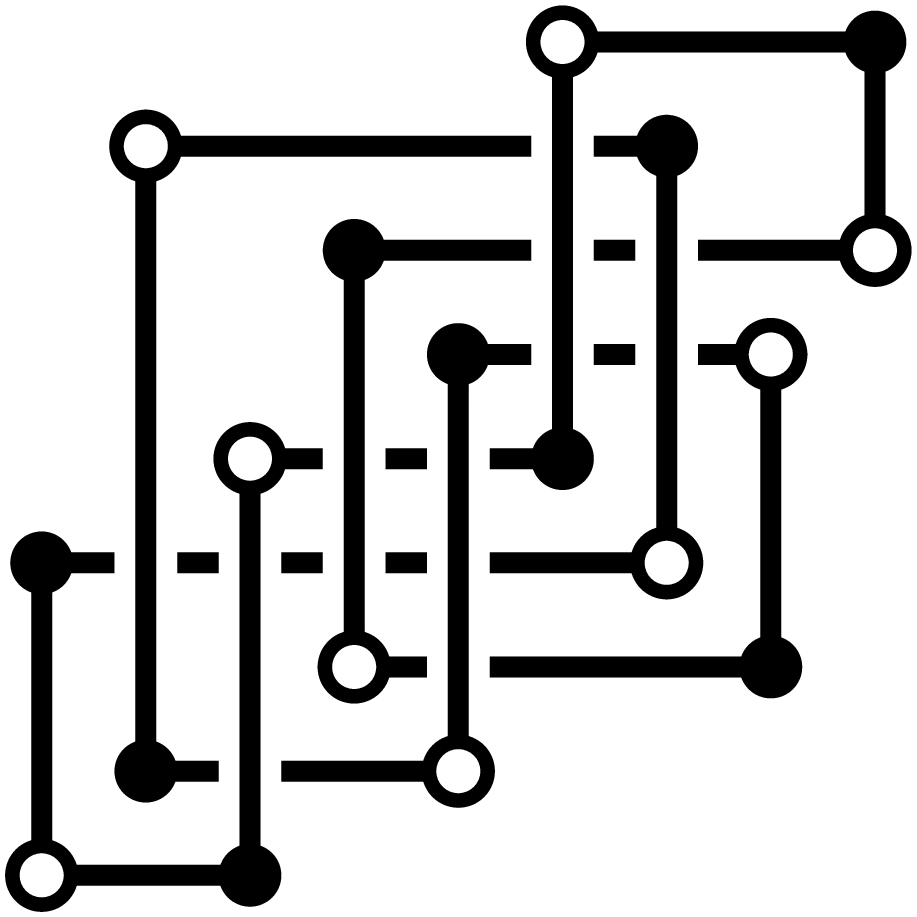}\\
&&&&&&&&&&&&&&\downarrow\\
\includegraphics[width=30pt]{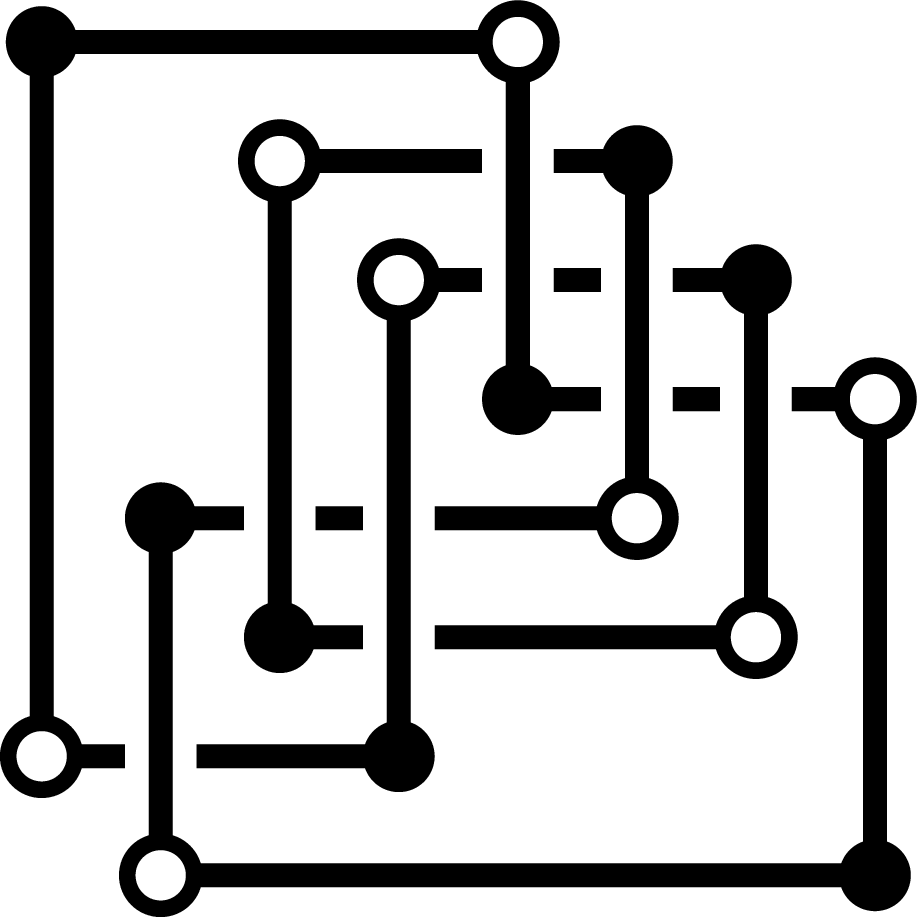}&\kern-.6em\raisebox{12pt}{$\leftarrow$}\kern-.6em&
\includegraphics[width=30pt]{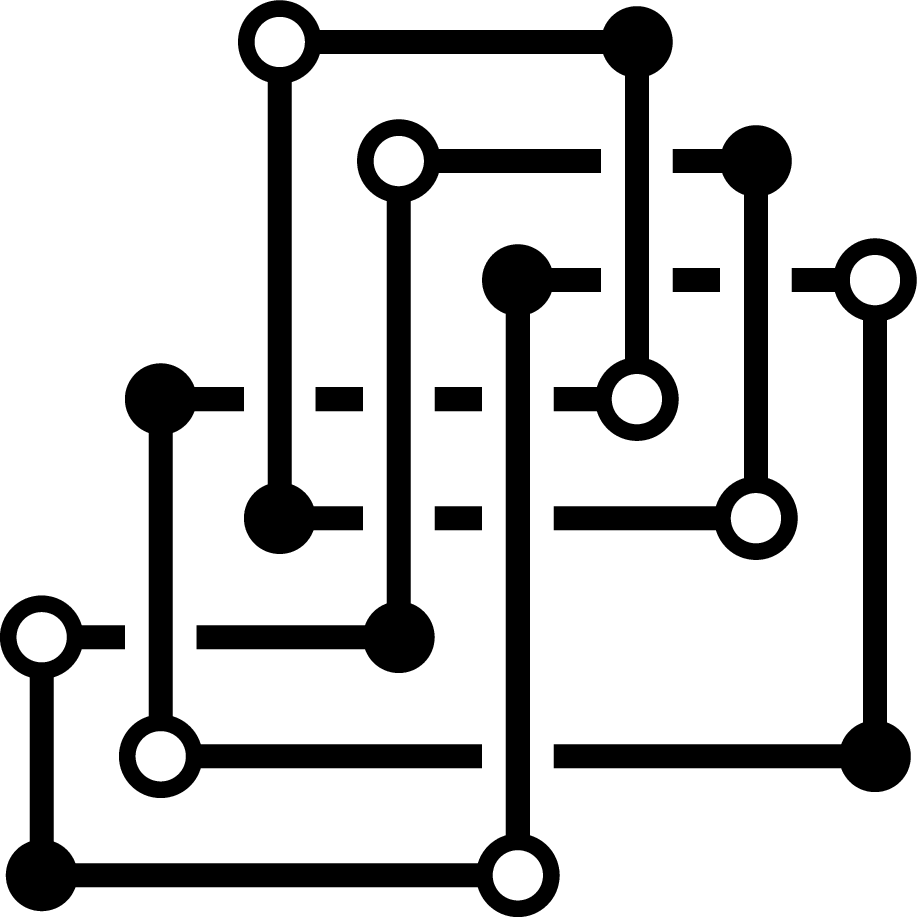}&\kern-.6em\raisebox{12pt}{$\leftarrow$}\kern-.6em&
\includegraphics[width=30pt]{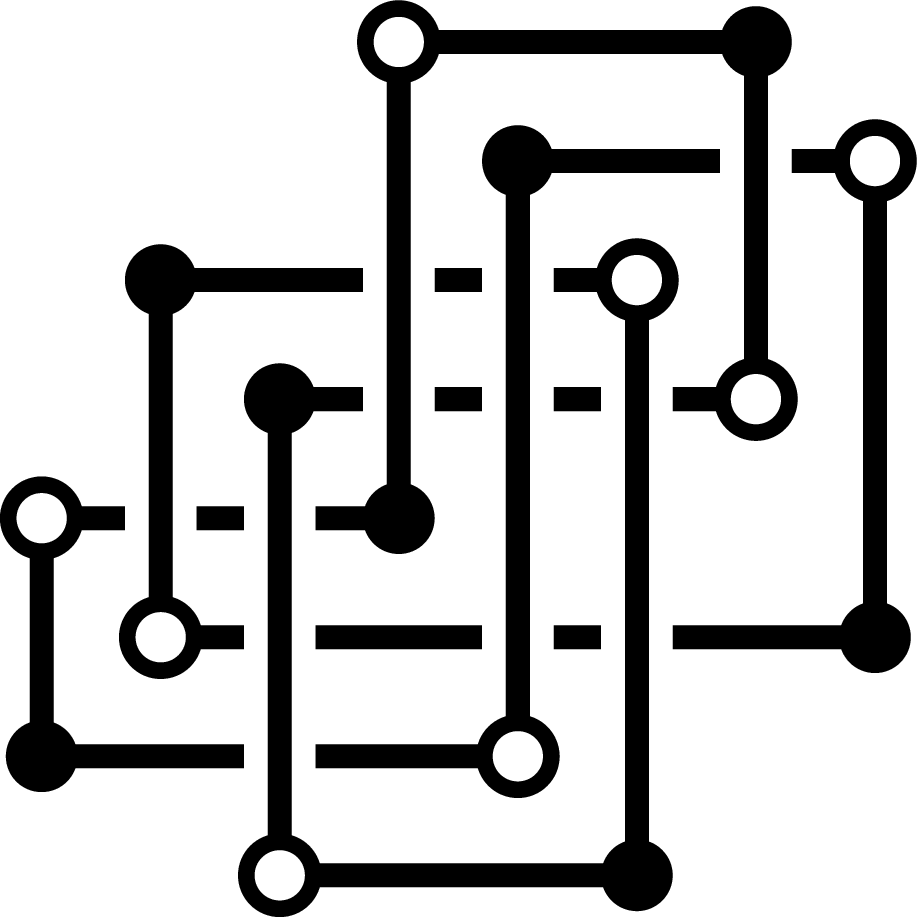}&\kern-.6em\raisebox{12pt}{$\leftarrow$}\kern-.6em&
\includegraphics[width=30pt]{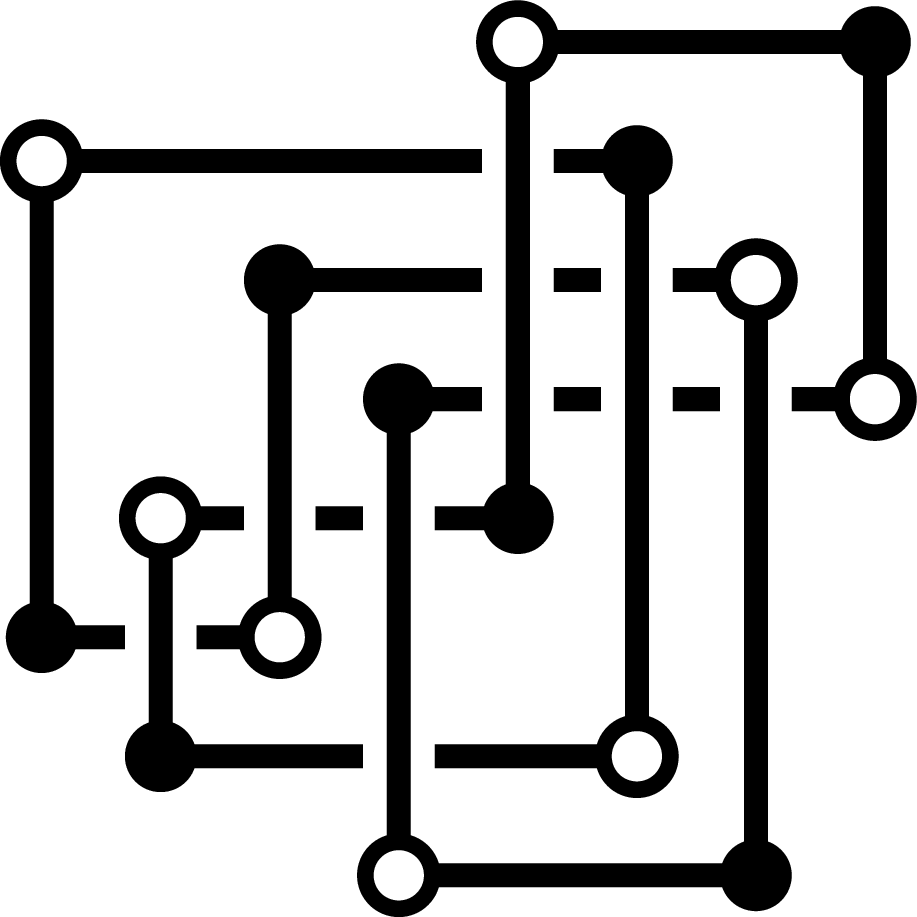}&\kern-.6em\raisebox{12pt}{$\leftarrow$}\kern-.6em&
\includegraphics[width=30pt]{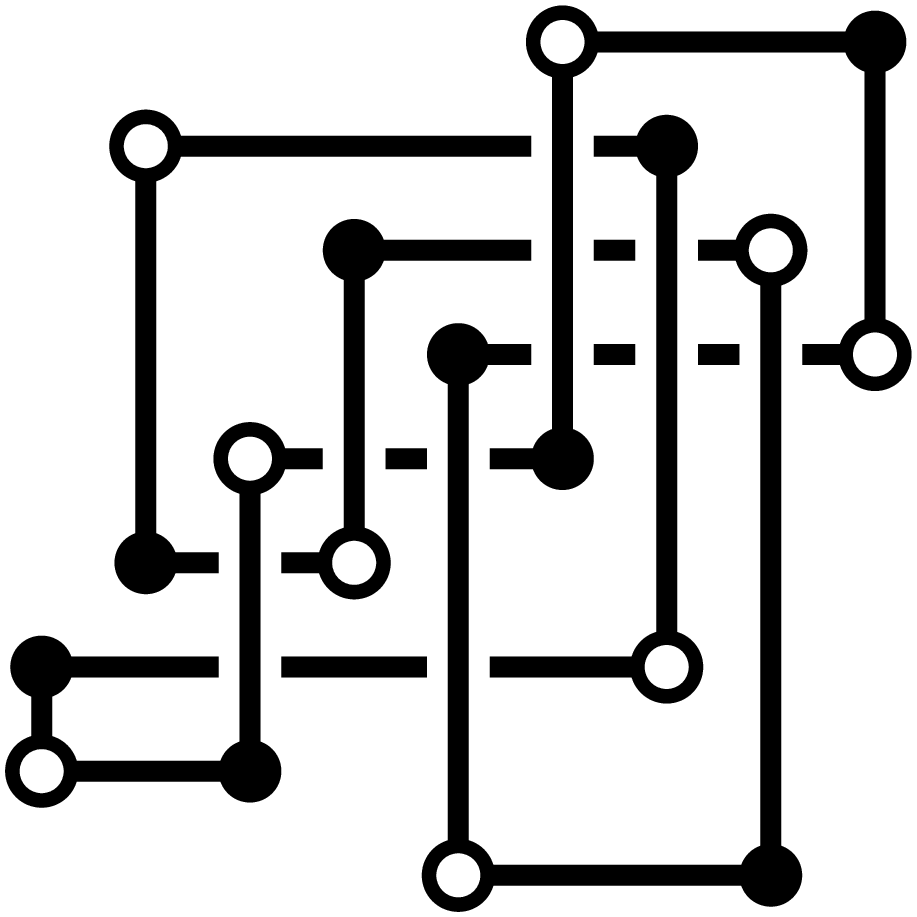}&\kern-.6em\raisebox{12pt}{$\leftarrow$}\kern-.6em&
\includegraphics[width=30pt]{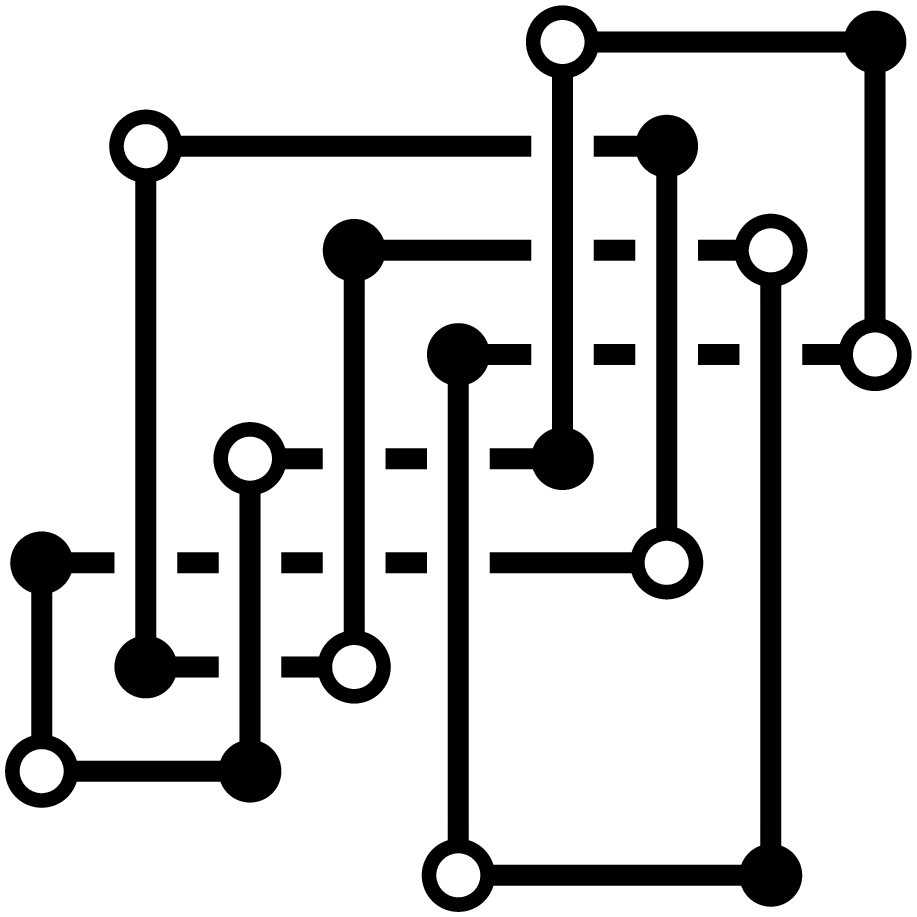}&\kern-.6em\raisebox{12pt}{$\leftarrow$}\kern-.6em&
\includegraphics[width=30pt]{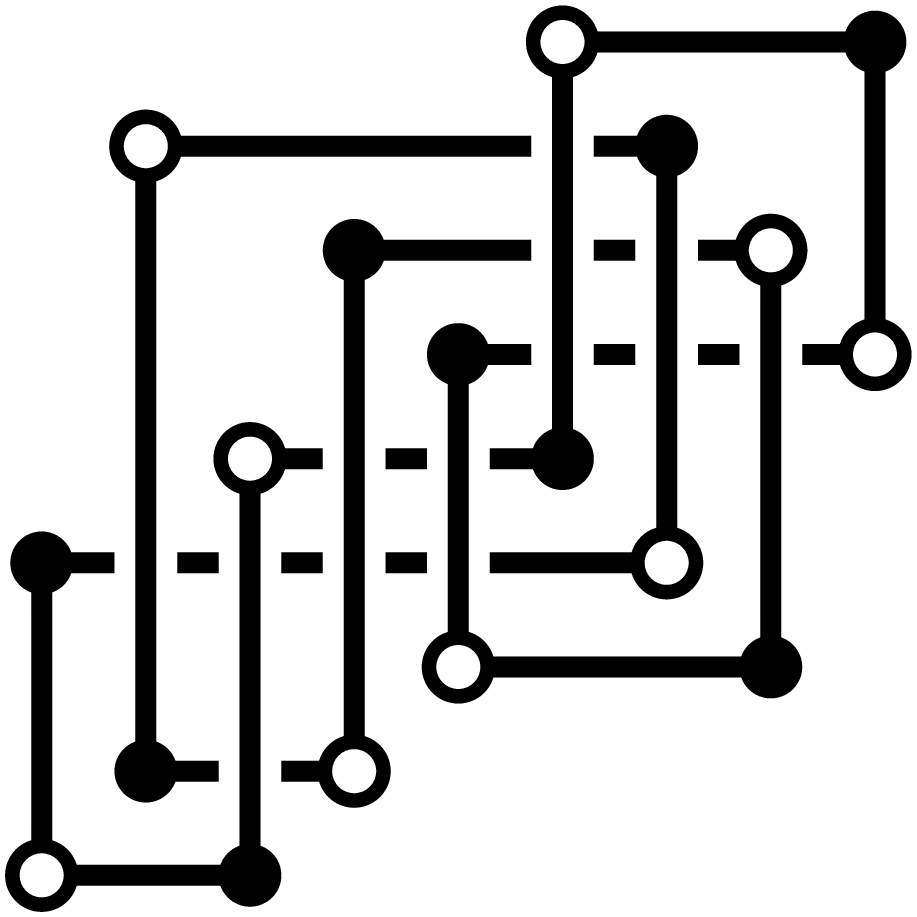}&\kern-.6em\raisebox{12pt}{$\leftarrow$}\kern-.6em&
\includegraphics[width=30pt]{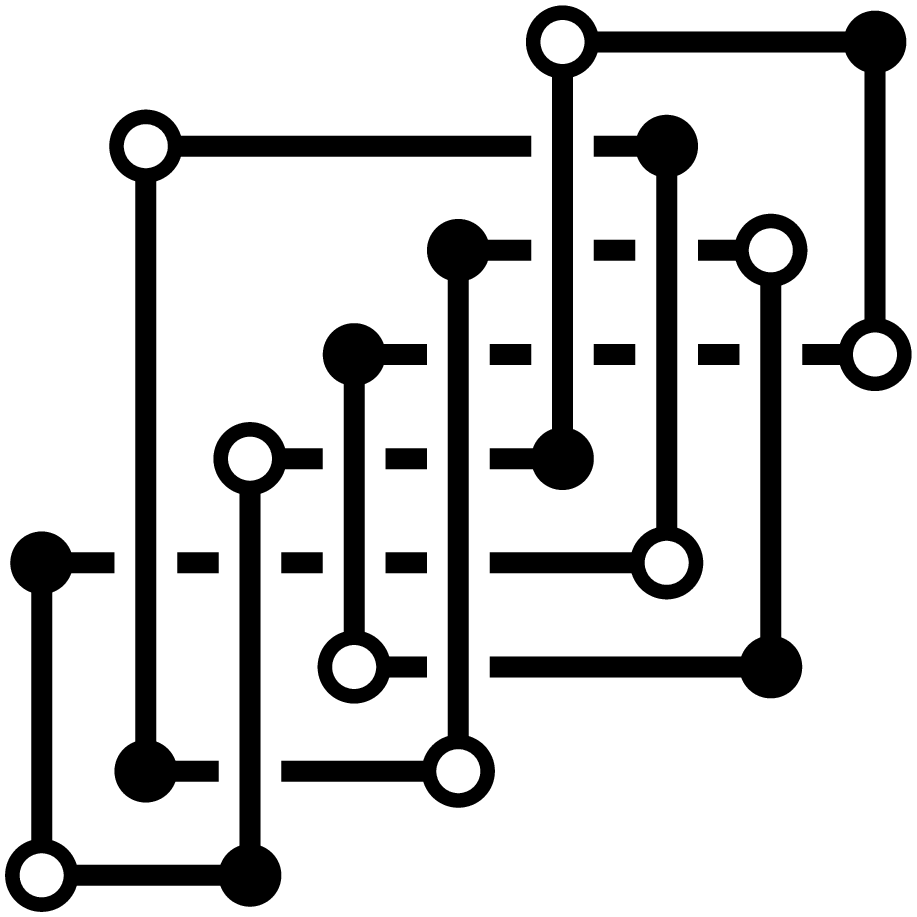}
\end{array},$$
and~$9_{42}^-=\mathscr L_-(9_{42}^{\mathrm R})$:
$$\begin{array}{ccccccccccccccccc}
\includegraphics[width=30pt]{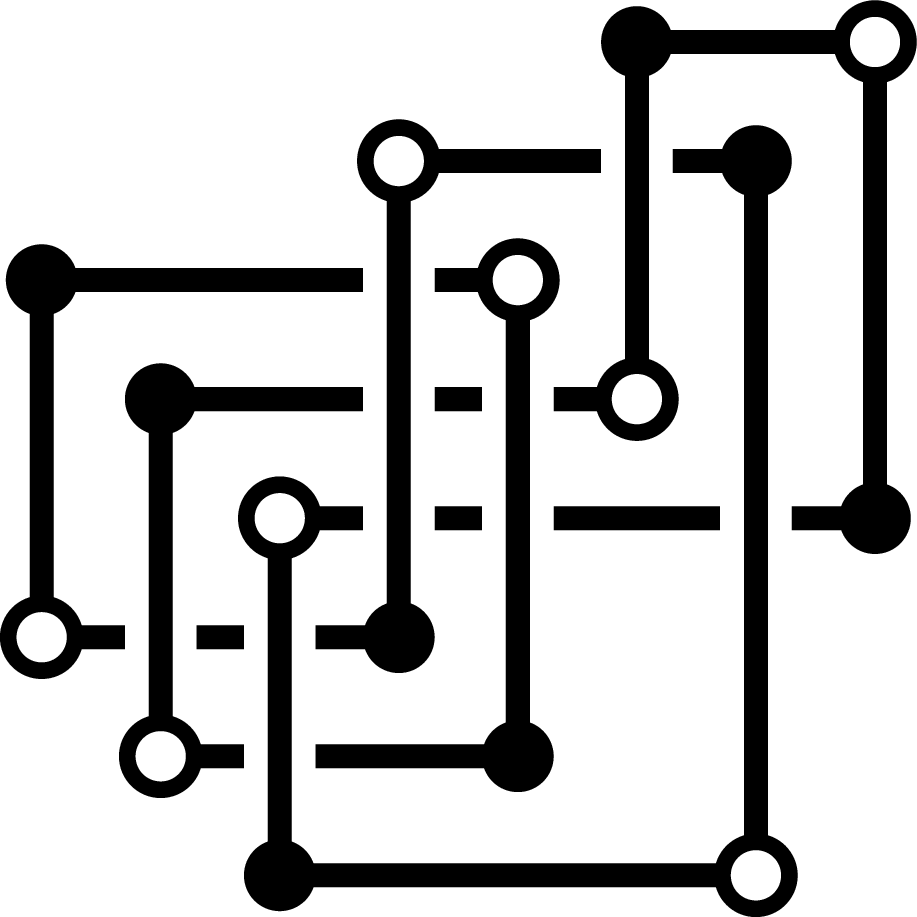}&\kern-.6em\raisebox{12pt}{$\rightarrow$}\kern-.6em&
\includegraphics[width=30pt]{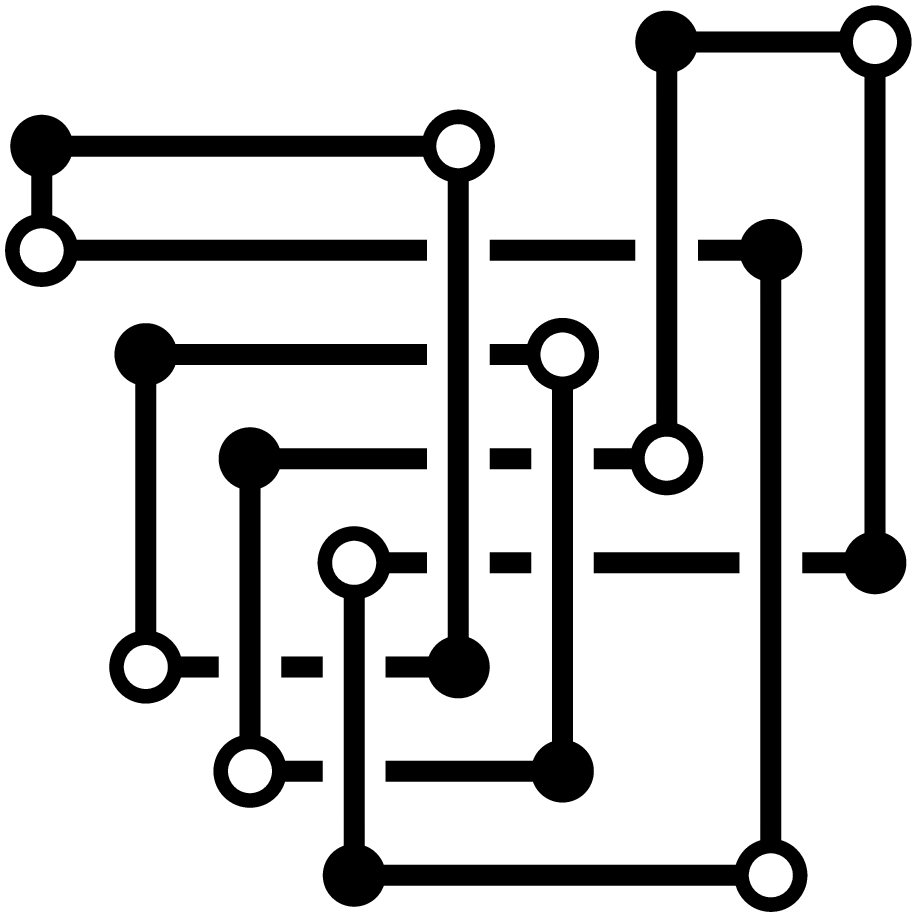}&\kern-.6em\raisebox{12pt}{$\rightarrow$}\kern-.6em&
\includegraphics[width=30pt]{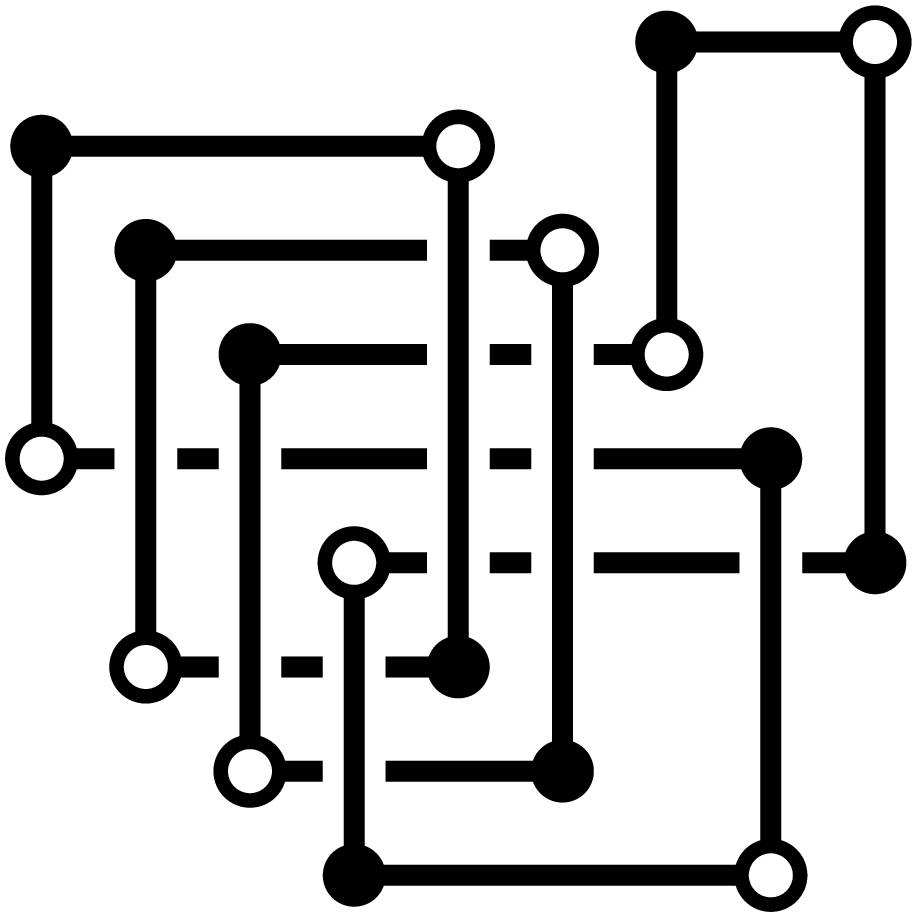}&\kern-.6em\raisebox{12pt}{$\rightarrow$}\kern-.6em&
\includegraphics[width=30pt]{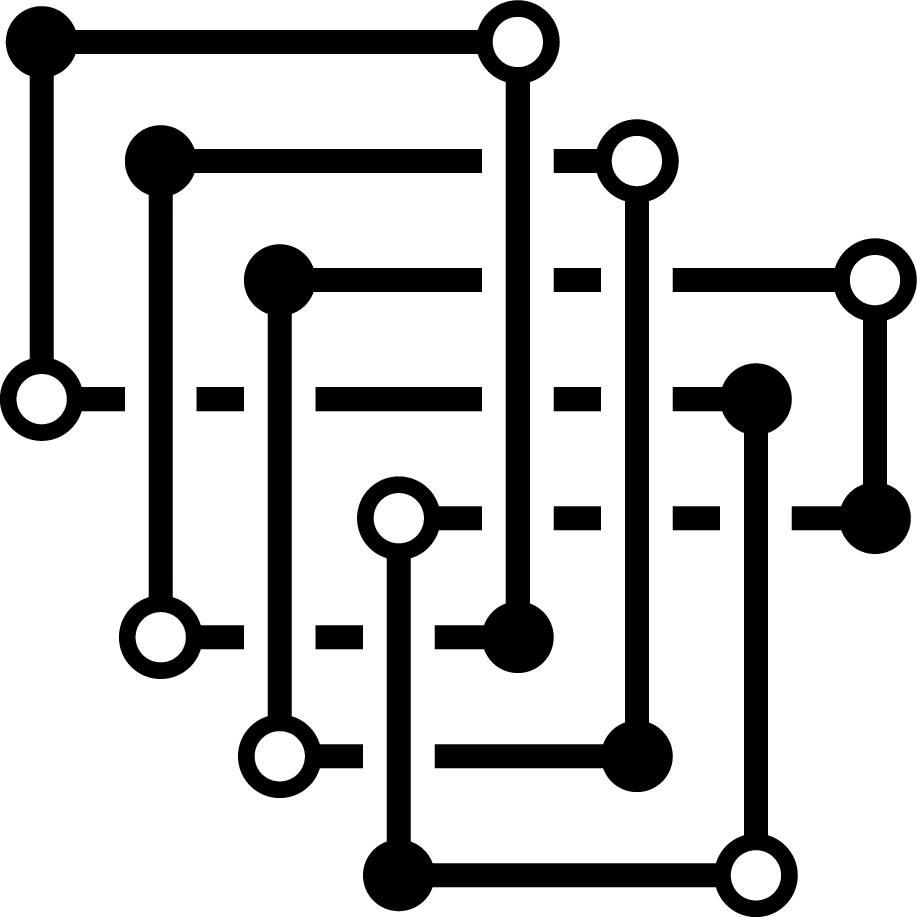}&\kern-.6em\raisebox{12pt}{$\rightarrow$}\kern-.6em&
\includegraphics[width=30pt]{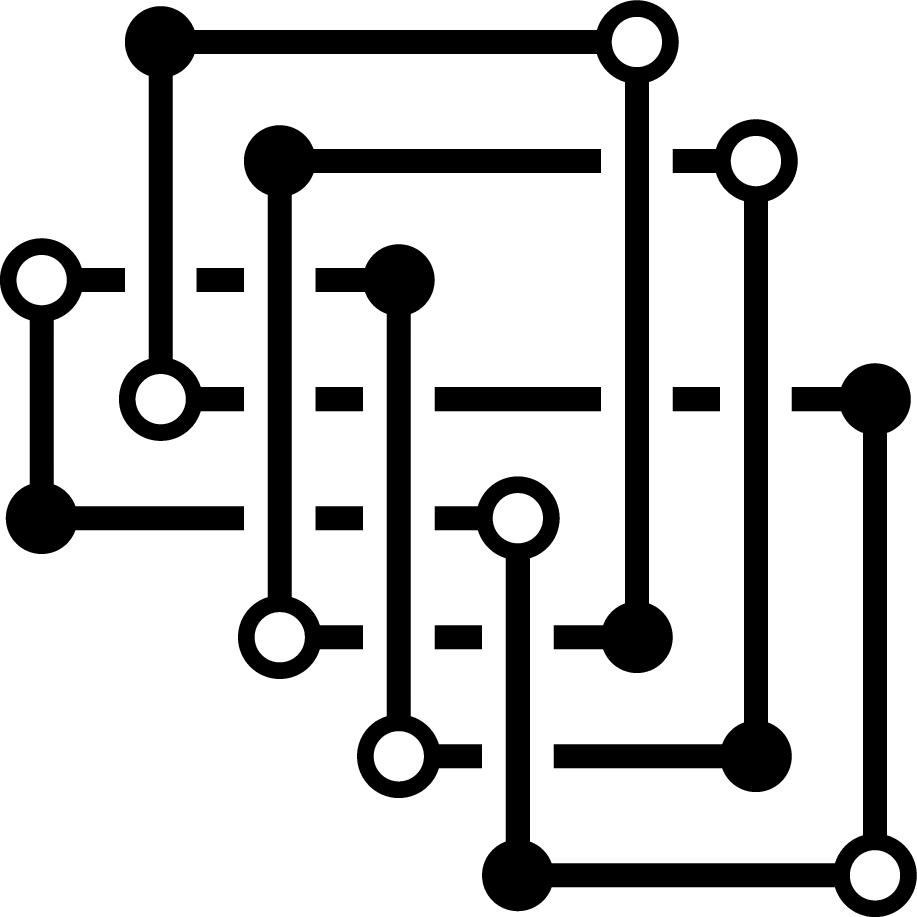}&\kern-.6em\raisebox{12pt}{$\rightarrow$}\kern-.6em&
\includegraphics[width=30pt]{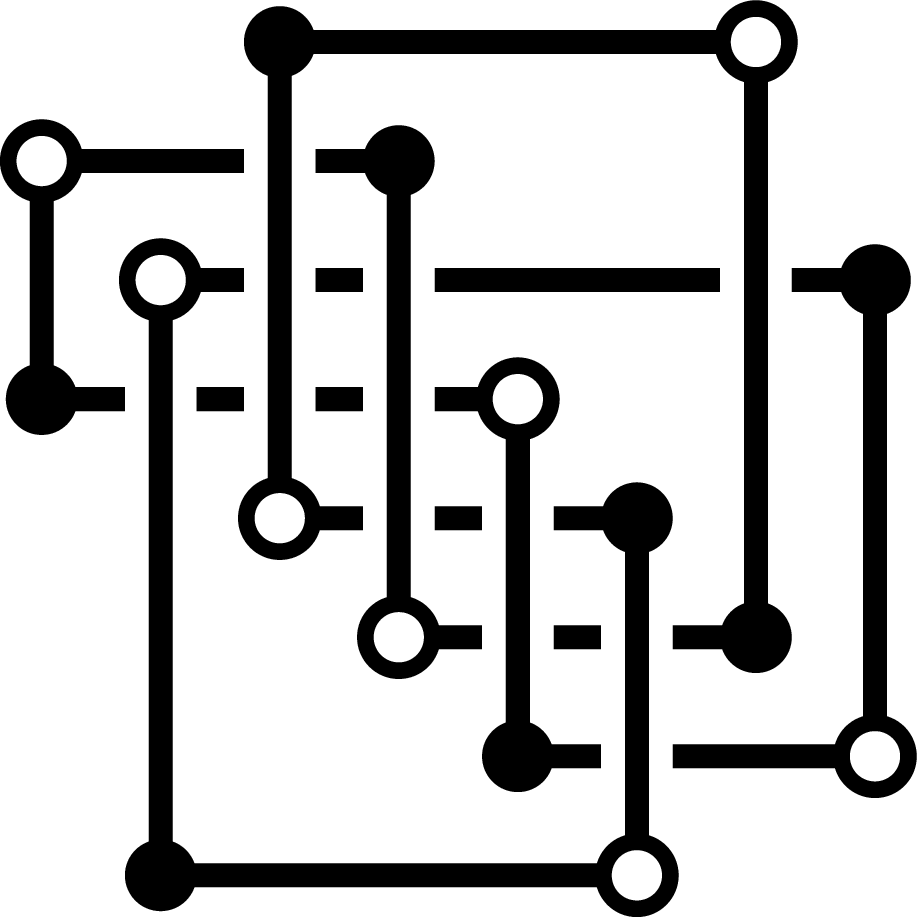}&\kern-.6em\raisebox{12pt}{$\rightarrow$}\kern-.6em&
\includegraphics[width=30pt]{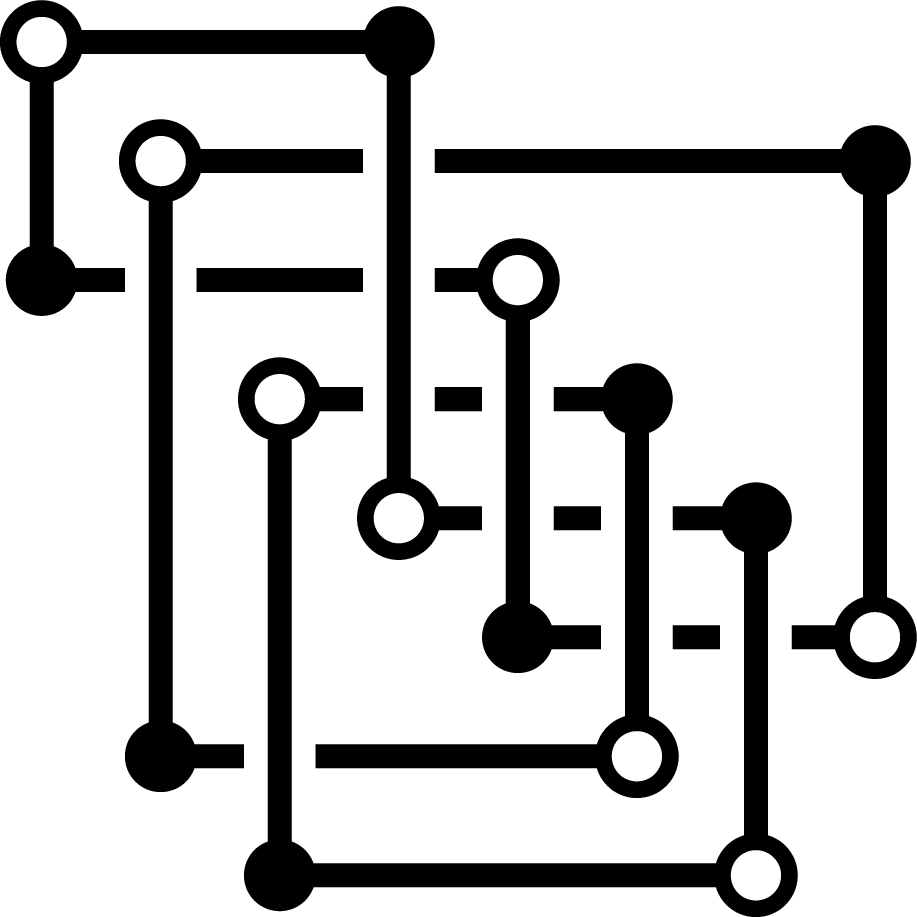}&\kern-.6em\raisebox{12pt}{$\rightarrow$}\kern-.6em&
\includegraphics[width=30pt]{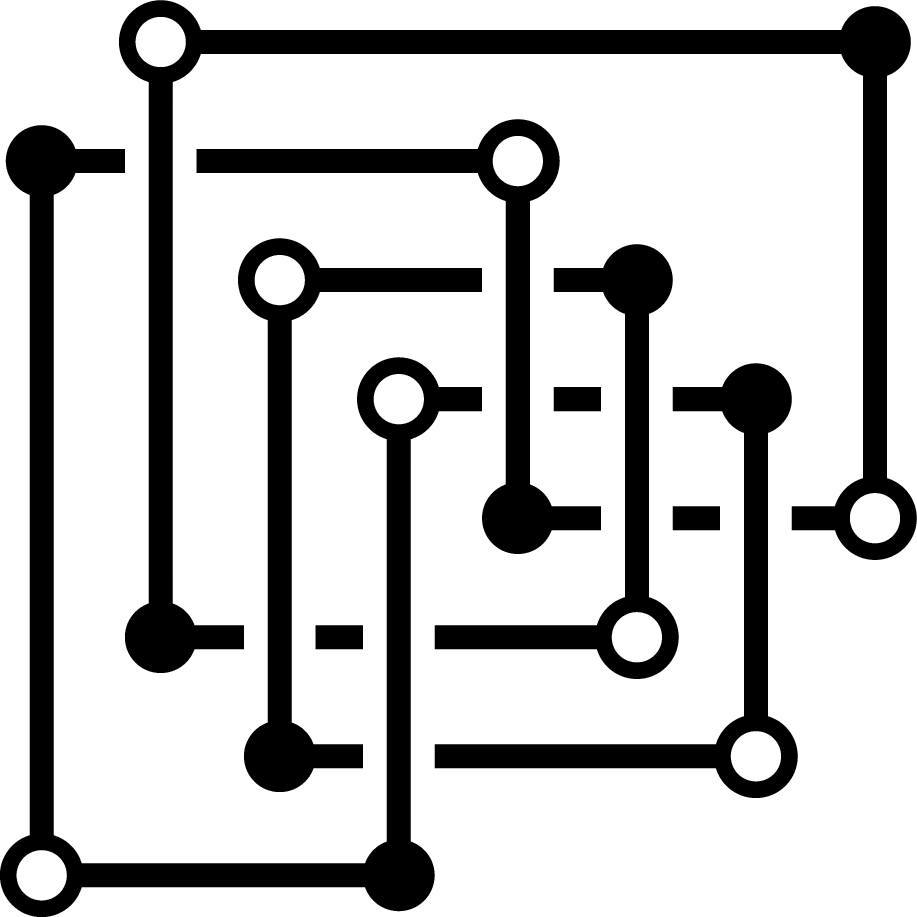}&\kern-.6em\raisebox{12pt}{$\rightarrow$}\kern-.6em&
\includegraphics[width=30pt]{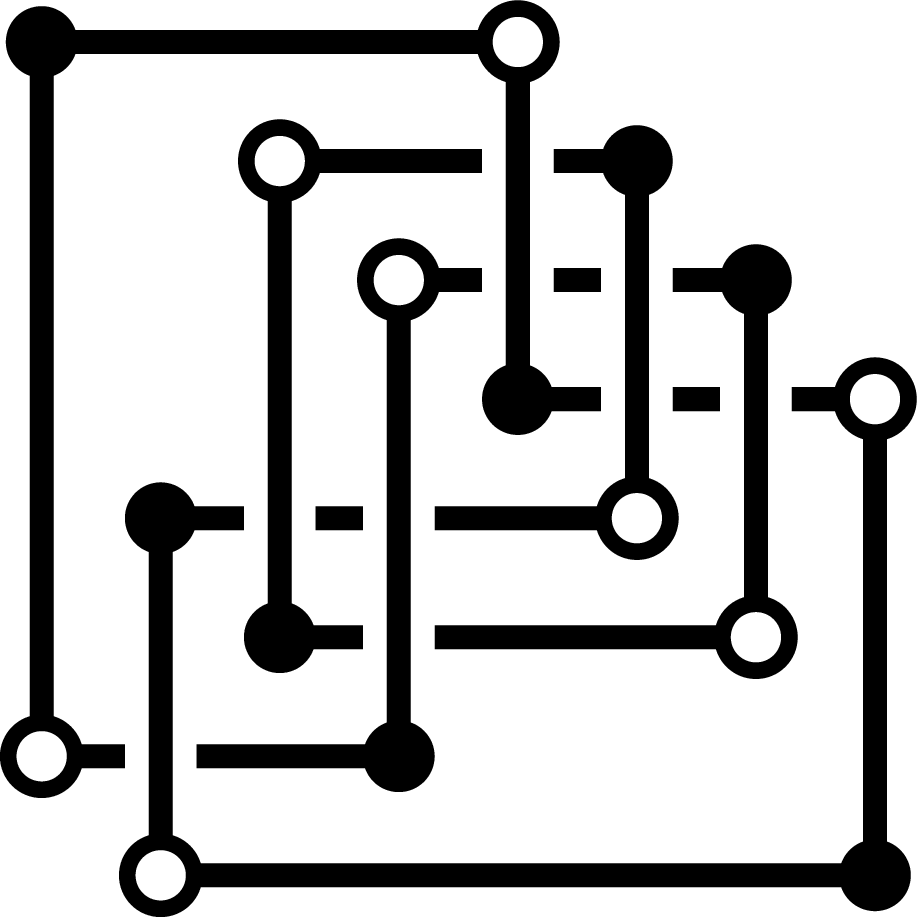}\\
\end{array},$$
which completes the proof.
\end{proof}

The proof of Proposition~\ref{9_42_prop} is summarized in Figure~\ref{42-fig}. In what follows
we present the proofs by similar schemes omitting the verbal description.
For a routine check of all equalities and inequalities of exchange classes used
in the proofs the reader is referred to~\cite{anc}.

In the proofs of Propositions~\ref{9_44_prop}, \ref{9_45_prop}, and~\ref{10_160_prop}
we may also silently use symmetries: an inequality~$X\ne Y$, where~$X$ and~$Y$
are some Legendrian or exchange classes, is equivalent to any of~$-X\ne-Y$, $\mu(X)\ne\mu(Y)$.
Another use of symmetries is as follows. If~$X$ and~$Y$ are Legendrian classes
such that~$X=-X$ and~$Y\ne-Y$ (similarly for~$\mu$ or~$-\mu$ in place of~`$-$'), then we immediately know
that~$X\notin\{Y,-Y,\mu(Y),-\mu(Y)\}$.

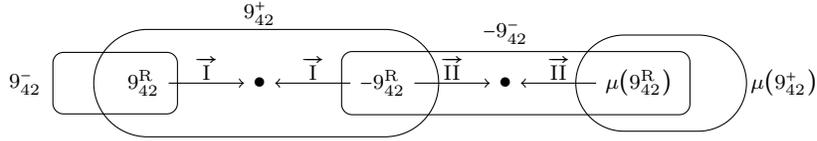
\begin{figure}[ht!]
\begin{center}
\begin{tikzpicture}[node distance = 2 cm,auto, ]
\node[] (1) {\footnotesize $9_{42}^{\mathrm R}$} ;
\node[left = 0.5cm of 1] (1v) {} ;
\node[left = 0.15cm of 1v] (1-) {\footnotesize $9_{42}^{-}$} ;
\node[right = 1cm of 1] (dot1) {$\bullet$} ;
\node[above = 0.45cm of dot1] (1+) {\footnotesize $9_{42}^{+}$} ;
\node[above = 0.15cm of dot1] (dot1v1) {} ;
\node[below = 0.15cm of dot1] (dot1v2) {} ;
\node[right = 1cm of dot1] (-1) {\footnotesize $-9_{42}^{\mathrm R}$} ;
\node[right = 1cm of -1] (dot2) {$\bullet$} ;
\node[above = 0.2cm of dot2] (-1-) {\footnotesize $-9_{42}^{-}$} ;
\node[right = 1cm of dot2] (mu1) {\footnotesize $\mu\bigl(9_{42}^{\mathrm R}\bigr)$} ;
\node[right = 0.8cm of mu1] (mu1+) {\footnotesize $\mu\bigl(9_{42}^{+}\bigr)$} ;
\node[right = 0.35cm of mu1] (mu1v) {} ;
\node[above = 0.15cm of mu1v] (dot2v1) {} ;
\node[below = 0.15cm of mu1v] (dot2v2) {} ;
\path[->,draw]
(1) edge node[midway,above = 1pt,fill=white,inner sep=0pt] {\footnotesize $\overrightarrow{\mathrm{I}}$} (dot1)
(-1) edge node[midway,above = 1pt,fill=white,inner sep=0pt] {\footnotesize $\overrightarrow{\mathrm{I}}$} (dot1)
(-1) edge node[midway,above = 1pt,fill=white,inner sep=0pt] {\footnotesize $\overrightarrow{\mathrm{II}}$} (dot2)
(mu1) edge node[midway,above = 1pt,fill=white,inner sep=0pt] {\footnotesize $\overrightarrow{\mathrm{II}}$} (dot2);
\begin{pgfonlayer}{background}
  \node[fit=(1)(1v), vfit] {};
  \node[fit=(1)(dot1)(-1)(dot1v1)(dot1v2), hfit] {};
  \node[fit=(-1)(dot2)(mu1), vfit] {};
  \node[fit=(mu1)(mu1v)(dot2v1)(dot2v2), hfit] {};
\end{pgfonlayer}
\end{tikzpicture}
\end{center}
\caption{Proof of Proposition~\ref{9_42_prop}} \label{42-fig}
\end{figure}

\begin{figure}[ht!]
\includegraphics[scale=.5]{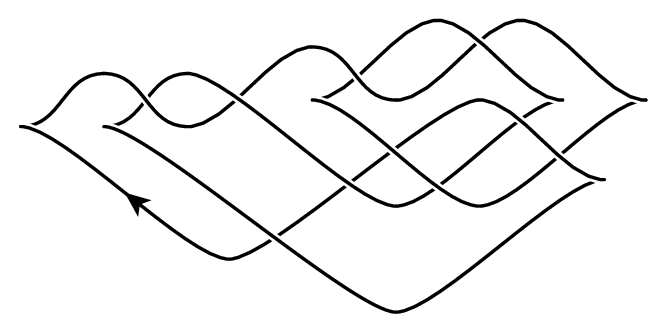}\put(-95,0){$9_{43}^+$}
\hskip1cm
\includegraphics[scale=.5]{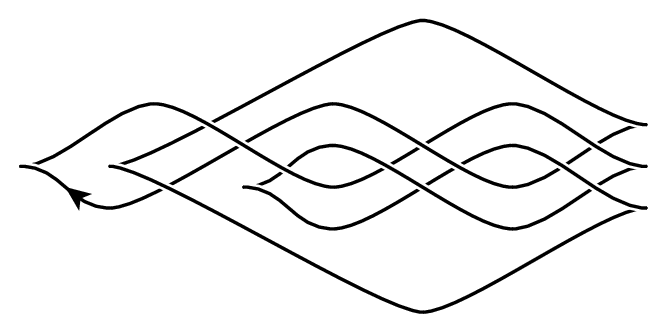}\put(-95,0){$r_\medvert(9_{43}^-)$}
\caption{Legendrian knots in Proposition~\ref{9_43_prop}}\label{9-43-knots-fig}
\end{figure}

\begin{prop}\label{9_43_prop}
For the $\xi_\pm$-Legendrian classes whose representatives are shown in Figure~\ref{9-43-knots-fig}, we have
$9_{43}^+\ne-9_{43}^+$ and $9_{43}^-\ne-\mu(9_{43}^-)$.
\end{prop}

The proof is presented in Figure~\ref{9-43-proof-fig}.

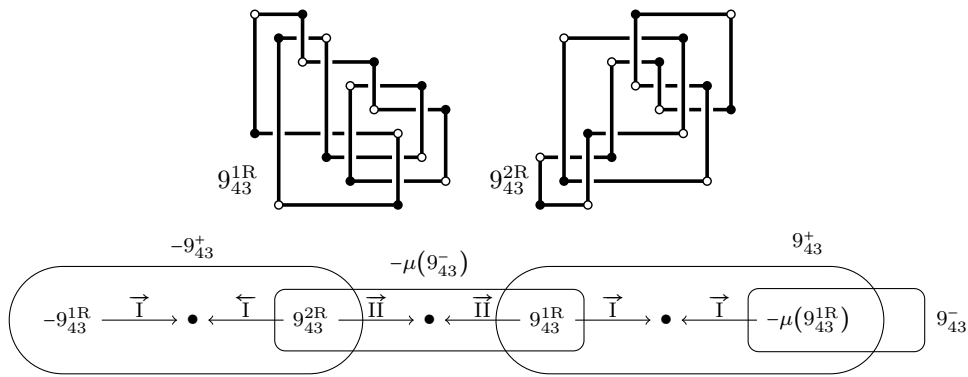
\begin{figure}[ht!]
\includegraphics[scale=.18]{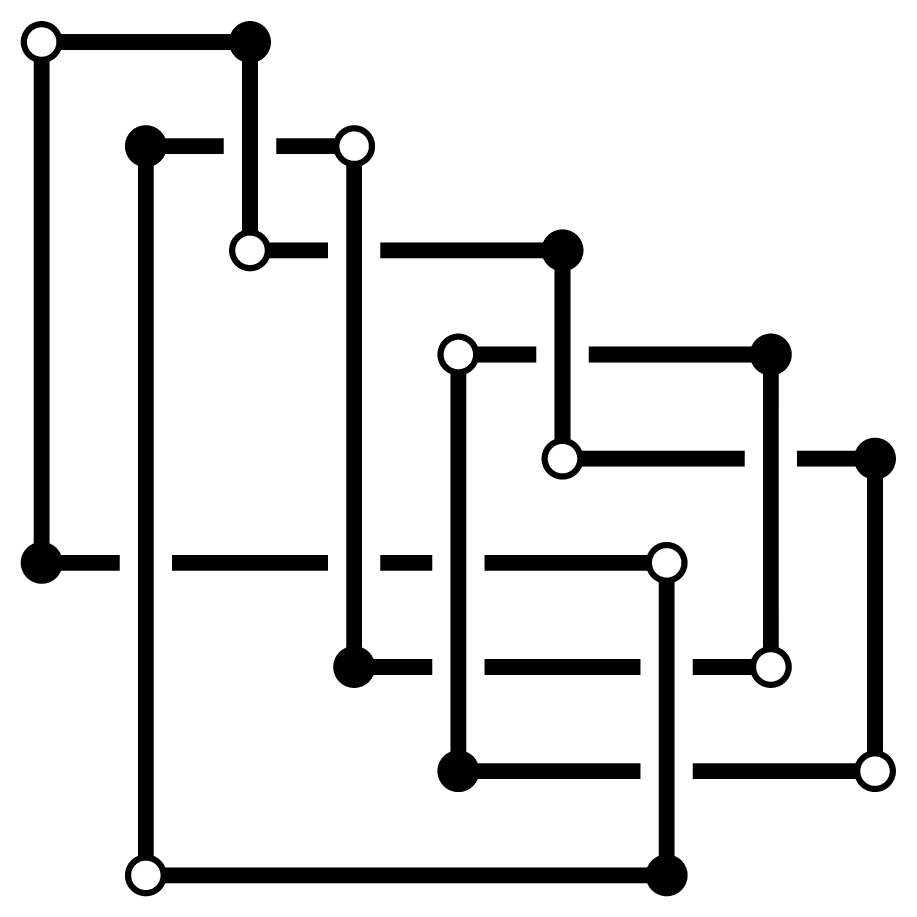}\put(-90,10){$9_{43}^{1\mathrm R}$}
\hskip1cm
\includegraphics[scale=.18]{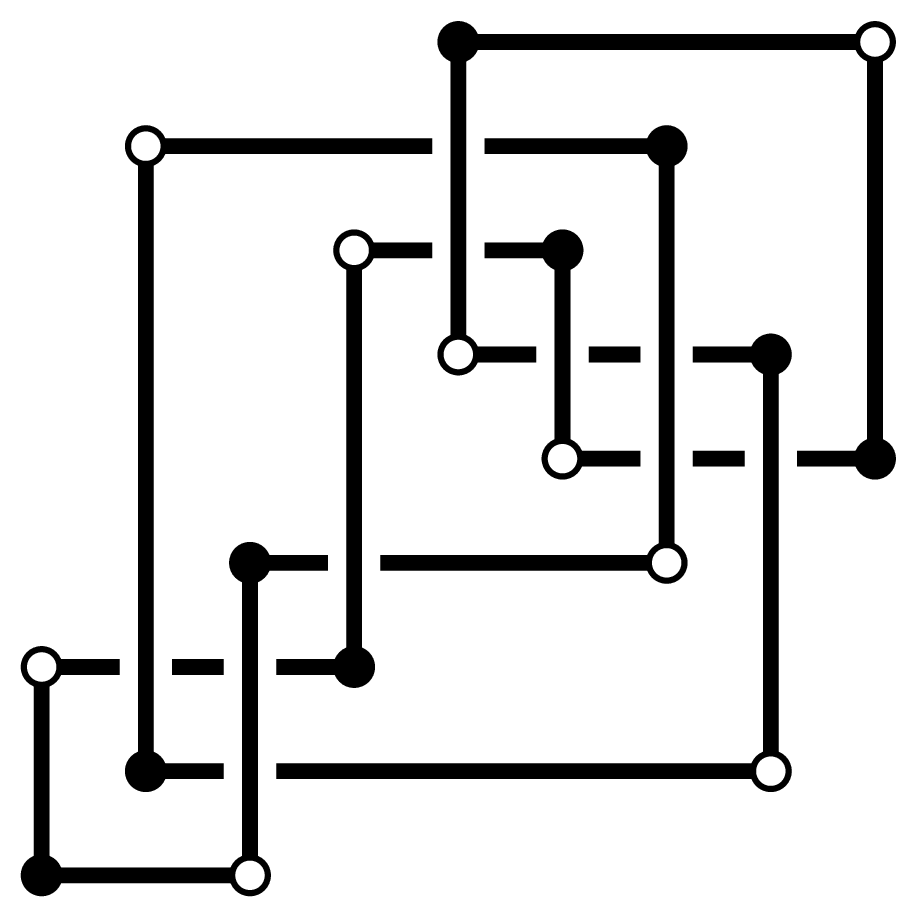}\put(-95,10){$9_{43}^{2\mathrm R}$}
\begin{center}
\begin{tikzpicture}[node distance = 2 cm,auto, ]
\node[] (-1) {\footnotesize $-9_{43}^{1\mathrm R}$} ;
\node[right = 1cm of -1] (dot1) {$\bullet$} ;
\node[above = 0.15cm of dot1] (dot1v1) {} ;
\node[below = 0.15cm of dot1] (dot1v2) {} ;
\node[right = 1cm of dot1] (2) {\footnotesize $9_{43}^{2\mathrm R}$} ;
\node[right = 1cm of 2] (dot2) {$\bullet$} ;
\node[above = 0.25cm of dot2] (-mu-) {\footnotesize $-\mu\bigl(9_{43}^{-}\bigr)$} ;
\node[right = 1cm of dot2] (1) {\footnotesize $9_{43}^{1\mathrm R}$} ;
\node[right = 1cm of 1] (dot3) {$\bullet$} ;
\node[above = 0.1cm of dot1v1] (-+) {\footnotesize $-9_{43}^{+}$} ;
\node[right = 1cm of dot3] (-mu1) {\footnotesize $-\mu\bigl(9_{43}^{1\mathrm R}\bigr)$} ;
\node[above = 0.45cm of -mu1] (+) {\footnotesize $9_{43}^{+}$} ;
\node[right = 0.5cm of -mu1] (-mu1v) {} ;
\node[right = 0.15cm of -mu1v] (-) {\footnotesize $9_{43}^{-}$} ;
\node[above right = 0.15cm and 0.5cm of dot3] (dot2v1) {} ;
\node[below right = 0.15cm and 0.5cm of dot3] (dot2v2) {} ;
\path[->,draw]
(-1) edge node[midway,above = 1pt,fill=white,inner sep=0pt] {\footnotesize $\overrightarrow{\mathrm{I}}$} (dot1)
(2) edge node[midway,above = 1pt,fill=white,inner sep=0pt] {\footnotesize $\overleftarrow{\mathrm{I}}$} (dot1)
(2) edge node[midway,above = 1pt,fill=white,inner sep=0pt] {\footnotesize $\overrightarrow{\mathrm{II}}$} (dot2)
(1) edge node[midway,above = 1pt,fill=white,inner sep=0pt] {\footnotesize $\overrightarrow{\mathrm{II}}$} (dot2)
(1) edge node[midway,above = 1pt,fill=white,inner sep=0pt] {\footnotesize $\overrightarrow{\mathrm{I}}$} (dot3)
(-mu1) edge node[midway,above = 1pt,fill=white,inner sep=0pt] {\footnotesize $\overrightarrow{\mathrm{I}}$} (dot3);
\begin{pgfonlayer}{background}
  \node[fit=(-1)(dot1)(2)(dot1v1)(dot1v2), hfit] {};
  \node[fit=(2)(dot2)(1), vfit] {};
  \node[fit=(1)(dot3)(-mu1)(dot2v1)(dot2v2), hfit] {};
  \node[fit=(-mu1)(-mu1v), vfit] {};
\end{pgfonlayer}
\end{tikzpicture}
\end{center}
\caption{Proof of Proposition~\ref{9_43_prop}} \label{9-43-proof-fig}
\end{figure}

\begin{figure}[ht!]
\includegraphics[scale=.5]{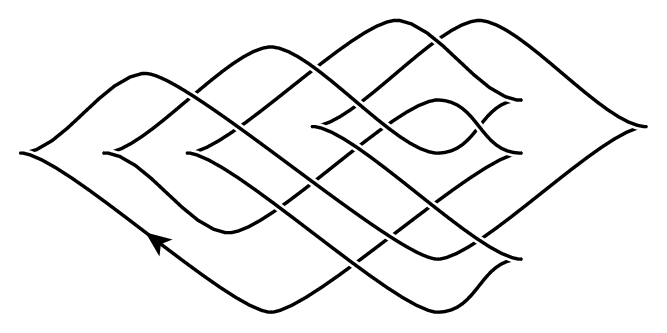}\put(-140,10){$9_{44}^{1+}$}
\hskip1cm
\includegraphics[scale=.5]{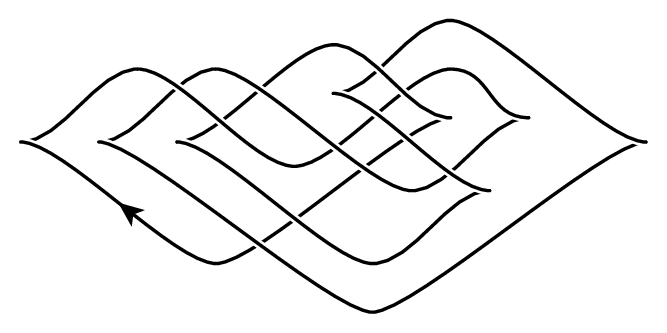}\put(-140,10){$9_{44}^{2+}$}
\\
\includegraphics[scale=.5]{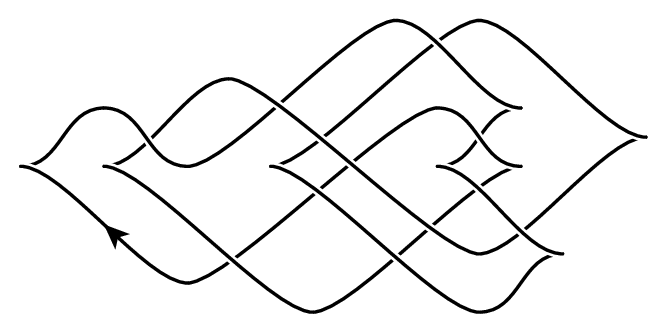}\put(-140,10){$9_{44}^{3+}$}
\hskip1cm
\includegraphics[scale=.5]{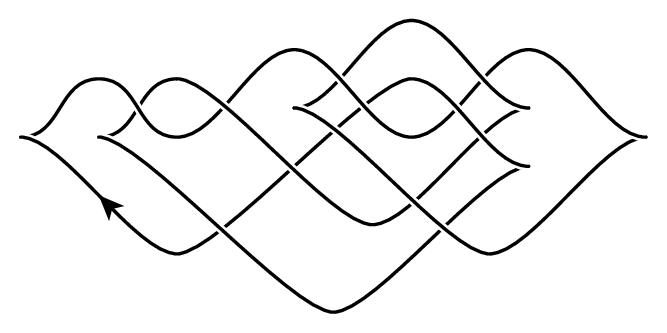}\put(-140,10){$r_\medvert(9_{44}^-)$}
\caption{The knots in Proposition~\ref{9_44_prop}}\label{9_44-fig}
\end{figure}
\begin{figure}[ht]
\includegraphics[scale=.18]{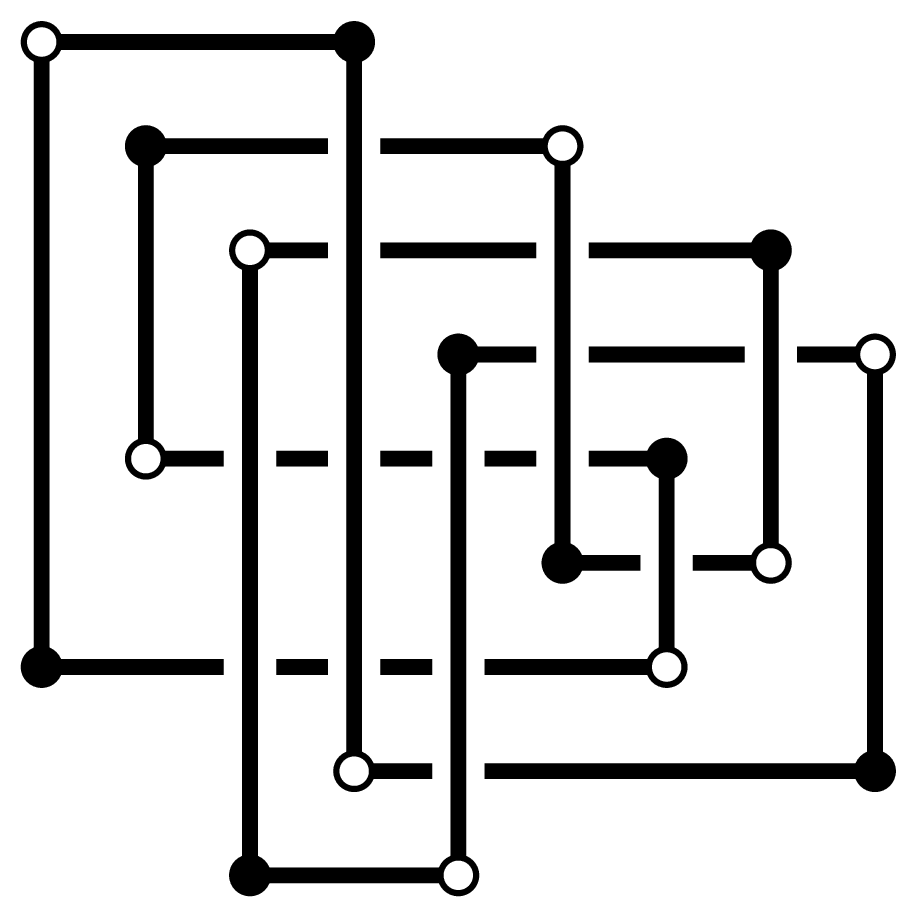}\put(-90,10){$9_{44}^{1\mathrm R}$}
\hskip1cm
\includegraphics[scale=.18]{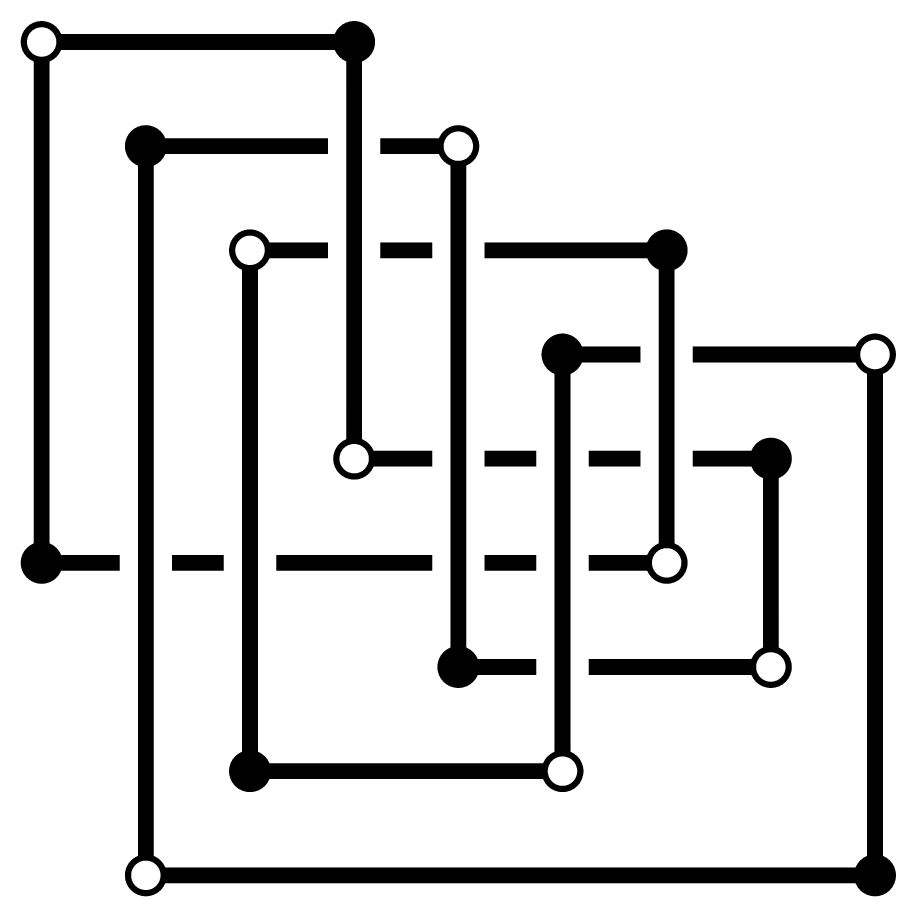}\put(-90,10){$9_{44}^{2\mathrm R}$}
\hskip1cm
\includegraphics[scale=.18]{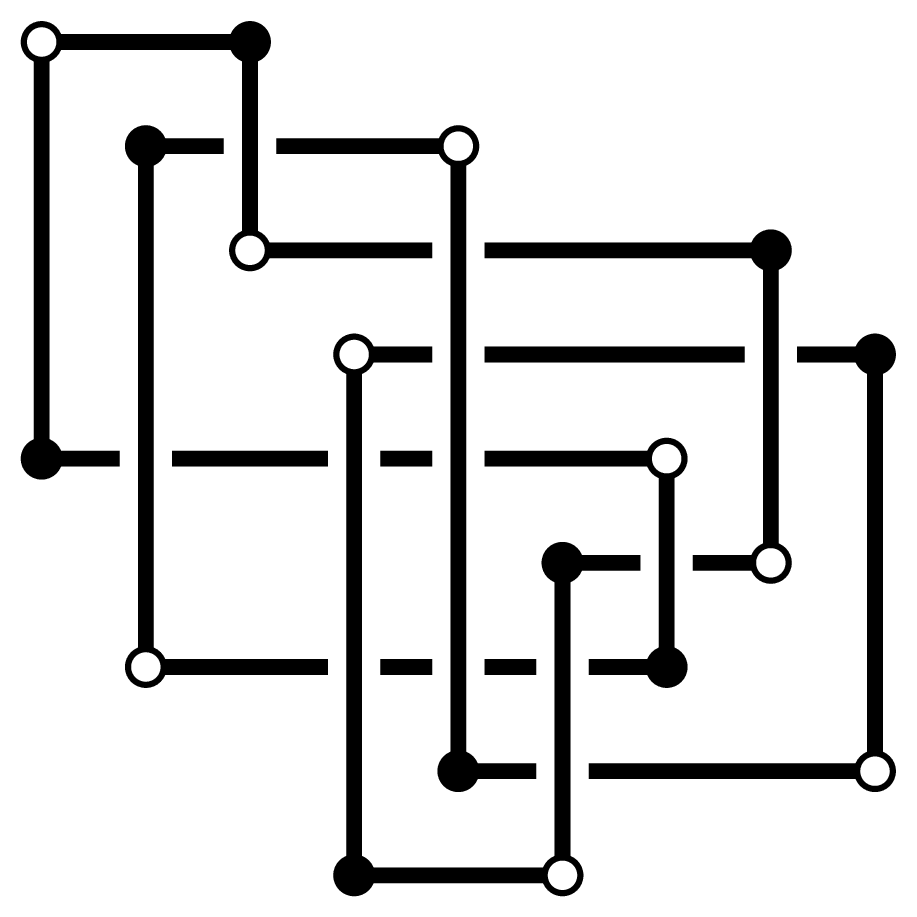}\put(-90,10){$9_{44}^{3\mathrm R}$}
\\
\includegraphics[scale=.18]{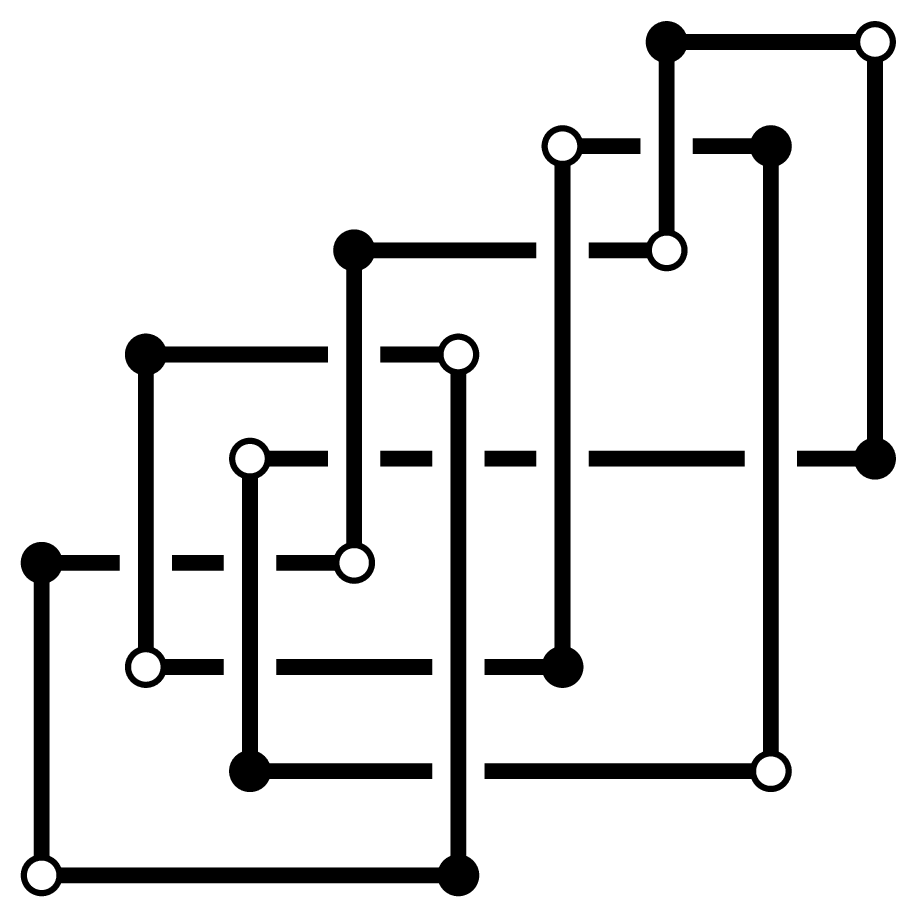}\put(-95,10){$9_{44}^{4\mathrm R}$}
\hskip1cm
\includegraphics[scale=.18]{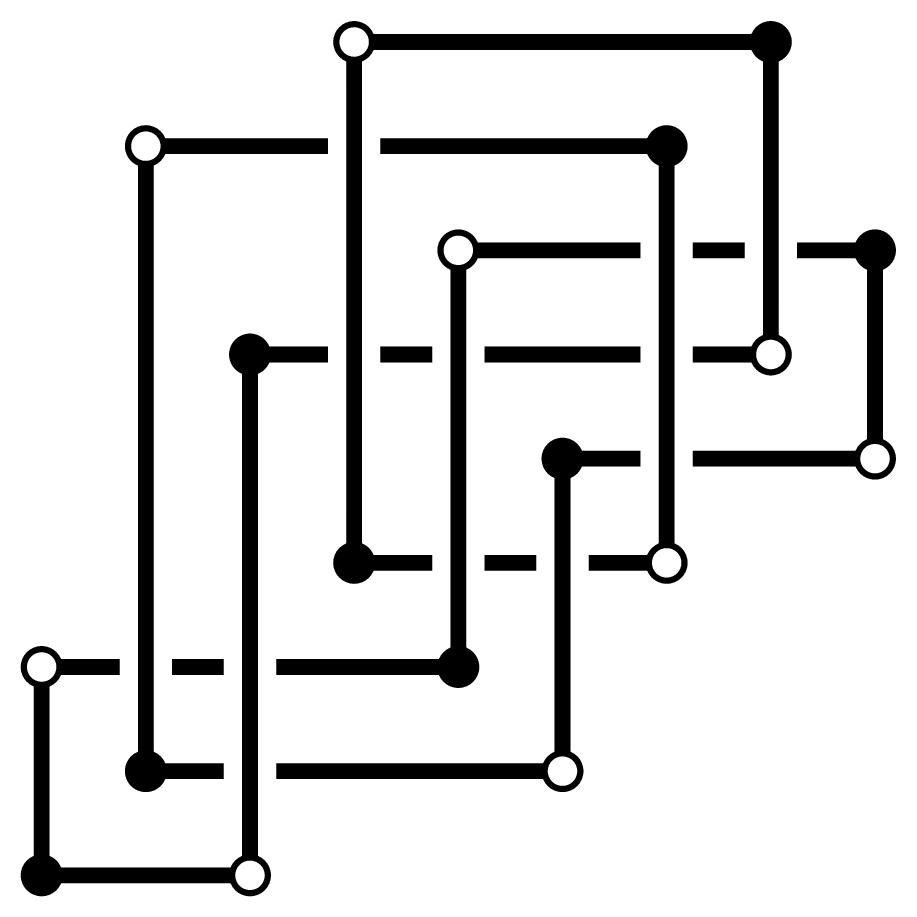}\put(-95,10){$9_{44}^{5\mathrm R}$}
\caption{Exchange classes used in the proof of Proposition~\ref{9_44_prop}}\label{9_44-rect-fig}
\end{figure}

\def\theenumi{\roman{enumi}}
\begin{prop}\label{9_44_prop}
For the $\xi_\pm$-Legendrian classes whose representatives are shown in Figure~\ref{9_44-fig}
the following holds\emph:
\begin{enumerate}
\item
the $\xi_+$-Legendrian classes
$9_{44}^{1+}$, $9_{44}^{2+}$, $9_{44}^{3+}$,
$-\mu(9_{44}^{1+})$,
$-\mu(9_{44}^{2+})$,
$-\mu(9_{44}^{3+})$ are pairwise distinct;
\item
for $k\in\{1,2,3,4\}$ the $\xi_+$-Legendrian classes
$S_+^k(9_{44}^{1+})$, $S_+^k(9_{44}^{2+})$, and~$S_+^k(9_{44}^{3+})$
are pairwise distinct;
\item
the $\xi_-$-Legendrian classes
$9_{44}^-$ and $-9_{44}^-$ are distinct.
\end{enumerate}
\end{prop}

\begin{proof}
Representatives of the exchange classes involved in the proof are shown in Figure~\ref{9_44-rect-fig}.
The fact that~$9_{44}^{1+},-\mu(9_{44}^{1+})\notin\{9_{44}^{2+},9_{44}^{3+}\}$ and~$S_+^k(9_{44}^{1+})\notin\bigl\{S_+^k(9_{44}^{2+}),S_+^k(9_{44}^{3+})\bigr\}$ for any
$k\in\mathbb N$, is established in~\cite{chong2013}. The proof of the remaining claims is presented in Figure~\ref{44} (where some of the known facts are also reproved).
\begin{figure}[ht!]
\begin{center}
\begin{tikzpicture}[node distance = 2 cm,auto, ]
\node[] (-mu1) {\footnotesize $-\mu\bigl(9_{44}^{1\mathrm R}\bigr)$} ;
\node[right = 1.5cm of -mu1] (5) {\footnotesize $9_{44}^{5\mathrm R}$} ;
\node[right = 1cm of 5] (dot1) {\footnotesize $\bullet$} ;
\node[right = 1cm of dot1] (2) {\footnotesize $9_{44}^{2\mathrm R}$} ;
\node[right = 1.5cm of 2] (-mu3) {\footnotesize $-\mu\bigl(9_{44}^{3\mathrm R}\bigr)$} ;
\node[below = 1cm of -mu1] (dot2) {} ;
\node[right = 1cm of dot2] (dot6) {\footnotesize $\bullet$} ;
\node[below = 1cm of -mu3] (-4) {\footnotesize $-9_{44}^{4\mathrm R}$} ;
\node[left = 1.5cm of -4] (dot8) {\footnotesize $\bullet$} ;
\node[below = 2.5cm of -mu1] (1) {\footnotesize $9_{44}^{1\mathrm R}$} ;
\node[below = 2.5cm of 5] (4) {\footnotesize $9_{44}^{4\mathrm R}$} ;
\node[right = 1cm of 4] (dot7) {\footnotesize $\bullet$} ;
\node[below = 2.5cm of 2] (3) {\footnotesize $9_{44}^{3\mathrm R}$} ;
\node[below = 2.5cm of -mu3] (-mu2) {\footnotesize $-\mu\bigl(9_{44}^{2\mathrm R}\bigr)$} ;
\node[below right = 2.3cm and 0.5cm of dot8] (dot9) {\footnotesize $\bullet$} ;
\node[below = 0.35cm of dot9] (dot10) {\footnotesize $\bullet$} ;
\node[below = 0.35cm of dot10] (dot11) {\footnotesize $\bullet$} ;
\node[below = 0.35cm of dot11] (dot110) {\footnotesize $\bullet$} ;
\node[above = 0.5cm of -mu1](-mu1+)  {};
\node[above = .8cm of -mu1](-mu1+n)  {\footnotesize $-\mu\bigl(9_{44}^{1+}\bigr)$};
\node[above right = 2.3cm and 0.5cm of dot8] (dot12) {\footnotesize $\bullet$} ;
\node[above = 0.35cm of dot12] (dot13) {\footnotesize $\bullet$} ;
\node[above = 0.35cm of dot13] (dot14) {\footnotesize $\bullet$} ;
\node[above = 0.35cm of dot14] (dot15) {\footnotesize $\bullet$} ;
\node[below = 0.5cm of 1](1+)  {};
\node[below = .8cm of 1](1+n)  {\footnotesize $9_{44}^{1+}$};
\node[above = 0.5cm of dot1](2+)  {};
\node[above = .8cm of dot1](2+n)  {\footnotesize $9_{44}^{2+}$};
\node[below = 0.5cm of dot7](3+)  {};
\node[below = .8cm of dot7](3+n)  {\footnotesize $9_{44}^{3+}$};
\node[right = 0.25cm of -mu2](-mu2+)  {};
\node[right = 0.6cm of -mu2](-mu2+n)  {\footnotesize $-\mu\bigl(9_{44}^{2+}\bigr)$};
\node[right = 0.25cm of -mu3](-mu3+)  {};
\node[right = .6cm of -mu3](-mu3+n)  {\footnotesize $-\mu\bigl(9_{44}^{3+}\bigr)$};
\node[left = 0.7cm of dot2](1-)  {\footnotesize $9_{44}^-$};
\node[right = 1.9cm of dot10](-1-)  {\footnotesize $-9_{44}^-$};
\path[->,draw]
(-mu1) edge node[midway,fill=white,inner sep=0pt] {\footnotesize $\overrightarrow{\mathrm{II}}$} (dot6)
(5) edge node[midway,fill=white,inner sep=0pt] {\footnotesize $\overrightarrow{\mathrm{II}}$} (dot6)
(4) edge node[midway,fill=white,inner sep=0pt] {\footnotesize $\overrightarrow{\mathrm{II}}$} (dot6)
(1) edge node[midway,fill=white,inner sep=0pt] {\footnotesize $\overrightarrow{\mathrm{II}}$} (dot6)
(5) edge node[midway,above = 1pt,fill=white,inner sep=0pt] {\footnotesize $\overrightarrow{\mathrm{I}}$} (dot1)
(2) edge node[midway,above = 1pt,fill=white,inner sep=0pt] {\footnotesize $\overrightarrow{\mathrm{I}}$} (dot1)
(4) edge node[midway,above = 1pt,fill=white,inner sep=0pt] {\footnotesize $\overrightarrow{\mathrm{I}}$} (dot7)
(3) edge node[midway,above = 1pt, fill=white,inner sep=0pt] {\footnotesize $\overrightarrow{\mathrm{I}}$} (dot7)
(2) edge node[midway, fill=white,inner sep=0pt] {\footnotesize $\overrightarrow{\mathrm{II}}$} (dot8)
(3) edge node[midway, fill=white,inner sep=0pt] {\footnotesize $\overrightarrow{\mathrm{II}}$} (dot8)
(-mu2) edge node[midway, fill=white,inner sep=0pt] {\footnotesize $\overrightarrow{\mathrm{II}}$} (dot8)
(-mu3) edge node[midway, fill=white,inner sep=0pt] {\footnotesize $\overrightarrow{\mathrm{II}}$} (dot8)
(-4) edge node[midway,above= 1pt, fill=white,inner sep=0pt] {\footnotesize $\overleftarrow{\mathrm{II}}$} (dot8)
(3) edge node[midway,above right = 1pt and 1pt, fill=white,inner sep=0pt] {\footnotesize $\overleftarrow{\mathrm{II}}$} (dot9)
(-mu2) edge node[midway,above left = 1pt and 1pt, fill=white,inner sep=0pt] {\footnotesize $\overleftarrow{\mathrm{II}}$} (dot9)
(dot9) edge node[midway,left = 1pt, fill=white,inner sep=0pt] {\footnotesize $\overleftarrow{\mathrm{II}}$} (dot10)
(dot10) edge node[midway,left = 1pt, fill=white,inner sep=0pt] {\footnotesize $\overleftarrow{\mathrm{II}}$}  (dot11)
(dot11) edge node[midway,left = 1pt, fill=white,inner sep=0pt] {\footnotesize $\overleftarrow{\mathrm{II}}$}  (dot110)
(2) edge node[midway,above = 4pt , fill=white,inner sep=0pt] {\footnotesize $\overleftarrow{\mathrm{II}}$} (dot12)
(-mu3) edge node[midway,above =4pt, fill=white,inner sep=0pt] {\footnotesize $\overleftarrow{\mathrm{II}}$} (dot12)
(dot12) edge node[midway,left = 1pt, fill=white,inner sep=0pt] {\footnotesize $\overleftarrow{\mathrm{II}}$} (dot13)
(dot13) edge node[midway,left = 1pt, fill=white,inner sep=0pt] {\footnotesize $\overleftarrow{\mathrm{II}}$}  (dot14)
(dot14) edge node[midway,left = 1pt, fill=white,inner sep=0pt] {\footnotesize $\overleftarrow{\mathrm{II}}$}  (dot15);
\begin{pgfonlayer}{background}
  \node[fit=(-mu1)(5)(1)(4)(dot2)(dot6), vfit] {};
  \node[fit=(5)(dot1)(2)(2+), hfit] {};
  \node[fit=(4)(dot7)(3)(3+), hfit] {};
  \node[fit=(2)(-mu2)(dot8)(-4)(3)(-mu3)(dot9)(dot110)(dot12)(dot15), vfit] {};
  \node[fit=(-mu1)(-mu1+), hfit] {};
  \node[fit=(1)(1+), hfit] {};
  \node[fit=(-mu3)(-mu3+), hfit] {};
  \node[fit=(-mu2)(-mu2+), hfit] {};
\end{pgfonlayer}
\end{tikzpicture}
\end{center}
\caption{Proof of Proposition~\ref{9_44_prop}}\label{44}
\end{figure}
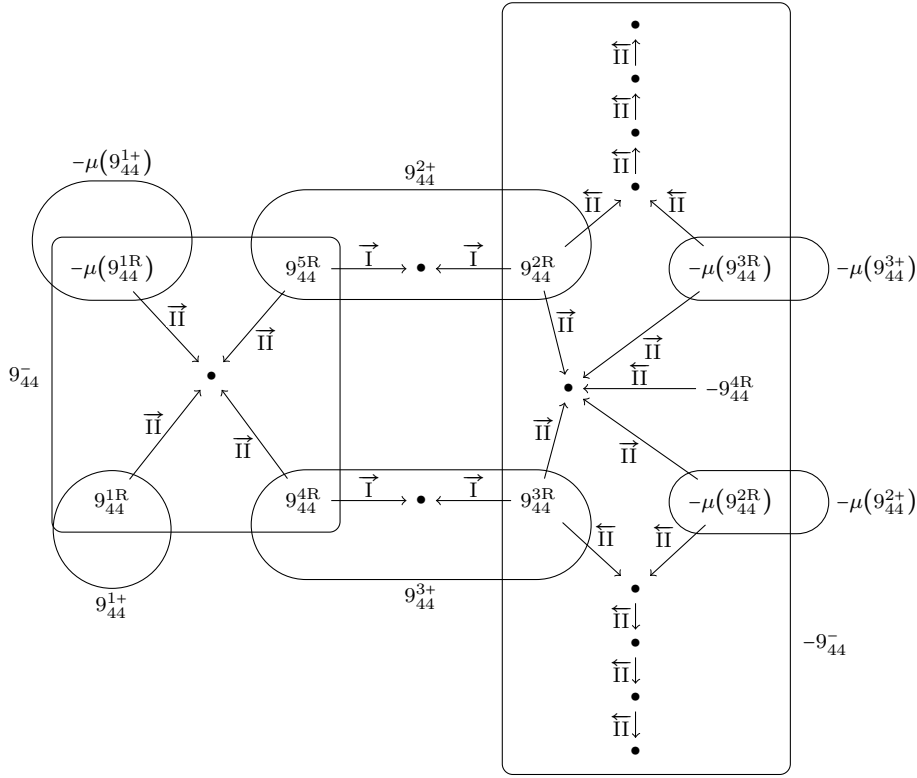
\end{proof}

\begin{rema}
It is conjectured in~\cite{chong2013} that the $\xi_+$-Legendrian classes
$S_+^k(9_{44}^{2+})$ and~$S_+^k(9_{44}^{3+})$
are distinct for any~$k\in\mathbb N$, not only~$k\leqslant4$.
The method of this paper allows, in principle, to test the claim for any fixed~$k$, and
this has been done by the authors for~$k\leqslant4$. (For larger~$k$, the simple---and far
from being optimized---exhaustive
search, which we used to test diagrams for exchange-equivalence, takes too much time.)

Proving the claim for all~$k$ is equivalent to distinguishing certain transverse knots.
The present technique has been upgraded in~\cite{trleg} to an algorithmic solution of this problem
(in the case of knots with trivial orientation-preserving symmetry group). This reduces the task
of verifying the inequality~$S_+^k(9_{44}^{2+})\ne S_+^k(9_{44}^{3+})$ for all~$k\in\mathbb N$
to a finite exhaustive search, which is still to be done.

A similar remark applies to part~(iii) of Proposition~\ref{9_45_prop} and
parts~(ii) of Propositions~\ref{10_128_prop} and~\ref{10_160_prop}.
In the last two cases the required exhaustive search appears to be trivial,
so the question for the knots~$10_{128}$ and~$10_{160}$ is settled in~\cite{trleg} completely.
\end{rema}

\begin{figure}[ht!]
\includegraphics[scale=.5]{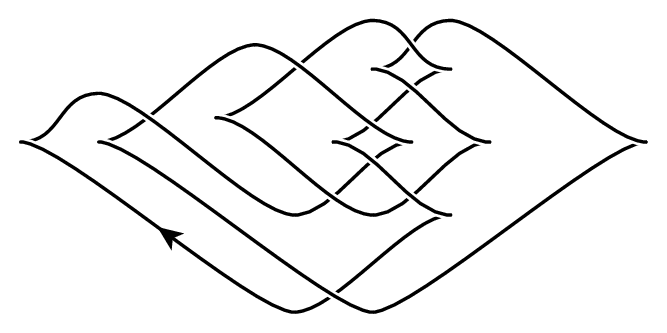}\put(-140,10){$9_{45}^{1+}$}
\hskip1cm
\includegraphics[scale=.5]{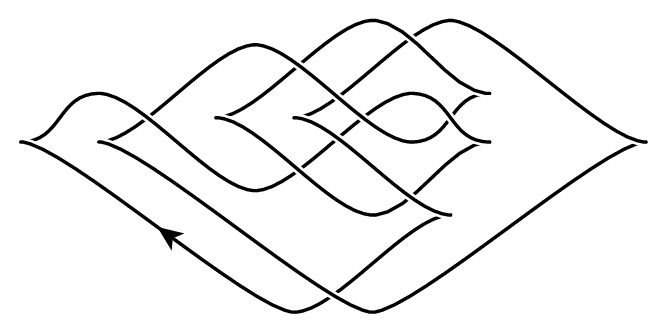}\put(-140,10){$9_{45}^{2+}$}
\\
\includegraphics[scale=.5]{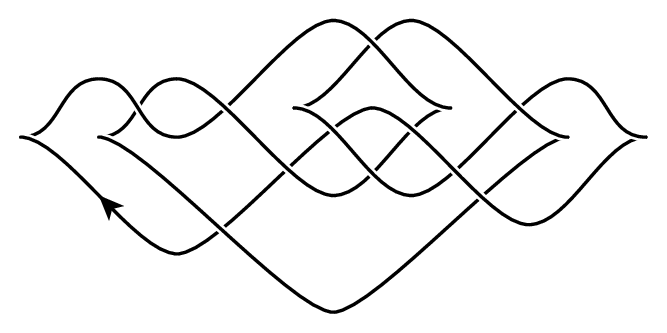}\put(-140,10){$9_{45}^{3+}$}
\hskip1cm
\includegraphics[scale=.5]{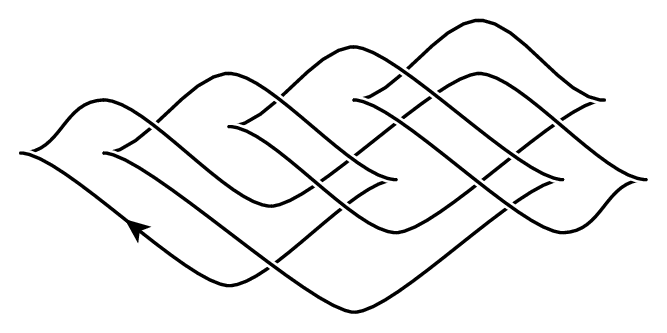}\put(-150,10){$r_\medvert(9_{45}^{1-})$}
\\
\includegraphics[scale=.5]{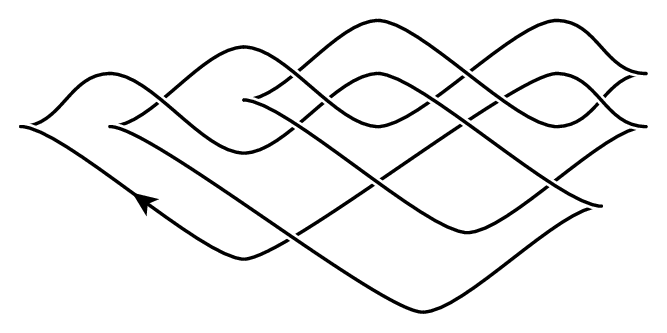}\put(-140,10){$r_\medvert(9_{45}^{2-})$}
\caption{The knots in Proposition~\ref{9_45_prop}}\label{45-knots-fig}
\end{figure}

\begin{figure}[ht]
\includegraphics[scale=.18]{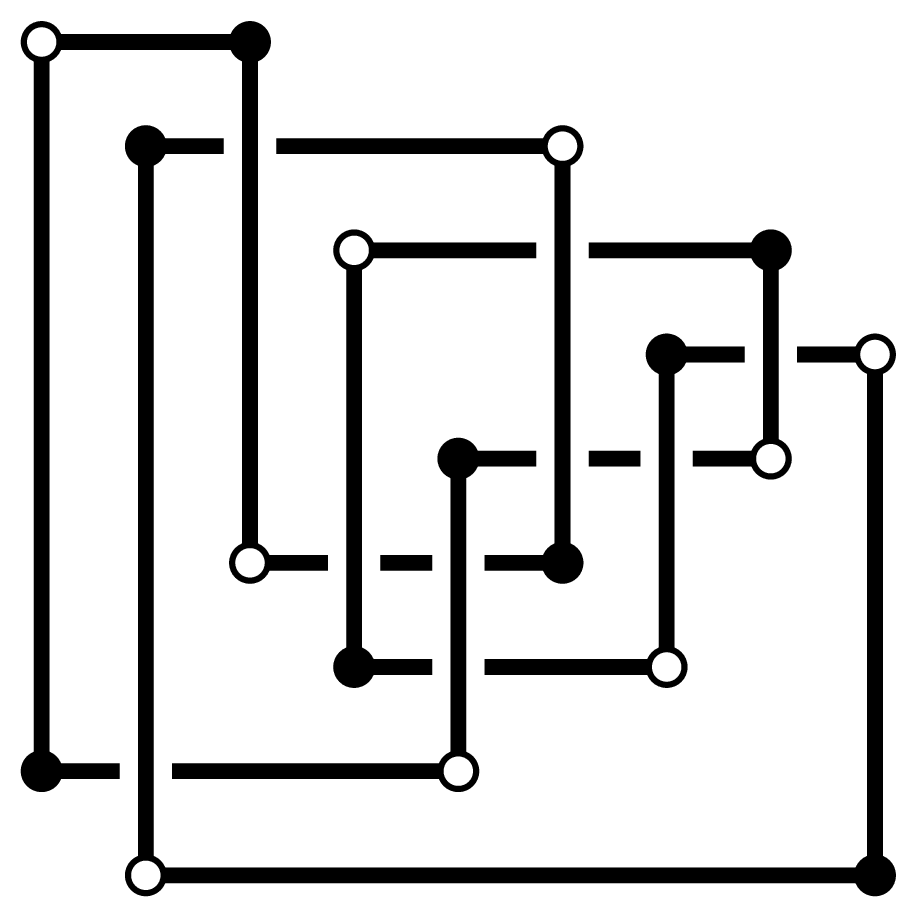}\put(-95,10){$9_{45}^{1\mathrm R}$}
\hskip1cm
\includegraphics[scale=.18]{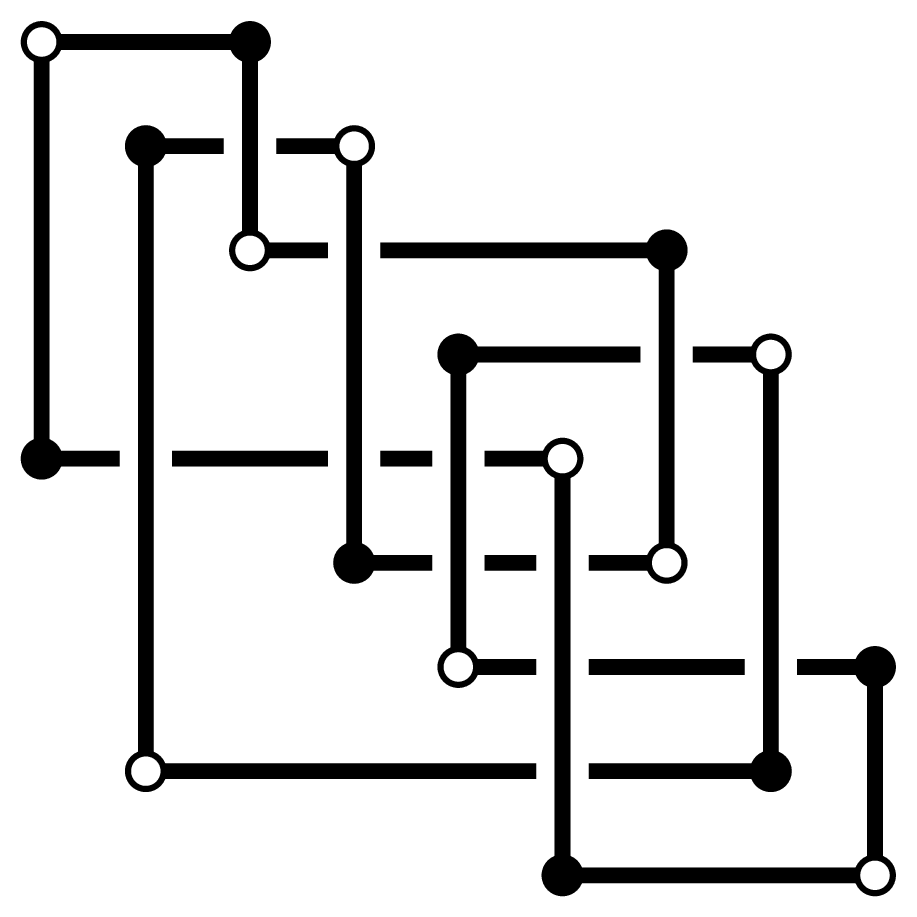}\put(-85,10){$9_{45}^{2\mathrm R}$}
\hskip1cm
\includegraphics[scale=.18]{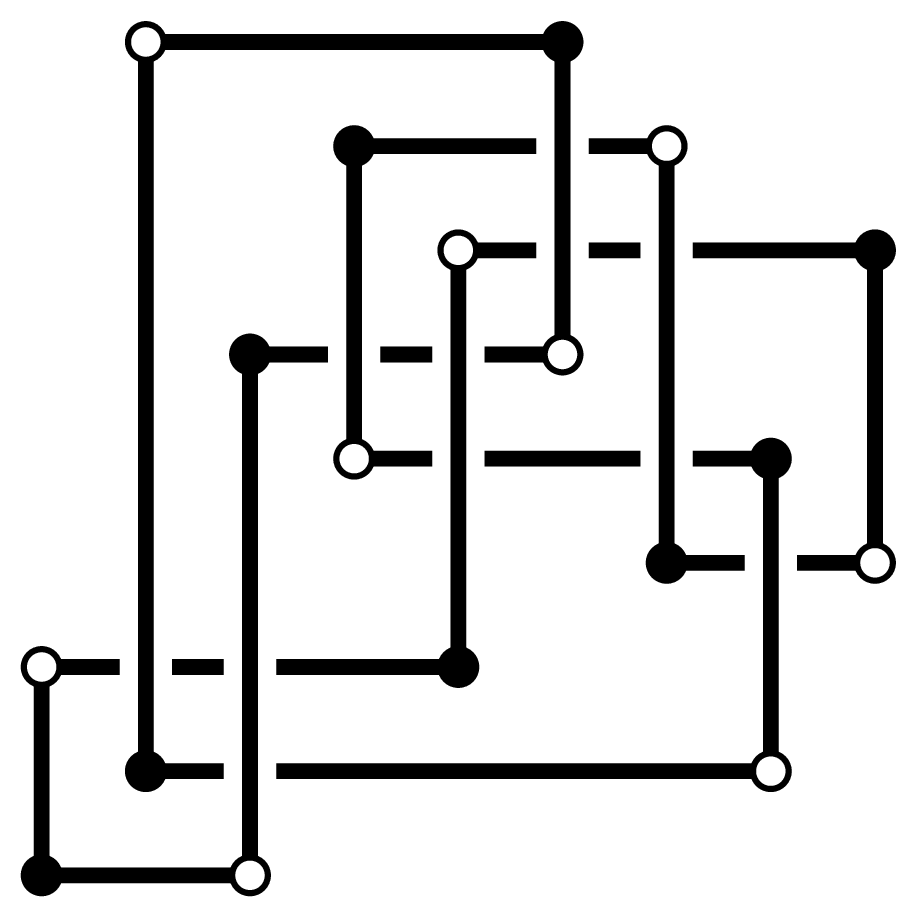}\put(-95,10){$9_{45}^{3\mathrm R}$}
\\
\includegraphics[scale=.18]{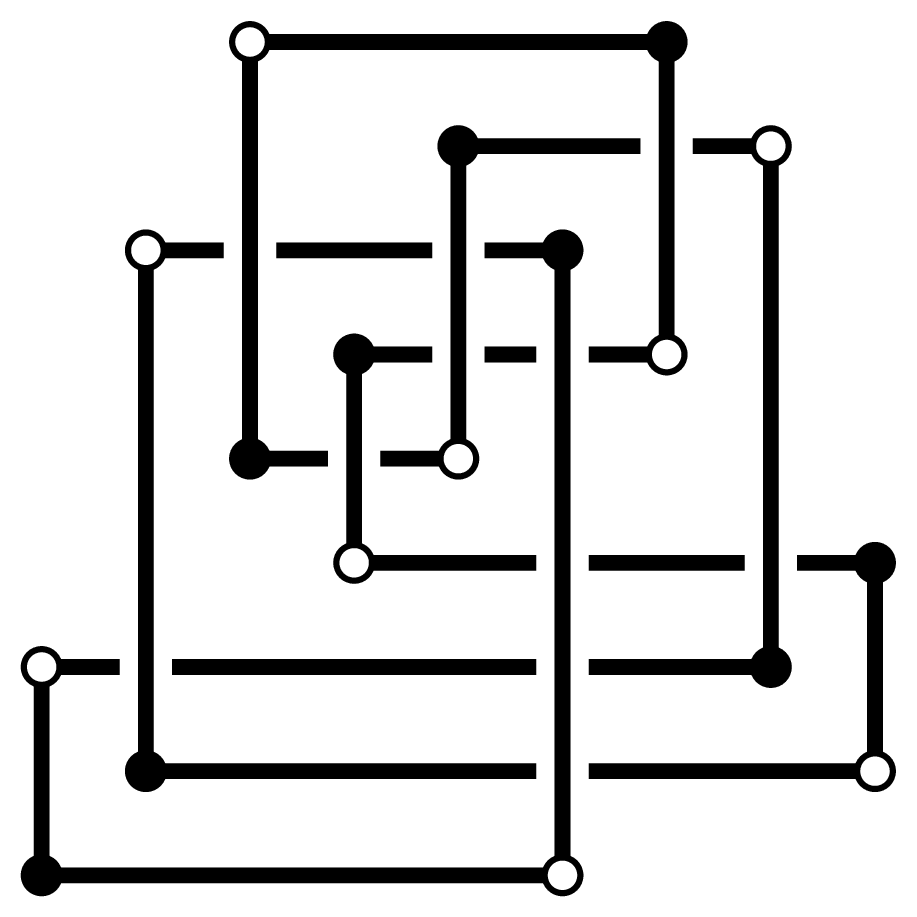}\put(-95,10){$9_{45}^{4\mathrm R}$}
\hskip1cm
\includegraphics[scale=.18]{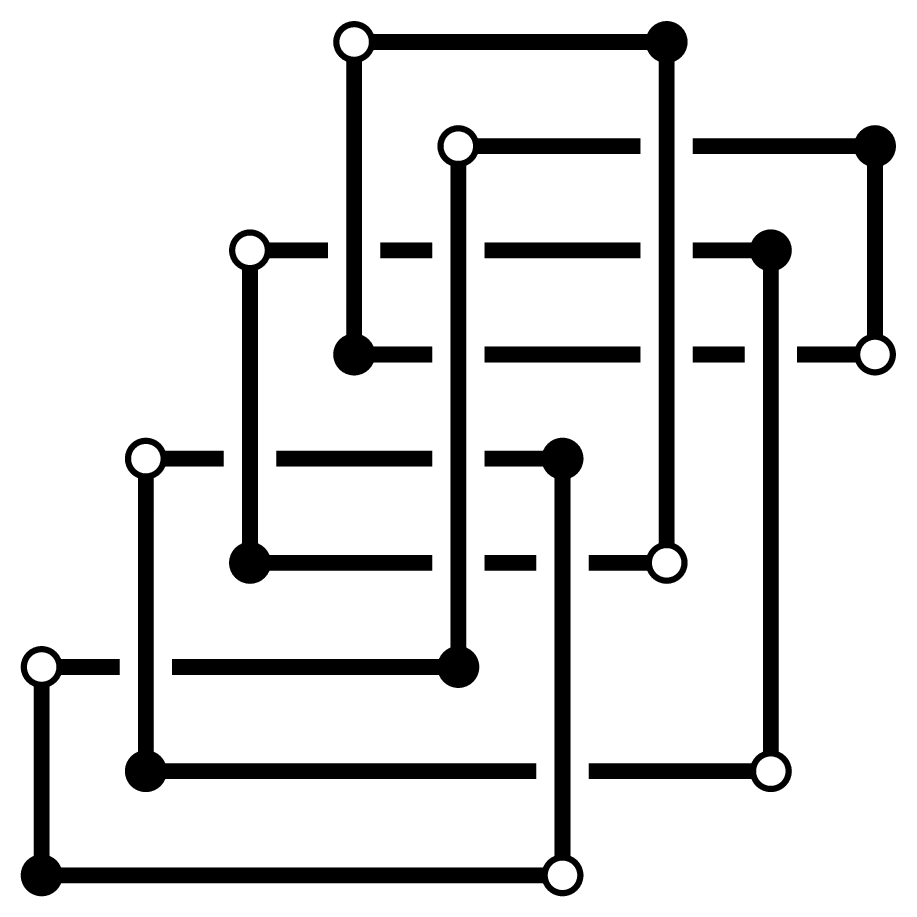}\put(-95,10){$9_{45}^{5\mathrm R}$}
\hskip1cm
\includegraphics[scale=.18]{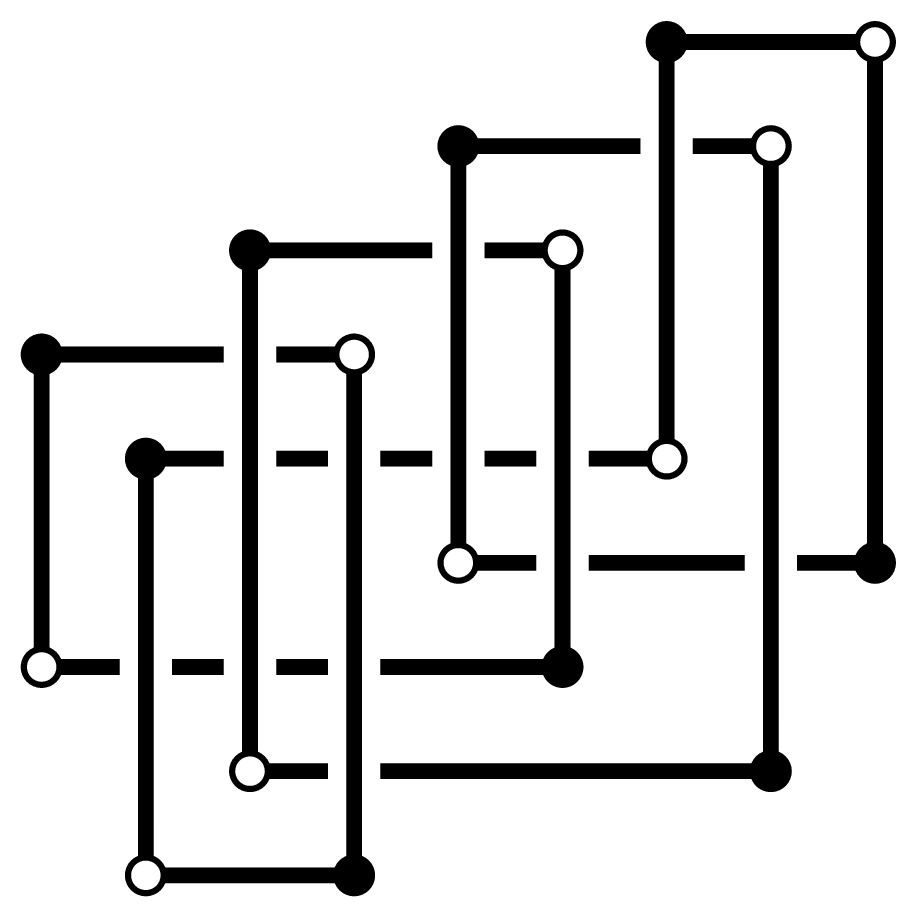}\put(-90,10){$9_{45}^{6\mathrm R}$}
\caption{Exchange classes used in the proof of Proposition~\ref{9_45_prop}}\label{9_45-rect-fig}
\end{figure}

\begin{prop}\label{9_45_prop}
For the $\xi_\pm$-Legendrian classes whose representatives are shown in Figure~\ref{45-knots-fig}
the following holds\emph:
\begin{enumerate}
\item
$9_{45}^{1+}$, $9_{45}^{2+}$, $9_{45}^{3+}$, $-\mu(9_{45}^{1+})$, and~$-\mu(9_{45}^{3+})$ are pairwise distinct\emph;
\item
$9_{45}^{1-}$, $-9_{45}^{1-}$, $\mu(9_{45}^{1-})$, $-\mu(9_{45}^{1-})$, $9_{45}^{2-}$, and~$\mu(9_{45}^{2-})$
are pairwise distinct\emph;
\item
for~$k\in\{1,2,3\}$ the $\xi_-$-Legendrian classes~$S_+^k(9_{45}^{2-})$ and~$S_+^k\bigl(-\mu(9_{45}^{2-})\bigr)$
are distinct.
\end{enumerate}
\end{prop}

\begin{proof}
Representatives of the exchange classes involved in the proof are shown in Figure~\ref{9_45-rect-fig}.
It is established in~\cite{chong2013} that~$9_{45}^{2+}=-\mu(9_{45}^{2+})$. So, to
prove part~(i) of the proposition it suffices to show that~$9_{45}^{1+}$, $9_{45}^{3+}$, $-\mu(9_{45}^{1+})$,
and~$-\mu(9_{45}^{3+})$ are pairwise distinct. The proof of this and of part~(iii) is
presented in Figure~\ref{45_1-proof-fig}.
\begin{figure}[ht!]
\begin{center}
\begin{tikzpicture}[node distance = 2 cm,auto, ]
\node[] (1) {\footnotesize $9_{45}^{1\mathrm R}$} ;
\node[right = 2cm of 1] (4) {\footnotesize $9_{45}^{4\mathrm R}$} ;
\node[below right = 0.8cm and 0.8cm of 1] (dot1) {$\bullet$} ;
\node[below = 1.3cm of dot1](dot1v) {} ;
\node[above = 1.3cm of dot1](dot2v) {} ;
\node[left = 1.65cm of dot1](2-) {\footnotesize $9_{45}^{2-}$} ;
\node[below = 2cm of 1] (2) {\footnotesize $9_{45}^{2\mathrm R}$} ;
\node[right = 2cm of 2] (3) {\footnotesize $9_{45}^{3\mathrm R}$} ;
\node[below = 1cm of 3] (dot2) {$\bullet$} ;
\node[below = 0.5cm of dot2] (dot21) {$\bullet$} ;
\node[below = 0.5cm of dot21] (dot22) {$\bullet$} ;
\node[right = 1cm of dot2] (dot3) {$\bullet$} ;
\node[below = 2cm of dot3] (-mu3+) {\footnotesize $-\mu\bigl(9_{45}^{3+}\bigr)$} ;
\node[right = 1cm of 4] (dot5) {$\bullet$} ;
\node[right = 0.5cm of dot5] (dot6) {$\bullet$} ;
\node[above = 0.15cm of dot6] (dot6v1) {} ;
\node[above = 0.15cm of dot6v1] (-mu1+) {\footnotesize $-\mu\bigl(9_{45}^{1+}\bigr)$} ;
\node[below = 0.15cm of dot6] (dot6v2) {} ;
\node[right = 0.5cm of dot6] (dot7) {$\bullet$} ;
\node[right = 1cm of dot7] (-mu1) {\footnotesize $-\mu\bigl(9_{45}^{1\mathrm R}\bigr)$} ;
\node[right = 1.8cm of 3] (-mu2) {\footnotesize $-\mu\bigl(9_{45}^{2\mathrm R}\bigr)$} ;
\node[below = 0.97cm of -mu2] (dot8) {$\bullet$} ;
\node[below = 0.5cm of dot8] (dot81) {$\bullet$} ;
\node[below = 0.5cm of dot81] (dot82) {$\bullet$} ;
\node[left = 0.35cm of 1] (1v) {} ;
\node[left = 0.15cm of 1v] (1+) {\footnotesize $9_{45}^{1+}$} ;
\node[left = 0.35cm of 2] (3v) {} ;
\node[left = 0.15cm of 3v] (3+) {\footnotesize $9_{45}^{3+}$} ;
\node[above = 0cm of -mu2] (-mu2v) {} ;
\node[above = 0.15cm of -mu2v] (-mu2+) {\footnotesize $-\mu\bigl(9_{45}^{2-}\bigr)$} ;
\path[->,draw]
(1) edge node[midway,fill=white,above right = 1pt and 1pt,inner sep=0pt] {\footnotesize $\overrightarrow{\mathrm{II}}$} (dot1)
(4) edge node[midway,fill=white,above left = 1pt and 1pt,inner sep=0pt] {\footnotesize $\overrightarrow{\mathrm{II}}$} (dot1)
(2) edge node[midway,fill=white,below right = 1pt and 1pt,inner sep=0pt] {\footnotesize $\overrightarrow{\mathrm{II}}$} (dot1)
(3) edge node[midway,fill=white,below left = 1pt and 1pt,inner sep=0pt] {\footnotesize $\overrightarrow{\mathrm{II}}$} (dot1)
(4) edge node[midway,fill=white,above = 1pt ,inner sep=0pt] {\footnotesize $\overrightarrow{\mathrm{I}}$} (dot5)
(dot5) edge node[midway,fill=white,above = 1pt ,inner sep=0pt] 
{\footnotesize $\overleftarrow{\mathrm{I}}$} (dot6)
(dot7) edge node[midway,fill=white,above = 1pt ,inner sep=0pt] {\footnotesize $\overrightarrow{\mathrm{I}}$} (dot6)
(-mu1) edge node[midway,fill=white,above = 1pt ,inner sep=0pt] {\footnotesize $\overleftarrow{\mathrm{I}}$} (dot7)
(3) edge node[midway,fill=white,left = 1pt ,inner sep=0pt] {\footnotesize $\overrightarrow{\mathrm{I}}$} (dot2)
(dot2) edge node[midway,fill=white,left = 1pt ,inner sep=0pt] {\footnotesize $\overrightarrow{\mathrm{I}}$} (dot21)
(dot21) edge node[midway,fill=white,left = 1pt ,inner sep=0pt] {\footnotesize $\overrightarrow{\mathrm{I}}$} (dot22)
(dot2) edge node[midway,fill=white,below = 1pt ,inner sep=0pt] {\footnotesize $\overleftarrow{\mathrm{I}}$} (dot3)
(dot8) edge node[midway,fill=white,below = 1pt ,inner sep=0pt] {\footnotesize $\overleftarrow{\mathrm{I}}$} (dot3)
(-mu2) edge node[midway,fill=white,left = 1pt ,inner sep=0pt] {\footnotesize $\overrightarrow{\mathrm{I}}$} (dot8)
(dot8) edge node[midway,fill=white,left = 1pt ,inner sep=0pt] {\footnotesize $\overrightarrow{\mathrm{I}}$} (dot81)
(dot81) edge node[midway,fill=white,left = 1pt ,inner sep=0pt] {\footnotesize $\overrightarrow{\mathrm{I}}$} (dot82);
\begin{pgfonlayer}{background}
  \node[fit=(1)(2)(4)(3)(dot1v)(dot2v), vfit] {};
  \node[fit=(-mu2)(-mu2v),vfit]{};
  \node[fit=(3)(-mu2)(dot82)(dot22), hfit] {};
  \node[fit=(4)(-mu1)(dot6v1)(dot6v2), hfit] {};
  \node[fit=(1)(1v), hfit] {};
  \node[fit=(2)(3v), hfit] {};
\end{pgfonlayer}
\end{tikzpicture}
\end{center}
\caption{Proof of parts~(i) and~(iii) of Proposition~\ref{9_45_prop}} \label{45_1-proof-fig}
\end{figure}
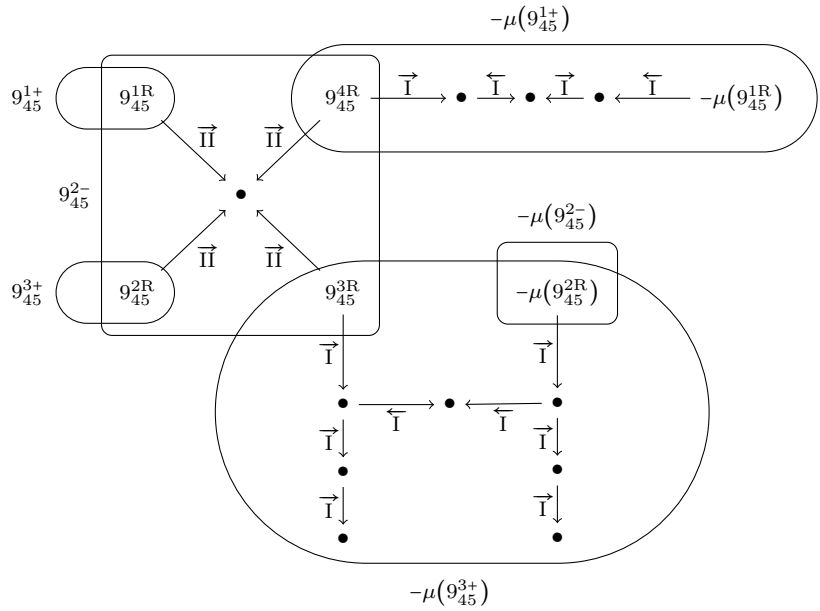
\begin{figure}[ht!]
\begin{center}
\begin{tikzpicture}[node distance = 2 cm,auto, ]
\node[] (6) {\footnotesize $9_{45}^{6\mathrm R}$} ;
\node[right = 1.5cm of 6] (dot1) {$\bullet$} ;
\node[above = 0.35cm of dot1] (dot1v) {} ;
\node[above = 0.15cm of dot1v] (1-) {\footnotesize $9_{45}^{1-}$} ;
\node[right = 1.5cm of dot1] (5) {\footnotesize $9_{45}^{5\mathrm R}$} ;
\node[right = 1cm of 5] (dot2) {$\bullet$} ;
\node[below = 1cm of dot2] (mu6) {\footnotesize $\mu\bigl(9_{45}^{6\mathrm R}\bigr)$} ;
\node[below = 0.35cm of mu6] (mu4v) {} ;
\node[below = 0.15cm of mu4v] (mu1-) {\footnotesize $\mu\bigl(9_{45}^{1-}\bigr)$} ;
\node[right = 0.5cm of dot2] (dot3) {$\bullet$} ;
\node[right = 0.5cm of dot3] (dot4) {$\bullet$} ;
\node[right = 1cm of dot4] (-6) {\footnotesize $-9_{45}^{6\mathrm R}$} ;
\node[above = 0.35cm of -6] (-4v) {} ;
\node[above = 0.15cm of -4v] (-1-) {\footnotesize $-9_{45}^{1-}$} ;
\path[->,draw]
(6) edge node[midway,fill=white,above = 1pt,inner sep=0pt] {\footnotesize $\overrightarrow{\mathrm{II}}$} (dot1)
(5) edge node[midway,fill=white,above = 1pt,inner sep=0pt] {\footnotesize $\overleftarrow{\mathrm{II}}$} (dot1)
(5) edge node[midway,fill=white,above = 1pt,inner sep=0pt] {\footnotesize $\overleftarrow{\mathrm{I}}$} (dot2)
(dot2) edge node[midway,fill=white,above = 1pt,inner sep=0pt] {\footnotesize $\overrightarrow{\mathrm{I}}$} (dot3)
(dot4) edge node[midway,fill=white,above = 1pt,inner sep=0pt] {\footnotesize $\overleftarrow{\mathrm{I}}$} (dot3)
(-6) edge node[midway,fill=white,above = 1pt,inner sep=0pt] {\footnotesize $\overrightarrow{\mathrm{I}}$} (dot4)
(mu6) edge node[midway,fill=white,right = 1pt,inner sep=0pt] {\footnotesize $\overleftarrow{\mathrm{I}}$} (dot2);
\begin{pgfonlayer}{background}
  \node[fit=(5)(-6)(mu6), hfit] {};
  \node[fit=(6)(5)(dot1v), vfit] {};
  \node[fit=(mu6)(mu4v), vfit] {};
  \node[fit=(-6)(-4v), vfit] {};
\end{pgfonlayer}
\end{tikzpicture}
\end{center}
\caption{Proof of part~(ii) of Proposition~\ref{9_45_prop}} \label{45_2-proof-fig}
\end{figure}
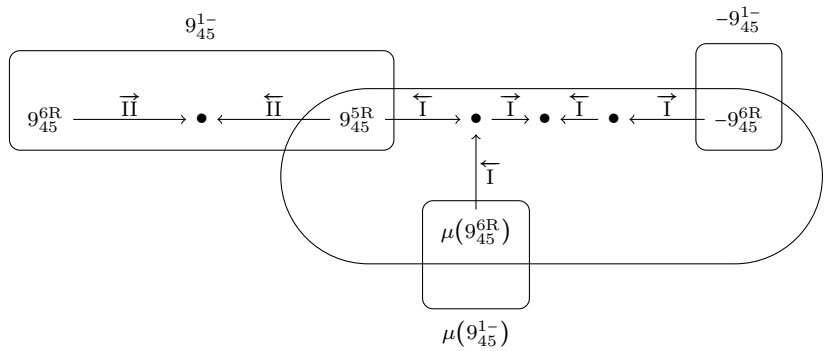

It is already established in~\cite{chong2013} that~$9_{45}^{2-}=-9_{45}^{2-}\notin\bigl\{
9_{45}^{1-},-9_{45}^{1-},\mu(9_{45}^{1-}),-\mu(9_{45}^{1-}),
\mu(9_{45}^{2-})\bigr\}$,
so it remains to show that~$9_{45}^{1-}$, $-9_{45}^{1-}$, $\mu(9_{45}^{1-})$, and~$-\mu(9_{45}^{1-})$
are pairwise distinct. To this end, it suffices to show that some three of these four classes are pairwise distinct.
This is done in Figure~\ref{45_2-proof-fig}.
\end{proof}

\begin{prop}\label{10_128_prop}
For the Legendrian classes whose representatives are shown in Figure~\ref{10_128-leg-fig} the following holds\emph:
\begin{enumerate}
\item
the Legendrian classes $10_{128}^{1+}$, $10_{128}^{2+}$, $-\mu\bigl(10_{128}^{1+}\bigr)$,
and~$-\mu\bigl(10_{128}^{2+}\bigr)$ are pairwise distinct\emph;
\item
for any~$k\in\{1,2,3,4\}$ the Legendrian classes~$S_-^k(10_{128}^{1+})=S_-^k(10_{128}^{2+})$
and~$S_-^k\bigl(-\mu(10_{128}^{1+})\bigr)=S_-^k\bigl(-\mu(10_{128}^{2+})\bigr)$ are distinct.
\end{enumerate}
\end{prop}

\begin{proof}
The proof is presented in Figure~\ref{128-proof-fig}.
\begin{figure}[ht!]
\includegraphics[scale=.5]{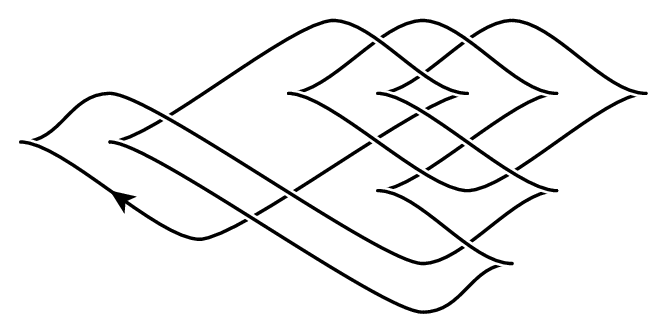}\put(-130,10){$10_{128}^{1+}$}
\hskip1cm
\includegraphics[scale=.5]{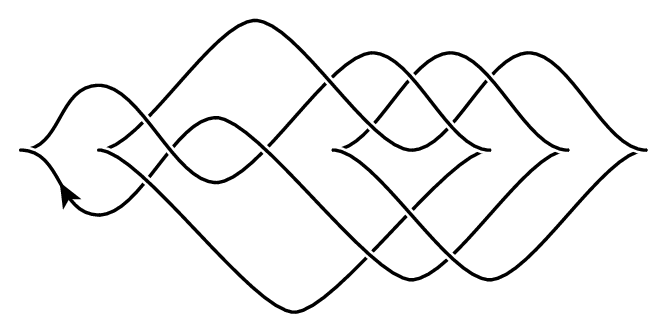}\put(-140,10){$10_{128}^{2+}$}
\caption{Knots in Proposition~\ref{10_128_prop}}\label{10_128-leg-fig}
\end{figure}
\begin{figure}[ht]
\includegraphics[scale=.18]{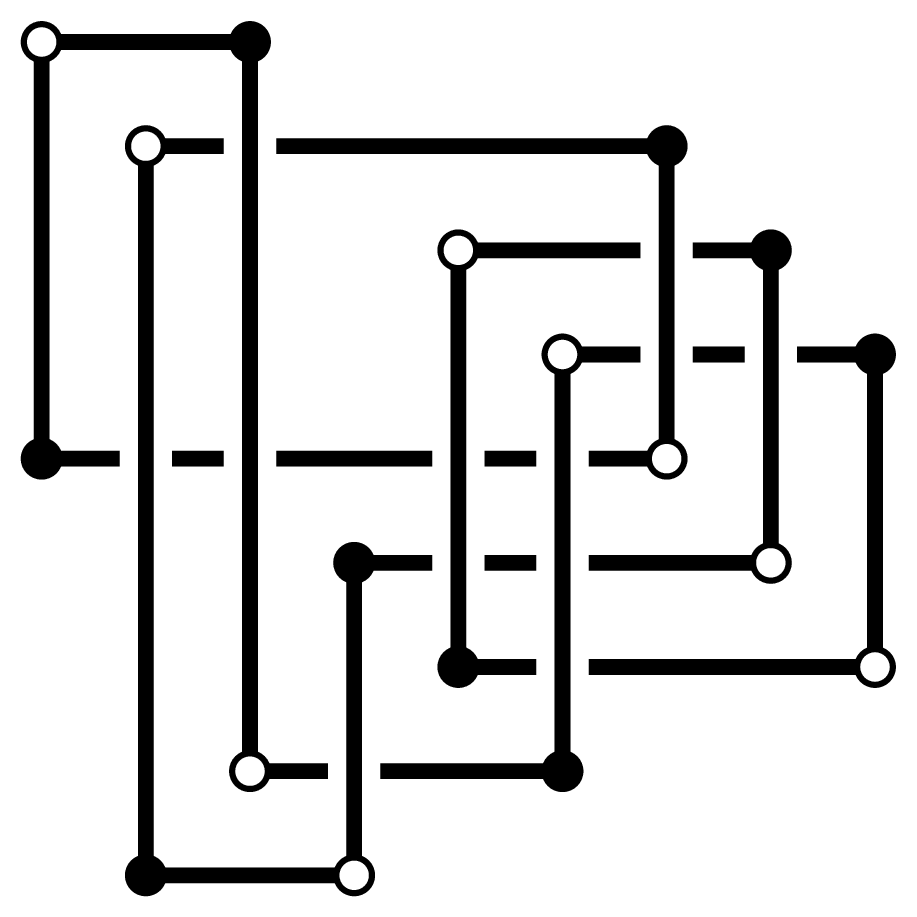}\put(-90,10){$10_{128}^{1\mathrm R}$}
\hskip1cm
\includegraphics[scale=.18]{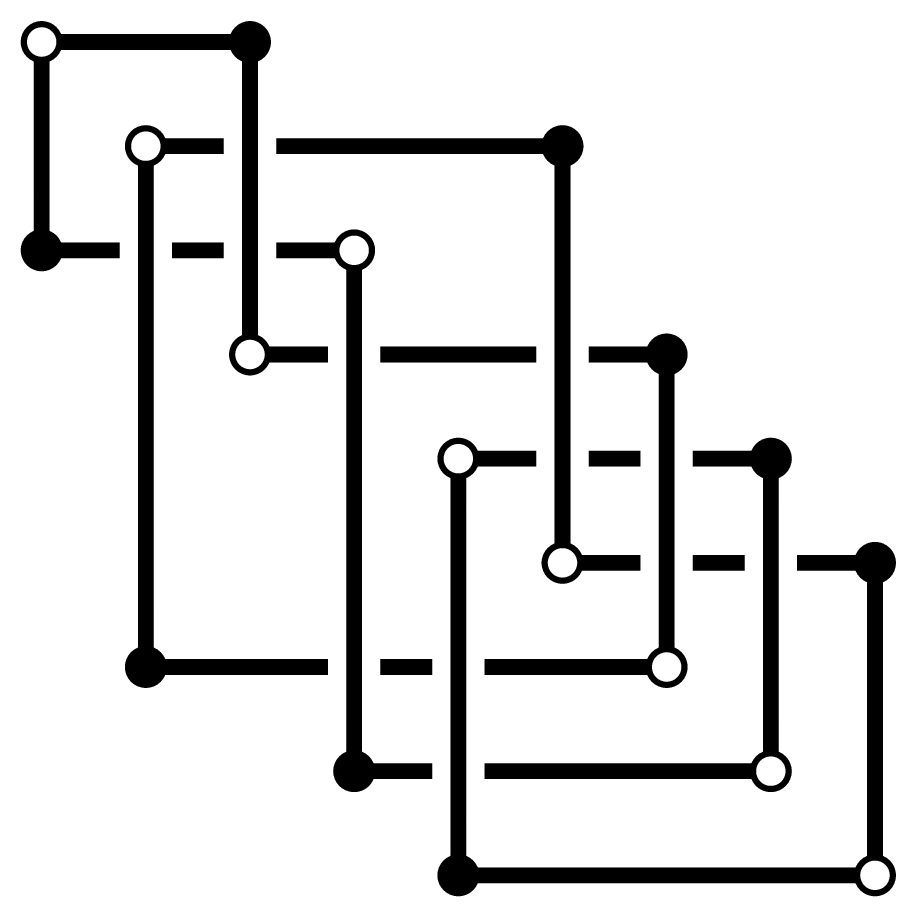}\put(-90,10){$10_{128}^{2\mathrm R}$}
\begin{center}
\begin{tikzpicture}[node distance = 2 cm,auto, ]
\node[] (1) {\footnotesize $10_{128}^{1\mathrm R}$} ;
\node[above = 0.15cm of 1] (1+) {\footnotesize $10_{128}^{1+}$} ;
\node[right = 1cm of 1] (dot1) {$\bullet$} ;
\node[below = 0.5cm of dot1] (dot11) {$\bullet$} ;
\node[below = 0.5cm of dot11] (dot12) {$\bullet$} ;
\node[below = 0.5cm of dot12] (dot13) {$\bullet$} ;
\node[right = 1cm of dot1] (2) {\footnotesize $10_{128}^{2\mathrm R}$} ;
\node[above = 0.15cm of 2] (2+) {\footnotesize $10_{128}^{2+}$} ;
\node[right = 1cm of 2] (dot2) {$\bullet$} ;
\node[right = 1cm of dot2] (-mu1) {\footnotesize $-\mu\bigl(10_{128}^{1\mathrm R}\bigr)$} ;
\node[above = 0.15cm of -mu1] (-mu1+) {\footnotesize $-\mu\bigl(10_{128}^{1+}\bigr)$} ;
\node[right = 1cm of -mu1] (dot3) {$\bullet$} ;
\node[below = 0.5cm of dot3] (dot31) {$\bullet$} ;
\node[below = 0.5cm of dot31] (dot32) {$\bullet$} ;
\node[below = 0.5cm of dot32] (dot33) {$\bullet$} ;
\node[right = 1cm of dot3] (-mu2) {\footnotesize $-\mu\bigl(10_{128}^{2\mathrm R}\bigr)$} ;
\node[above = 0.15cm of -mu2] (-mu2+) {\footnotesize $-\mu\bigl(10_{128}^{2+}\bigr)$} ;
\path[->,draw]
(1) edge node[midway,above = 1pt,fill=white,inner sep=0pt] {\footnotesize $\overrightarrow{\mathrm{II}}$} (dot1)
(2) edge node[midway,above = 1pt,fill=white,inner sep=0pt] {\footnotesize $\overrightarrow{\mathrm{II}}$} (dot1)
(2) edge node[midway,above = 1pt,fill=white,inner sep=0pt] {\footnotesize $\overleftarrow{\mathrm{II}}$} (dot2)
(-mu1) edge node[midway,above = 1pt,fill=white,inner sep=0pt] {\footnotesize $\overleftarrow{\mathrm{II}}$} (dot2)
(-mu1) edge node[midway,above = 1pt,fill=white,inner sep=0pt] {\footnotesize $\overrightarrow{\mathrm{II}}$} (dot3)
(-mu2) edge node[midway,above = 1pt,fill=white,inner sep=0pt] {\footnotesize $\overrightarrow{\mathrm{II}}$} (dot3)
(dot1) edge node[midway,right = 1pt,fill=white,inner sep=0pt] {\footnotesize $\overrightarrow{\mathrm{II}}$} (dot11)
(dot11) edge node[midway,right = 1pt,fill=white,inner sep=0pt] {\footnotesize $\overrightarrow{\mathrm{II}}$} (dot12)
(dot12) edge node[midway,right = 1pt,fill=white,inner sep=0pt] {\footnotesize $\overrightarrow{\mathrm{II}}$} (dot13)
(dot3) edge node[midway,right = 1pt,fill=white,inner sep=0pt] {\footnotesize $\overrightarrow{\mathrm{II}}$} (dot31)
(dot31) edge node[midway,right = 1pt,fill=white,inner sep=0pt] {\footnotesize $\overrightarrow{\mathrm{II}}$} (dot32)
(dot32) edge node[midway,right = 1pt,fill=white,inner sep=0pt] {\footnotesize $\overrightarrow{\mathrm{II}}$} (dot33);
\begin{pgfonlayer}{background}
  \node[fit=(1), hfit] {};
  \node[fit=(2), hfit] {};
  \node[fit=(-mu1), hfit] {};
  \node[fit=(-mu2), hfit] {};
\end{pgfonlayer}

\end{tikzpicture}
\end{center}
\caption{Proof of Proposition~\ref{10_128_prop}} \label{128-proof-fig}
\end{figure}
\end{proof}

\begin{prop}\label{10_160_prop}
For the Legendrian classes whose representatives are shown in Figure~\ref{10_160-leg-fig} the following holds\emph:
\begin{enumerate}
\item
the classes
$10_{160}^{1+}{}=-10_{160}^{1+}$, $\mu\bigl(10_{160}^{1+}\bigr)$, $10_{160}^{2+}$, $-10_{160}^{2+}$, $ \mu\bigl(10_{160}^{2+}\bigr)$,
and~$-\mu\bigl(10_{160}^{2+}\bigr)$ are pairwise distinct\emph;
\item
for any~$k\in\{1,2,3,4\}$ the classes~$S_-^k(10_{160}^{2+})$ and~$S_-^k(-10_{160}^{2+})$ are distinct.
\end{enumerate}
\end{prop}
\begin{proof}
The proof is presented in Figure~\ref{160-proof-fig}. (The~$\xi_-$-Legendrian class~$10_{160}^-$
can be guessed from the scheme. We don't provide a picture as this class is not involved in any
of our statements.)
\begin{figure}[ht!]
\includegraphics[scale=.5]{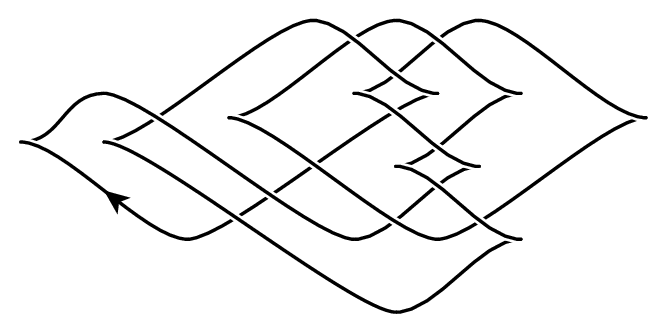}\put(-130,10){$10_{160}^{1+}$}
\hskip1cm
\includegraphics[scale=.5]{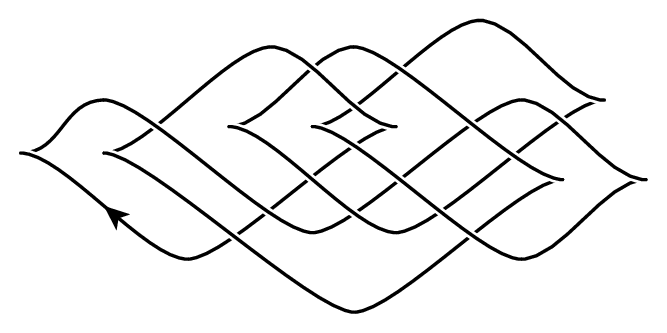}\put(-140,10){$10_{160}^{2+}$}
\caption{Knots in Proposition~\ref{10_160_prop}}\label{10_160-leg-fig}
\end{figure}
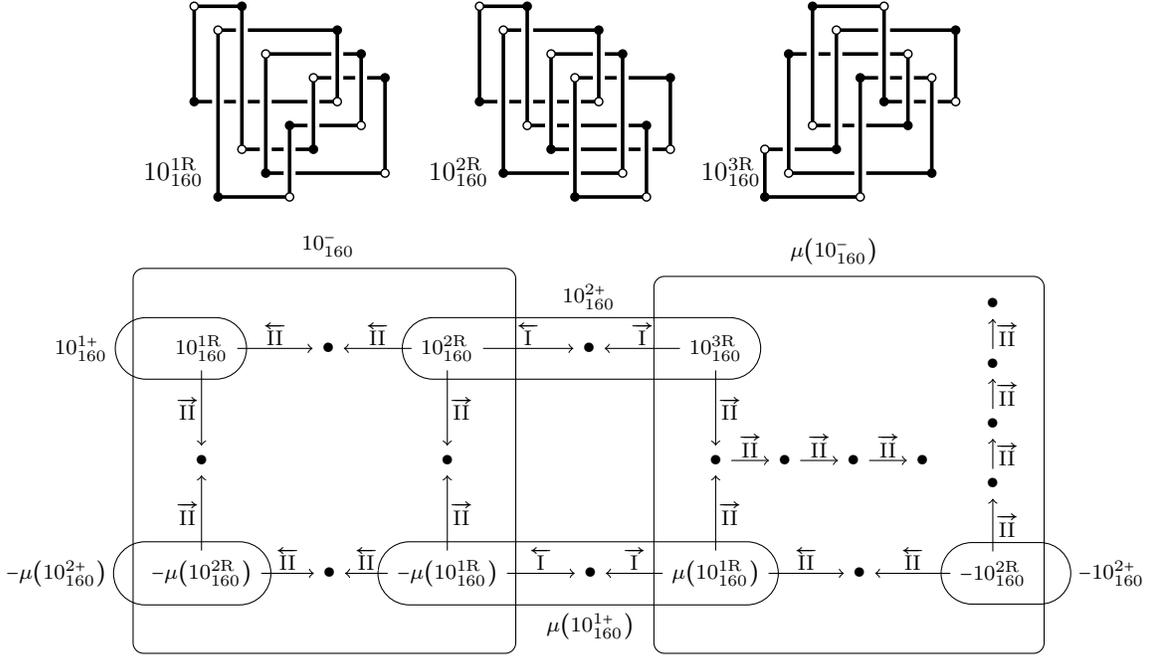
\begin{figure}[ht]
\includegraphics[scale=.18]{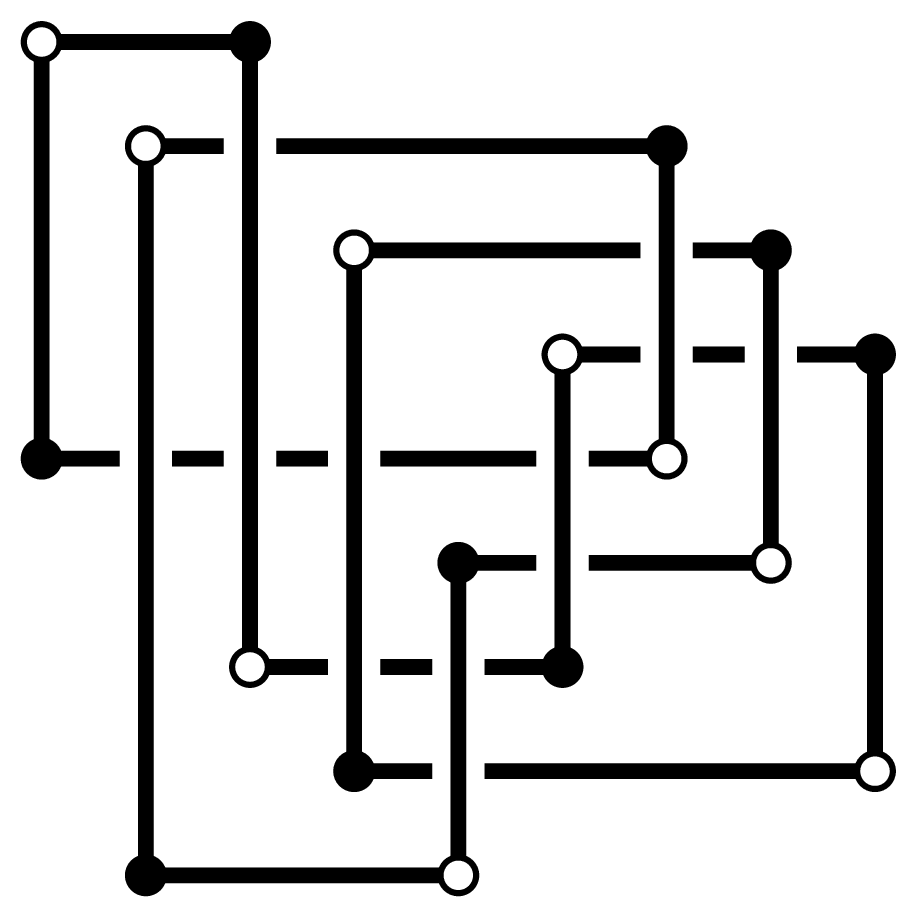}\put(-95,10){$10_{160}^{1\mathrm R}$}
\hskip1cm
\includegraphics[scale=.18]{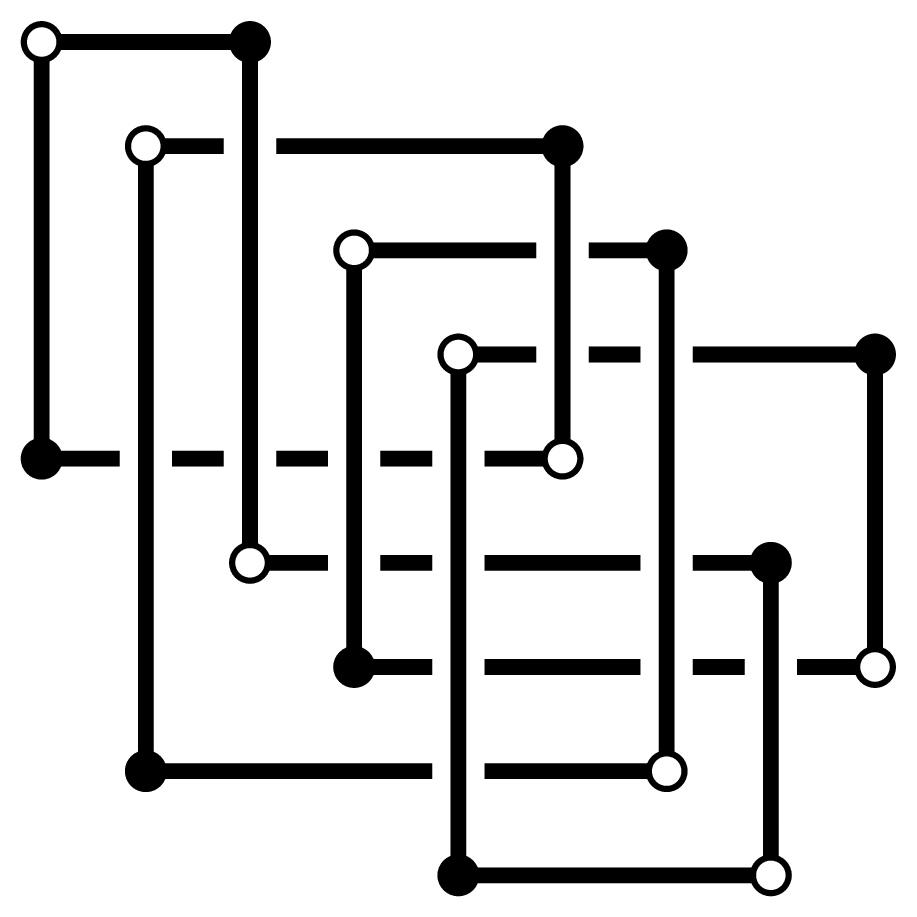}\put(-95,10){$10_{160}^{2\mathrm R}$}
\hskip1cm
\includegraphics[scale=.18]{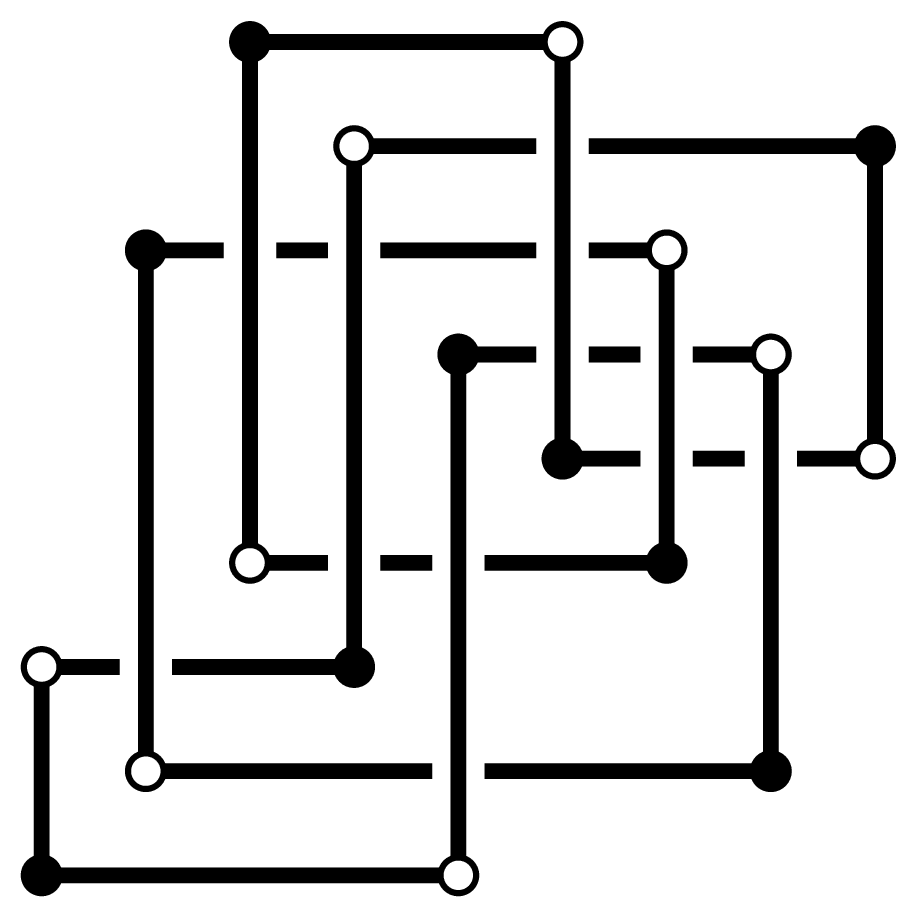}\put(-100,10){$10_{160}^{3\mathrm R}$}
\begin{center}
\begin{tikzpicture}[node distance = 2 cm,auto, ]
\node[] (1) {\footnotesize $10_{160}^{1\mathrm R}$} ;
\node[left = 0.3cm of 1] (1v) {} ;
\node[left = 0.1cm of 1v] (1+) {\footnotesize $10_{160}^{1+}$} ;
\node[right = 1cm of 1] (dot1) {$\bullet$} ;
\node[above = 0.9cm of dot1] (1-) {\footnotesize $10_{160}^{-}$} ;
\node[below = 1cm of 1] (dot2) {$\bullet$} ;
\node[below = 1cm of dot2] (-mu2) {\footnotesize $-\mu\bigl(10_{160}^{2\mathrm R}\bigr)$} ;
\node[left = 0.0cm of -mu2] (-mu2v) {} ;
\node[left = 0.1cm of -mu2v] (-mu2+) {\footnotesize $-\mu\bigl(10_{160}^{2+}\bigr)$} ;
\node[right = 0.9cm of dot1] (2) {\footnotesize $10_{160}^{2\mathrm R}$} ;
\node[below = 1cm of 2] (dot3) {$\bullet$} ;
\node[below = 1cm of dot3] (-mu1) {\footnotesize $-\mu\bigl(10_{160}^{1\mathrm R}\bigr)$} ;
\node[right = 0.7cm of -mu2] (dot4) {$\bullet$} ;
\node[right = 1.2cm of 2] (dot5) {$\bullet$} ;
\node[above = 0.2cm of dot5] (2+) {\footnotesize $10_{160}^{2+}$};
\node[right = 1cm of dot5] (3) {\footnotesize $10_{160}^{3\mathrm R}$} ;
\node[right = 1cm of 3] (dot10) {$\bullet$} ;
\node[right = 1cm of dot10] (mu2) {\footnotesize $\mu\bigl(10_{160}^{2\mathrm R}\bigr)$} ;
\node[right = 0.1cm of mu2] (mu2v) {} ;
\node[right = 0.1cm of mu2v] (mu2+) {\footnotesize $\mu\bigl(10_{160}^{2+}\bigr)$} ;
\node[below = 1cm of 3] (dot7) {$\bullet$} ;
\node[right = 0.5cm of dot7] (dot71) {$\bullet$} ;
\node[right = 0.5cm of dot71] (dot72) {$\bullet$} ;
\node[above = 0.4cm of dot72] (dot73) {$\bullet$} ;
\node[right = 0.9cm of -mu1] (dot6) {$\bullet$} ;
\node[below = 1cm of dot7] (mu1) {\footnotesize $\mu\bigl(10_{160}^{1\mathrm R}\bigr)$} ;
\node[below = 0.2cm of dot6] (mu1+) {\footnotesize $\mu\bigl(10_{160}^{1+}\bigr)$} ;
\node[below = -0.2cm of mu1+] (-mu1+) {\footnotesize $|\hspace{0.3mm}|$} ;
\node[below = 0.2cm of mu1+] (-mu1+) {\footnotesize $-\mu\bigl(10_{160}^{1+}\bigr)$} ;
\node[right = 1cm of mu1] (dot8) {$\bullet$} ;
\node[right = 1cm of dot8] (-2) {\footnotesize $-10_{160}^{2\mathrm R}$} ;
\node[above = 0.7cm of -2] (dot9) {$\bullet$} ;
\node[right = 0.1cm of -2] (-2v) {} ;
\node[right = 0.1cm of -2v] (-2+) {\footnotesize $-10_{160}^{2+}$} ;
\node[left = 0.5cm of dot9] (dot91) {$\bullet$} ;
\node[above = 0.5cm of dot91] (dot92) {$\bullet$} ;
\node[right = 0.5cm of dot92] (dot93) {$\bullet$} ;
\node[above = 0.5cm of dot1] (dot1v) {} ;
\node[below = 0.5cm of dot4] (dot4v) {} ;
\node[below = 0.5cm of dot8] (dot8v) {} ;
\node[above right = 0.3cm and 2cm of 3] (3v) {} ;
\node[above left = 0.15cm and 0.2cm of 3v] (mu1-) {\footnotesize $\mu\bigl(10_{160}^{-}\bigr)$} ;
\node[left = 0.0cm of dot2] (dot2v) {} ;
\path[->,draw]
(1) edge node[midway,above = 1pt,fill=white,inner sep=0pt] {\footnotesize $\overleftarrow{\mathrm{II}}$} (dot1)
(1) edge node[midway,left = 1pt,fill=white,inner sep=0pt] {\footnotesize $\overrightarrow{\mathrm{II}}$} (dot2)
(-mu2) edge node[midway,left = 1pt,fill=white,inner sep=0pt] {\footnotesize $\overrightarrow{\mathrm{II}}$} (dot2)
(2) edge node[midway,above = 1pt,fill=white,inner sep=0pt] {\footnotesize $\overleftarrow{\mathrm{II}}$} (dot1)
(2) edge node[midway,right = 1pt,fill=white,inner sep=0pt] {\footnotesize $\overrightarrow{\mathrm{II}}$} (dot3)
(-mu1) edge node[midway,right = 1pt,fill=white,inner sep=0pt] {\footnotesize $\overrightarrow{\mathrm{II}}$} (dot3)
(-mu2) edge node[midway,above = 1pt,fill=white,inner sep=0pt] {\footnotesize $\overleftarrow{\mathrm{II}}$} (dot4)
(-mu1) edge node[midway,above = 1pt,fill=white,inner sep=0pt] {\footnotesize $\overleftarrow{\mathrm{II}}$} (dot4)
(2) edge node[midway,above = 1pt,fill=white,inner sep=0pt] {\footnotesize $\overleftarrow{\mathrm{I}}$} (dot5)
(3) edge node[midway,above = 1pt,fill=white,inner sep=0pt] {\footnotesize $\overrightarrow{\mathrm{I}}$} (dot5)
(-mu1) edge node[midway,above = 1pt,fill=white,inner sep=0pt] {\footnotesize $\overleftarrow{\mathrm{I}}$} (dot6)
(mu1) edge node[midway,above = 1pt,fill=white,inner sep=0pt] {\footnotesize $\overrightarrow{\mathrm{I}}$} (dot6)
(mu1) edge node[midway,right= 1pt,fill=white,inner sep=0pt] {\footnotesize $\overrightarrow{\mathrm{II}}$} (dot7)
(3) edge node[midway,right= 1pt ,fill=white,inner sep=0pt] {\footnotesize $\overrightarrow{\mathrm{II}}$} (dot7)
(dot7) edge node[midway,above = 1pt,fill=white,inner sep=0pt] {\footnotesize $\overrightarrow{\mathrm{II}}$} (dot71)
(dot71) edge node[midway,above = 1pt,fill=white,inner sep=0pt] {\footnotesize $\overrightarrow{\mathrm{II}}$} (dot72)
(dot72) edge node[midway,right = 1pt,fill=white,inner sep=0pt] {\footnotesize $\overrightarrow{\mathrm{II}}$} (dot73)
(mu1) edge node[midway,above= 1pt,fill=white,inner sep=0pt] {\footnotesize $\overleftarrow{\mathrm{II}}$} (dot8)
(-2) edge node[midway,above= 1pt,fill=white,inner sep=0pt] {\footnotesize $\overleftarrow{\mathrm{II}}$} (dot8)
(-2) edge node[midway,right= 1pt,fill=white,inner sep=0pt] {\footnotesize $\overrightarrow{\mathrm{II}}$} (dot9)
(dot9) edge node[midway,below= 1pt,fill=white,inner sep=0pt] {\footnotesize $\overrightarrow{\mathrm{II}}$} (dot91)
(dot91) edge node[midway,right= 1pt,fill=white,inner sep=0pt] {\footnotesize $\overrightarrow{\mathrm{II}}$} (dot92)
(dot92) edge node[midway,above= 1pt,fill=white,inner sep=0pt] {\footnotesize $\overrightarrow{\mathrm{II}}$} (dot93)
(3) edge node[midway,above= 1pt,fill=white,inner sep=0pt] {\footnotesize $\overleftarrow{\mathrm{II}}$} (dot10)
(mu2) edge node[midway,above= 1pt,fill=white,inner sep=0pt] {\footnotesize $\overleftarrow{\mathrm{II}}$} (dot10);

\begin{pgfonlayer}{background}
  \node[fit=(3)(-2)(mu1)(dot72)(dot92)(3v)(dot8v), vfit] {};
  \node[fit=(1)(2)(-mu2)(-mu1)(dot1v)(dot4v), vfit] {};
  \node[fit=(-mu1)(mu1), hfit] {};
  \node[fit=(2)(3), hfit] {};
  \node[fit=(1)(1v), hfit] {};
  \node[fit=(-mu2)(-mu2v), hfit] {};
  \node[fit=(-2)(-2v), hfit] {};
  \node[fit=(mu2)(mu2v), hfit] {};
\end{pgfonlayer}
\end{tikzpicture}
\end{center}
\caption{Proof of Proposition~\ref{10_160_prop}} \label{160-proof-fig}
\end{figure}
\end{proof}

\begin{rema}
The fact that~$10_{160}^{1+}\notin\bigl\{10_{160}^{2+},-10_{160}^{2+},\mu(10_{160}^{2+}),-\mu(10_{160}^{2+})\bigr\}$
and~$10_{160}^{1+}=-10_{160}^{1+}$ is established already in~\cite{chong2013}.
\end{rema}

\def\theenumi{\arabic{enumi}}
\begin{proof}[Proof of Theorem~\ref{not-equiv-thm}]
The front projections of~$K_1$ and~$K_2$ shown in Figure~\ref{monster-knots-fig} are produced from
two rectangular diagrams~$R_1$ and~$R_2$, respectively, via the procedure described in Section~\ref{rd_of_knots-sec}
and illustrated in Figure~\ref{associated-legendrian-fig}. Thus, we have~$K_i\in\mathscr L_+(R_i)$, $i=1,2$.

Now we recall the origin of~$R_1$, $R_2$. Shown in Figure~35 of~\cite{representability} is
a rectangular diagram~$\Pi$ of a surface such that:
\begin{enumerate}
\item
the associated surface~$\widehat\Pi$ is an annulus;
\item
the relative Thurston--Bennequin numbers~$\tb(\widehat R_i;\widehat\Pi)$, $i=1,2$, vanish;
\item
$\widehat\Pi$ can be endowed with an orientation so that~$\partial\widehat\Pi=\widehat R_1\cup(-\widehat R_2)$;
\item
$\Pi$ has the form~$\{r_i\}_{i=1,2,\ldots,74}$, where, for each~$i=1,\ldots,74$ the intersection~$r_{i-1}\cap r_i$
is the bottom left vertex of~$r_i$ (we put~$r_0=r_{74}$).
\end{enumerate}

The last condition in this list means that there are~$\theta_0,\theta_1,\ldots,\theta_{74}=\theta_0\in\mathbb S^1$
and~$\varphi_0,\varphi_1,\ldots,\varphi_{74}=\varphi_0\in\mathbb S^1$ such
that~$r_i=[\theta_{i-1};\theta_i]\times[\varphi_{i-1};\varphi_i]$
and~$R_1\cup R_2=\{(\theta_{i-1},\varphi_i),(\theta_i,\varphi_{i-1})\}_{i=1,\ldots,74}$.
Moreover, the signs of the vertices~$(\theta_{i-1},\varphi_i)$ and $(\theta_i,\varphi_{i-1})$ in~$R_1\cup R_2$
are opposite.

We now show that a sequence of elementary moves including a type~II stabilization, exchange moves, and a type~II destabilization
transforms~$R_1\cup R_2$ to a rectangular diagram of a link in which the connected components become
combinatorially equivalent. To this end, pick an~$\varepsilon>0$ smaller than one half of
the length of any interval~$[\theta_i;\theta_j]$ and~$[\varphi_i;\varphi_j]$, $i\ne j$, and make the following
replacements in~$R_1\cup R_2$:
$$\begin{aligned}
(\theta_1,\varphi_0)&\rightsquigarrow
(\theta_0-\varepsilon,\varphi_0),(\theta_0-\varepsilon,\varphi_1-\varepsilon),(\theta_1,\varphi_1-\varepsilon)
&&\text{(type~II stabilization)},\\
(\theta_1,\varphi_1-\varepsilon),(\theta_1,\varphi_2)
&\rightsquigarrow
(\theta_2-\varepsilon,\varphi_1-\varepsilon),(\theta_2-\varepsilon,\varphi_2)
&&\text{(exchange)},\\
(\theta_2-\varepsilon,\varphi_2),(\theta_3,\varphi_2)
&\rightsquigarrow
(\theta_2-\varepsilon,\varphi_3-\varepsilon),(\theta_3,\varphi_3-\varepsilon)
&&\text{(exchange)},\\
(\theta_3,\varphi_3-\varepsilon),(\theta_3,\varphi_4)
&\rightsquigarrow
(\theta_4-\varepsilon,\varphi_3-\varepsilon),(\theta_4-\varepsilon,\varphi_4)
&&\text{(exchange)},\\
&\ldots\\
(\theta_{72}-\varepsilon,\varphi_{72}),(\theta_{73},\varphi_{72})
&\rightsquigarrow
(\theta_{72}-\varepsilon,\varphi_{73}-\varepsilon),(\theta_{73},\varphi_{73}-\varepsilon)
&&\text{(exchange)},\\
(\theta_{73},\varphi_{73}-\varepsilon),(\theta_{73},\varphi_0),(\theta_0+\varepsilon,\varphi_0)
&\rightsquigarrow
(\theta_0-\varepsilon,\varphi_{73}-\varepsilon)
&&\text{(type~II destabilization)}.
\end{aligned}$$
This sequence of moves is illustrated in Figure~\ref{r1->r2-fig}.
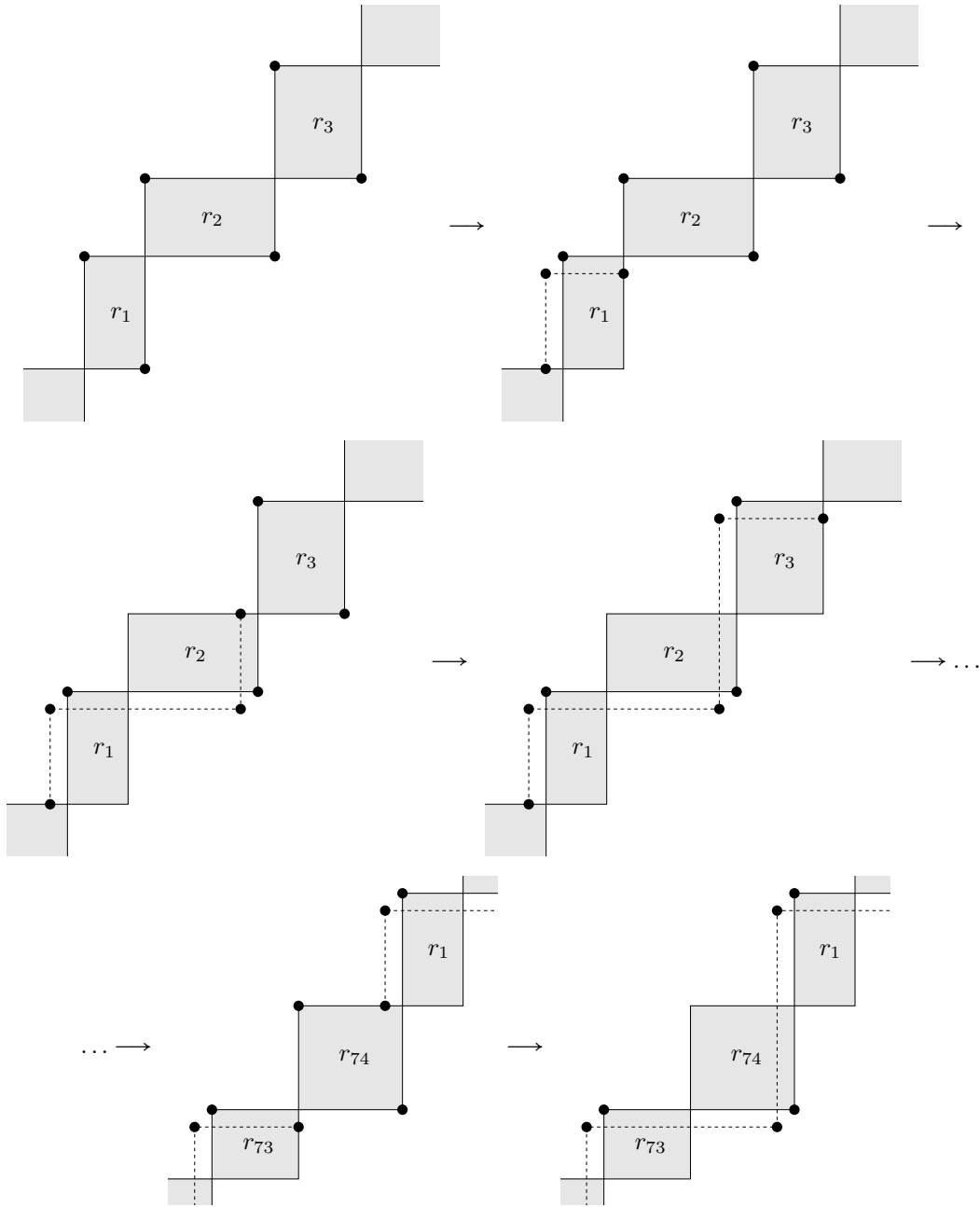
\begin{figure}[ht]
\includegraphics[scale=.35]{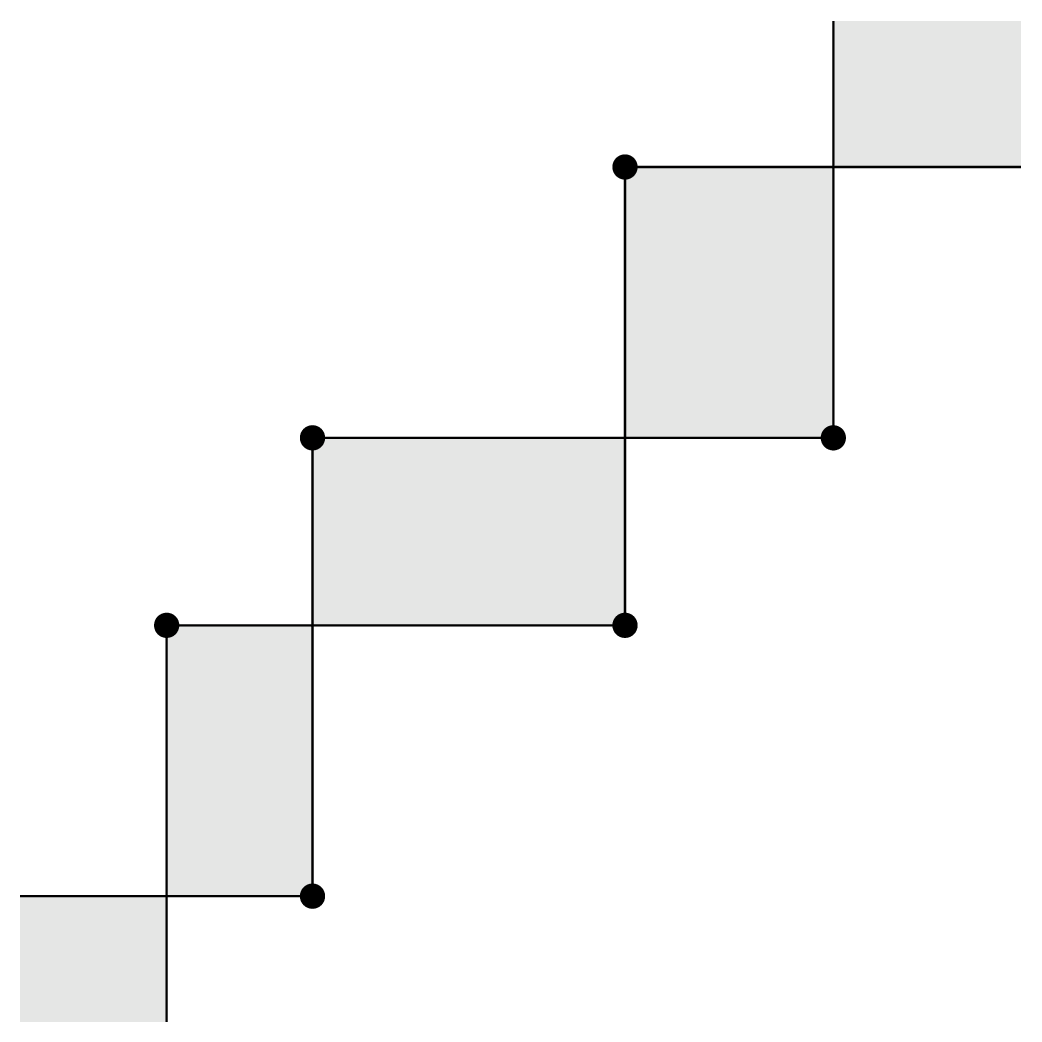}\put(-137,45){$r_1$}\put(-100,84){$r_2$}\put(-55,122){$r_3$}
\raisebox{80pt}{$\longrightarrow$}
\includegraphics[scale=.35]{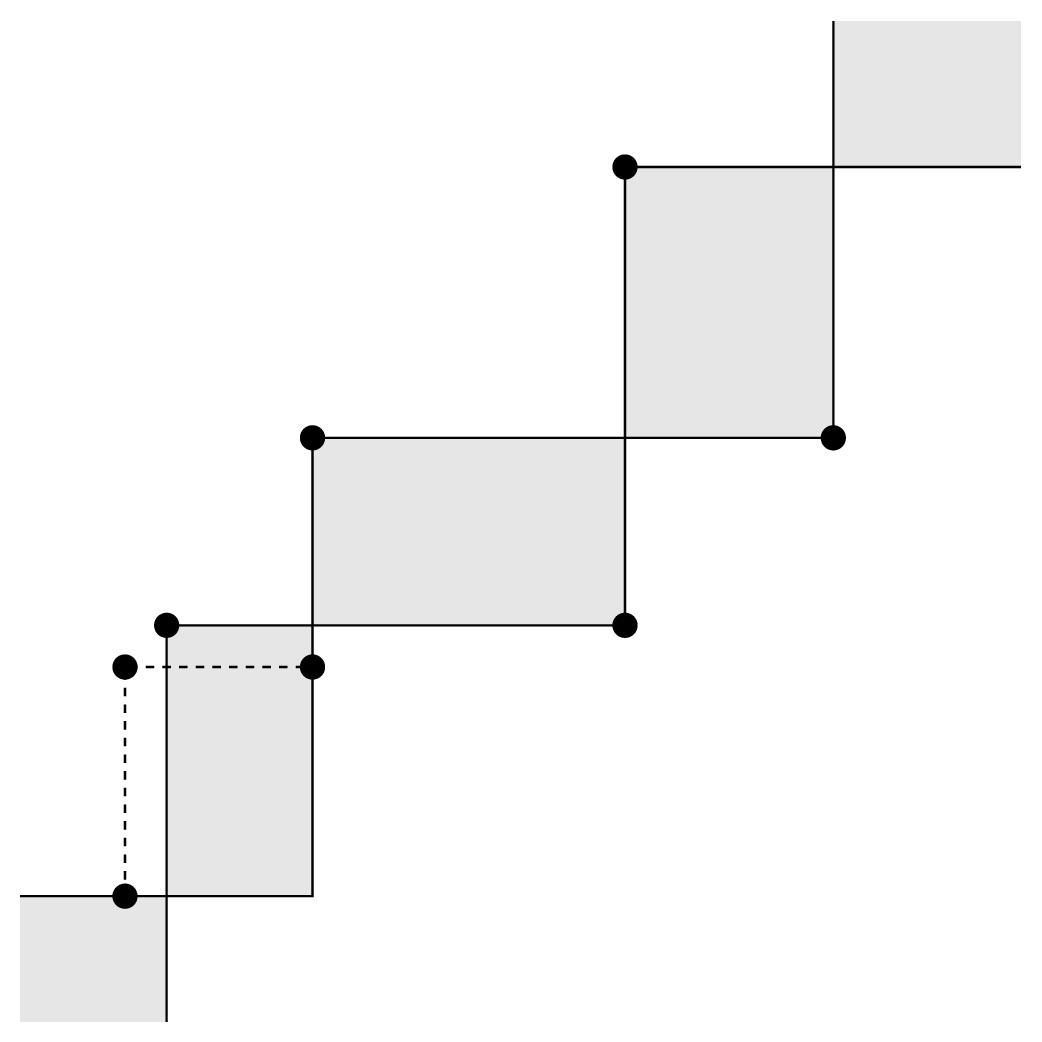}\put(-137,45){$r_1$}\put(-100,84){$r_2$}\put(-55,122){$r_3$}
\raisebox{80pt}{$\longrightarrow$}\\
\includegraphics[scale=.35]{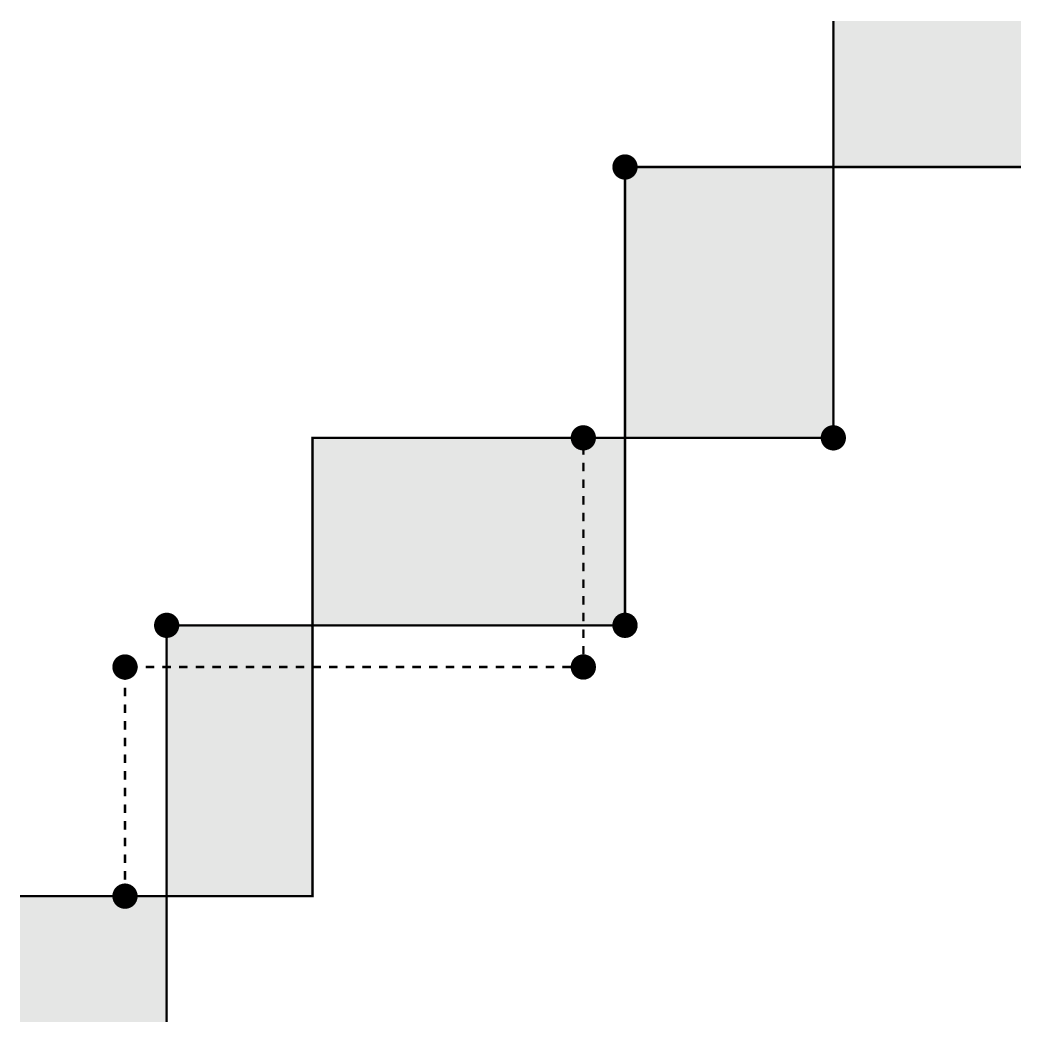}\put(-137,45){$r_1$}\put(-100,84){$r_2$}\put(-55,122){$r_3$}
\raisebox{80pt}{$\longrightarrow$}
\includegraphics[scale=.35]{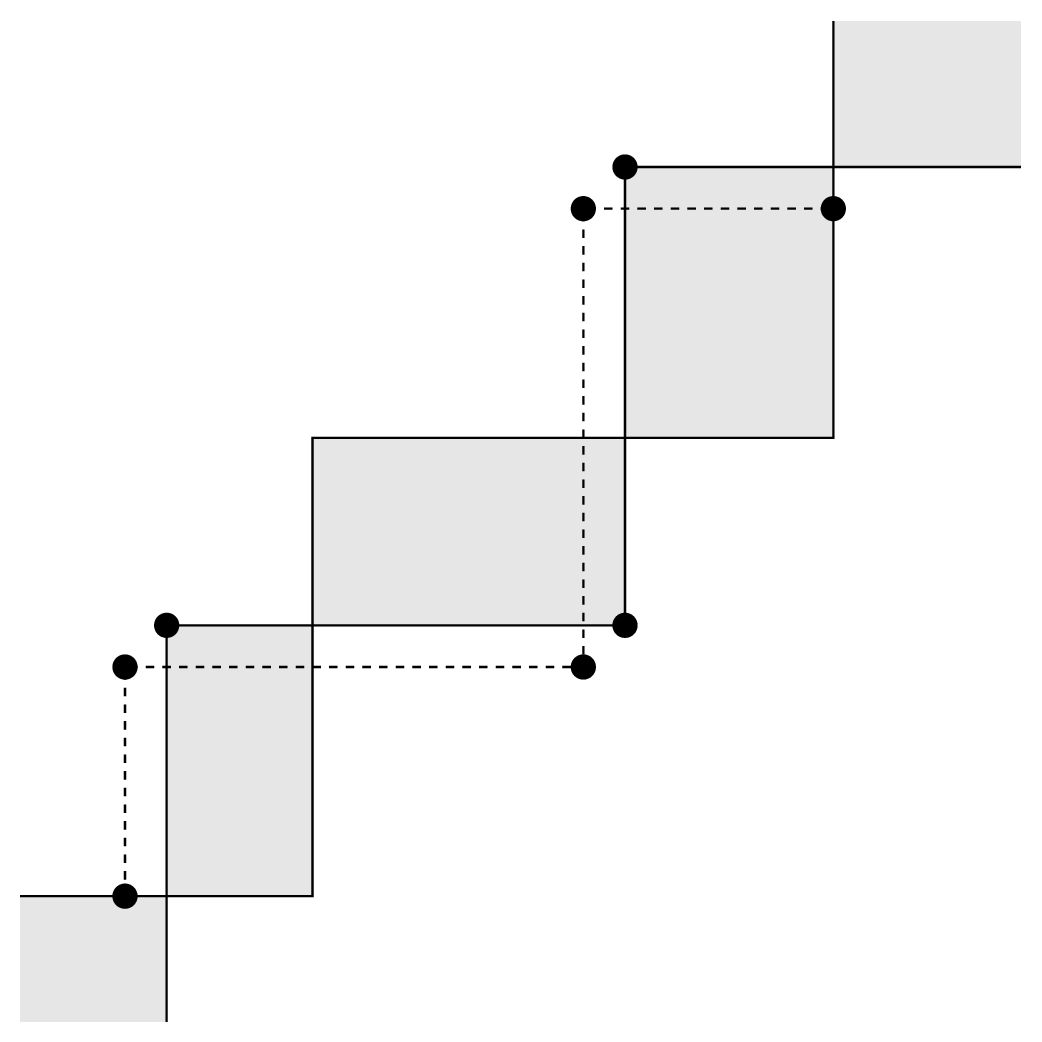}\put(-137,45){$r_1$}\put(-100,84){$r_2$}\put(-55,122){$r_3$}
\raisebox{80pt}{$\longrightarrow\ldots$}\\
\raisebox{65pt}{$\ldots\longrightarrow$}
\includegraphics[scale=.35]{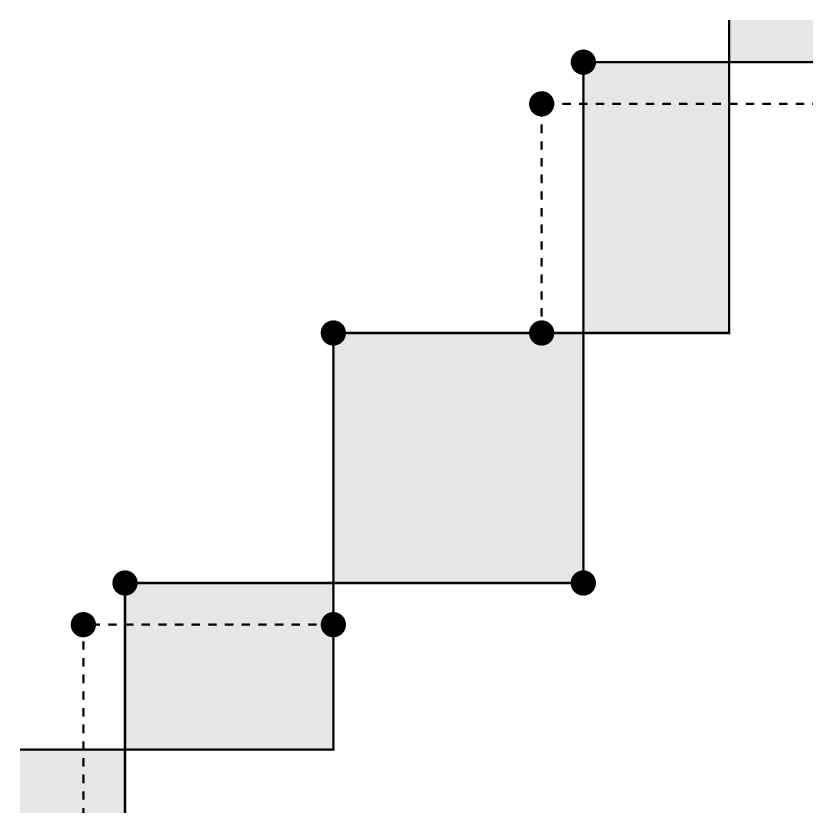}\put(-107,26){$r_{73}$}\put(-68,62){$r_{74}$}\put(-32,105){$r_1$}
\raisebox{65pt}{$\longrightarrow$}
\includegraphics[scale=.35]{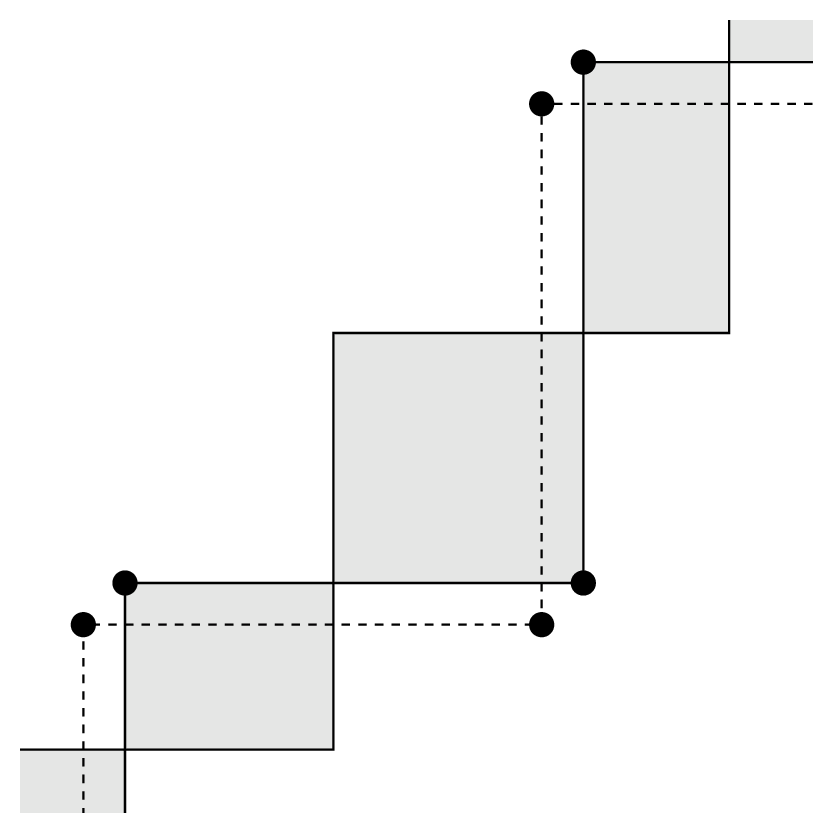}\put(-107,26){$r_{73}$}\put(-68,62){$r_{74}$}\put(-32,105){$r_1$}
\caption{Transforming one of~$R_1$ and~$R_2$ to the other by elementary moves}\label{r1->r2-fig}
\end{figure}

This proves that~$\mathscr L_-(R_1)=\mathscr L_-(R_2)$. The diagrams~$R_1$ and~$R_2$ are not combinatorially equivalent
and do not admit any non-trivial exchange move.
The knots represented by~$R_1$ and~$R_2$ have trivial orientation-preserving symmetry group by
Proposition~\ref{monster-knot-has-trivial-group-prop}.
Therefore, by Theorem~\ref{main-theo}, $\mathscr L_+(R_1)\ne\mathscr L_+(R_2)$.
\end{proof}

\section{Appendix: $K_1$ and~$K_2$ are not satellite knots}\label{append-sec}
Here we explain how to verify, with very little computations, that the complement of~$K_1$ (and $K_2$) contains no incompressible non-boundary-parallel torus.
To do so we use a method that can be viewed as a modification of Haken's method of normal surfaces, which allows one, in general, to find
all incompressible surfaces of minimal genus. Haken's algorithm in general
has very high computational complexity, which makes it infeasible to implement in most cases.
However, in certain cases including our particular one, a modified version of Haken's method can be efficiently used to search all incompressible
surfaces of \emph{non-negative Euler characteristic}.

First we describe the general idea for the reader well familiar with the difficulties
in using Haken's method in practice.
Haken's normal surfaces are encoded by certain \emph{normal coordinates} $x_1,\ldots,x_N$,
which take integer values. To determine a normal surface they must satisfy a bunch of conditions that are naturally
partitioned into the following three groups:
\begin{enumerate}
\item
\emph{non-negativity conditions}, which are the inequalities~$x_i\geqslant0$, $i=1,\ldots,N$;
\item
\emph{matching conditions}, which are linear equations with integer coefficients;
\item
\emph{compatibility conditions}, which are equations of the form~$x_ix_j=0$ for some set of pairs~$(i,j)$.
\end{enumerate}

The Euler characteristic of a normal surface~$F$ can be expressed as a linear combination
of the normal coordinates of~$F$ in numerous ways, and some of these expressions
have only non-positive coefficients.
If we are looking for normal surfaces of non-negative Euler
characteristic, for any such expression~$\sum_ia_ix_i$ with non-positive~$a_i$'s,
we may add the inequality~$\sum_ia_ix_i\geqslant0$ to the system.
Together with the non-negativity conditions this implies~$x_i=0$ whenever~$a_i<0$.
This reduces the number of variables in the system, and chances are that, after
the reduction, the space of solutions of the system of matching equations alone
has very small dimension.

Now we turn to our concrete case. The idea explained above will be realized
in quite different terms. The reduction of variables will occur in Lemma~\ref{maximal-rect-lem}.

The rectangular diagrams from which
the Legendrian knots~$K_1$ and~$K_2$ shown in Figure~\ref{monster-knots-fig}
are produced have~$37$ edges
of each direction. For this reason we rescale the coordinates~$\theta,\varphi$
on~$\mathbb T^2$ so that they take values in~$\mathbb R/(37\cdot\mathbb Z)$,
and the vertices of the diagrams will form a subset of~$\mathbb Z_{37}\times\mathbb Z_{37}$.

We will work with the knot~$K_1$. The corresponding rectangular diagram of a knot,
which we denote by~$R$, has the following
list of vertices:
\begin{multline*}
(0, 13), (0, 28), (1, 14), (1, 35), (2, 15), (2, 36), (3, 0), (3, 19), (4, 1), (4, 22), (5, 6), (5, 23), (6, 7), (6, 24),\\
 (7, 9), (7, 25), (8, 10), (8, 26), (9, 11), (9, 27), (10, 12), (10, 29), (11, 13), (11, 34), (12, 20), (12, 35),\\
 (13, 21),  (13, 36), (14, 8), (14, 22), (15, 9), (15, 31), (16, 10), (16, 32), (17, 11), (17, 33), (18, 18), (18, 34),\\
 (19, 4),(19, 19), (20, 5), (20, 20), (21, 6), (21, 21), (22, 7), (22, 23), (23, 8), (23, 30), (24, 12), (24, 31),\\
(25, 16), (25, 32), (26, 17), (26, 33), (27, 2), (27, 18), (28, 3), (28, 24), (29, 4), (29, 28), (30, 14), (30, 29),\\
(31, 15), (31, 30), (32, 0), (32, 16), (33, 1), (33, 17), (34, 2), (34, 25), (35, 3), (35, 26), (36, 5), (36, 27).
\end{multline*}

According to~\cite[Theorem~1]{representability} any incompressible torus in the complement
of~$\widehat R$ is isotopic to a surface of the form~$\widehat\Pi$, where~$\Pi$ is a rectangular
diagram of a surface. Let such a diagram~$\Pi$ be chosen so that the number of rectangles
in~$\Pi$ is as minimal as possible (which is equivalent to requesting that~$\widehat\Pi$
has minimal possible number of intersections with~$\mathbb S^1_{\tau=0}\cup\mathbb S^1_{\tau=1}$).
We fix it from now on.

With any rectangle~$r=[\theta';\theta'']\times[\varphi';\varphi'']$
with~$\{\theta',\theta'',\varphi',\varphi''\}\cap\mathbb Z_{37}=\varnothing$
we associate a \emph{type} which is a $4$-tuple~$(i,j,k,l)\in(\mathbb Z_{37})^4$
defined by the following conditions:
$$(i;i+1)\ni\theta',\quad(j;j+1)\ni\varphi',\quad
(k;k+1)\ni\theta'',\quad(l;l+1)\ni\varphi''.$$
Since~$\partial\widehat\Pi=\varnothing$ we have~$\{\theta',\theta'',\varphi',\varphi''\}\cap\mathbb Z_{37}=\varnothing$,
so every rectangle in~$\Pi$ has a type.

Recall from~\cite{distinguishing} that by \emph{an occupied level} of~$\Pi$ we mean
any meridian~$m_{\theta_0}=\{\theta_0\}\times\mathbb S^1$ or any longitude~$\ell_{\varphi_0}=\mathbb S^1\times\{\varphi_0\}$
that contains a vertex of some rectangle in~$\Pi$.

\begin{lemm}
There are no rectangles in~$\Pi$ of type~$(i,j,k,l)$ with~$i=k$ or~$j=l$.
\end{lemm}

\begin{proof}
Let~$r=[\theta';\theta'']\times[\varphi';\varphi'']$ be a rectangle of~$\Pi$
such that the annulus~$(\theta';\theta'')\times\mathbb S^1$
contains no occupied level of~$\Pi$. Then the interval~$(\theta';\theta'')$
contains at least one point from~$\mathbb Z_{37}$ since otherwise
the number of intersections of~$\widehat\Pi$ with~$\mathbb S^1_{\tau=1}$
could be reduced by an isotopy.

This implies that for \emph{any} rectangle~$r=[\theta';\theta'']\times[\varphi';\varphi'']$
of~$\Pi$ the intersection~$(\theta';\theta'')\cap\mathbb Z_{37}$ is non-empty.
Indeed, if there is an occupied level of~$\Pi$ contained in~$(\theta';\theta'')\times\mathbb S^1$,
then there is a rectangle~$[\theta''';\theta'''']\times[\varphi''';\varphi'''']$ in~$\Pi$ with
$[\theta''';\theta'''']\subset(\theta';\theta'')$. By taking the narrowest such rectangle
we will have that~$(\theta''';\theta'''')\times\mathbb S^1$ contains
no occupied level of~$\Pi$, and hence $(\theta''';\theta'''')$ has a non-empty intersection with~$\mathbb Z_{37}$.
Similarly,~$(\varphi';\varphi'')\cap\mathbb Z_{37}\ne\varnothing$ for any rectangle of~$\Pi$.

Now let~$(i,j,k,l)$ be the type of some rectangle~$r=[\theta';\theta'']\times[\varphi';\varphi'']\in\Pi$.
The equality~$i=k$ would mean that~$(\theta';\theta'')\subset(i;i+1)$
or~$(\theta'';\theta')\subset(i;i+1)$. The former case is impossible as we
have just seen. In the latter case, we must have~$(\varphi';\varphi'')\subset(j;j+1)$
as otherwise~$r$ would contain a vertex of~$R$.
Therefore, this case also does not occur, and we have~$i\ne k$.

The inequality~$j\ne l$ is established similarly.
\end{proof}

The type~$(i,j,k,l)$ of a rectangle~$r$ is said to be \emph{admissible} if~$r\cap R=\varnothing$.
It is said to be \emph{maximal} if it is admissible, and none of the types~$(i-1,j,k,l)$, $(i,j-1,k,l)$,
$(i,j,k+1,l)$, and~$(i,j,k,l+1)$ is admissible.

\begin{lemm}\label{maximal-rect-lem}
The type of any rectangle in~$\Pi$ is maximal.
\end{lemm}

\begin{proof}
Here we will use the fact that the diagram~$R$ is \emph{rigid}, which means that it admits no non-trivial exchange move.
In other words, for any two neighboring edges~$\{(i,j_1),(i,j_2)\}$, $\{(i+1,j_3),(i+1,j_4)\}$ or~$\{(j_1,i),(j_2,i)\}$, $\{(j_3,i+1),(j_4,i+1)\}$
of~$R$, exactly one of~$j_3,j_4$ lies in~$(j_1;j_2)$, and the other lies in~$(j_2;j_1)$.

Let~$\{(i,j_1),(i,j_2)\}$, $\{(i+1,j_3),(i+1,j_4)\}$ be two neighboring vertical edges of~$R$, and let~$m_{\theta_0}$ with $\theta_0\in(i;i+1)$
be an occupied level of~$\Pi$. Since the surface~$\widehat\Pi$ is closed the whole meridian~$m_{\theta_0}$
is covered by the vertical sides of rectangles in~$\Pi$. Therefore, there are rectangles~$r_1,r_2,\ldots,r_{2p}\in\Pi$
of the form
$$r_{2k-1}=[\theta_{2k-1};\theta_0]\times[\varphi_{2k-1};\varphi_{2k}],\quad
r_{2k}=[\theta_0;\theta_{2k}]\times[\varphi_{2k};\varphi_{2k+1}],\quad k=1,\ldots,p,$$
where~$\varphi_{2p+1}=\varphi_1$.

We claim that each interval~$[\varphi_k;\varphi_{k+1}]$, $k=1,\ldots,2p$,
contains at most one of~$j_1,j_2,j_3,j_4$. Indeed, let~$k$ be odd.
Then~$r_k$ has the form~$[\theta_k;\theta_0]\times[\varphi_k;\varphi_{k+1}]$.
Since it is disjoint from~$R\supset\{(i,j_1),(i,j_2)\}$ we must have either~$[\varphi_k;\varphi_{k+1}]\subset(j_1;j_2)$
or~$[\varphi_k;\varphi_{k+1}]\subset(j_2;j_1)$. Due to rigidity of~$R$ each of the intervals~$(j_1;j_2)$
and~$(j_2;j_1)$ contains exactly one of~$j_3,j_4$, hence the claim.
In the case when~$k$ is even the proof is similar with the roles of~$\{j_1,j_2\}$ and~$\{j_3,j_4\}$ exchanged.

Thus, $p$ is at least~$2$. We now claim that~$p$ is exactly~$2$. Indeed, the number of tiles of~$\widehat\Pi$
attached to the vertex corresponding to~$m_{\theta_0}$ is equal to~$2p$, and we
have just seen that~$2p\geqslant4$. The same applies similarly to any other vertex of the tiling.
Since every tile is a $4$-gon and the surface~$\widehat\Pi$ is a torus,
every vertex of the tiling must be adjacent to \emph{exactly} four tiles.

The equality~$p=2$ implies that every interval~$(\varphi_k;\varphi_{k+1})$, $k=1,2,3,4$,
contains exactly one of~$j_1,j_2,j_3,j_4$, which means that
the rectangles
$$[\theta_1;\theta_0+1]\times[\varphi_1;\varphi_2],\
[\theta_0-1;\theta_2]\times[\varphi_2;\varphi_3],\ 
[\theta_3;\theta_0+1]\times[\varphi_3;\varphi_4],\
[\theta_0-1;\theta_4]\times[\varphi_4;\varphi_1]$$
are not of an admissible type.
In other words, whenever~$\Pi$ contains a rectangle
of type~$(i,j,k,l)$ (respectively, of type~$(k,l,i,j)$),
the type~$(i-1,j,k,l)$ (respectively, $(k,l,i+1,j)$)
is not admissible. Since~$i\in\mathbb Z_{37}$ was chosen
arbitrarily, we can put it another way:
whenever~$\Pi$ contains a rectangle of type~$(i,j,k,l)$,
the types~$(i-1,j,k,l)$ and~$(i,j,k+1,l)$ are not admissible.

Similar reasoning applied to a horizontal occupied
level~$\ell_{\varphi_0}$ of~$\Pi$ instead of~$m_{\theta_0}$
shows that whenever~$\Pi$ contains a rectangle of type~$(i,j,k,l)$
the types~$(i,j-1,k,l)$ and~$(i,j,k,l+1)$ are not admissible.
Therefore, every rectangle in~$\Pi$ is of a maximal type.
\end{proof}

A simple exhaustive search shows that there are exactly 623 maximal types
of rectangles for~$R$. For every maximal type~$(i,j,k,l)$ we denote by~$x_{i,j,k,l}$
the number of rectangles of type~$(i,j,k,l)$ in~$\Pi$. From the fact that
every vertex of a rectangle in~$\Pi$ is shared by exactly two rectangles,
which are disjoint otherwise, we get the following \emph{matching conditions}:
\begin{equation}\label{matching-eq}
\sum_{k,l\in\mathbb Z_{37}}x_{i,j,k,l}=
\sum_{k,l\in\mathbb Z_{37}}x_{k,l,i,j},\quad (i,j)\in(\mathbb Z_{37})^2,
\end{equation}
where we put~$x_{i,j,k,l}=0$ unless~$(i,j,k,l)$ is a maximal type.
For a complete list of maximal types and matching conditions the reader is referred to~\cite{anc}.

It is now a direct check that the system~\eqref{matching-eq} is of rank~$621$, and thus has
two-dimensional solution space. Is is another direct check that only one
solution in this space, up to positive scale, satisfies the non-negativity conditions~$x_{i,j,k,l}\geqslant0$.
Therefore, there exists at most one isotopy class of incompressible tori
in the complement of~$K_1$, which implies that every incompressible
torus is boundary-parallel.

\end{document}